\documentclass[11pt]{amsart}
\usepackage{geometry}
\usepackage{setspace}
\usepackage{enumerate}

\usepackage{amsmath,amssymb,graphicx,amscd}
\usepackage{multicol}
\usepackage{mathtools}
\usepackage{scalerel}
\usepackage{color}
\usepackage{enumitem}
\usepackage{xcolor}
\usepackage{float}
\usepackage[numbers]{natbib}
\usepackage[all]{xy}
\usepackage{amsthm}
\usepackage{tikz}
\usepackage{tikz-cd}
\usepackage{mathtools}
\usepackage{graphicx}
\usepackage{multicol}
\usepackage{fancyhdr}
\usepackage{hyperref}
\hypersetup{colorlinks=true,citecolor=blue,linkcolor=cyan!100,urlcolor=black,filecolor=black}
\numberwithin{figure}{section}
\numberwithin{equation}{section}
\allowdisplaybreaks
\numberwithin{equation}{section}

\newtheorem{theorem}{Theorem}[section]
\newtheorem{proposition}[theorem]{Proposition}
\newtheorem{corollary}[theorem]{Corollary}
\newtheorem{lemma}[theorem]{Lemma}
\newtheorem{conjecture}[theorem]{Conjecture}
\newtheorem{question}[theorem]{Question}

\theoremstyle{definition}
\newtheorem{remark}[theorem]{Remark}
\newtheorem{example}[theorem]{Example}
\newtheorem{definition}[theorem]{Definition}

\newtheorem{notation}[theorem]{Notation}

\DeclareMathOperator{\ext}{\scalebox{1.2}{$\mathsf{\Lambda}$}}

\author{Kazuo Habiro}
\address{Department of Mathematics, Kyoto University, Kyoto 606-8502, Japan}
\email{habiro@math.kyoto-u.ac.jp}

\author{Anderson Vera}
\address{Department of Mathematics, Kyoto University, Kyoto 606-8502, Japan}
\email{aaveraa@unal.edu.co}


\date{\today}
\title{Double Johnson filtrations for mapping class groups}

\newcommand\nc{\newcommand}
\nc\rnc{\renewcommand}
\nc\np{\newpage}
\nc\no[1]{}
\nc\ho{{\hat\otimes }}
\nc\zzzcolon {\colon\thinspace}
\nc\zzzsharp {\sharp}
\nc\plim{\varprojlim}
\nc\modone {{\mathbf 1}}
\nc\g{{\mathfrak g}}
\nc\modk {{\mathbf k}}
\nc\modL {{\mathcal L}}
\nc\modN {{\mathbb N}}
\nc\modQ {{\mathbb Q}}
\nc\modR {{\mathcal R}}
\nc\R{{\mathbb R}}
\nc\modZ {{\mathbb Z}}
\nc\Z{\modZ}
\nc\ul{\underline}
\nc\modF {\mathcal{F}}
\nc\la{\langle}
\nc\ra{\rangle}
\nc\lala{\la\negthinspace\la}
\nc\rara{\ra\negthinspace\ra}
\nc\Der{\opn{Der}}
\nc\btau{\bar{\tau }}
\nc\bG{\bar{G}}
\nc\sh{\opn{\sharp}}
\nc\fa{\ \text{for all }}
\nc\lex{{\mathrm{lex}}}
\nc\bK{\bar{K}}
\nc\bL{\bar{L}}
\nc\bX{\bar{X}}
\nc\bY{\bar{Y}}
\nc\opn{\operatorname}
\nc\hide[1]{}
\nc\gr{\mathrm{gr}}
\nc\bubu{{\bu,\bu}}
\nc\tot{\mathrm{tot}}
\nc\trl{\triangleleft}
\nc\trr{\triangleright}
\nc\simeqto{\overset{\simeq}{\rightarrow }}
\nc\Aut{\opn{Aut}}
\nc\xto[1]{{\overset{#1}{\longrightarrow}}}
\nc\yto[1]{{\underset{#1}{\longrightarrow}}}
\nc\xyto[2]{{\overset{#1}{\underset{#2}{\longrightarrow}}}}
\nc\Tot{\opn{Tot}}
\nc\bu{\bullet}
\nc\ct{\overset{\cong}{\to}}

\nc\A{\mathcal A}
\nc\M{\mathcal M}
\nc\G{\mathcal G}
\rnc\H{\mathcal H}
\nc\FPB{\mathrm{FPB}}
\nc\Q{\mathbb{Q}}

\begin{document}

\begin{abstract}
  We first develop a general theory of Johnson filtrations and Johnson homomorphisms for a group $G$ acting on another group $K$ equipped with a filtration indexed by a ``good'' ordered commutative monoid.  Then, specializing it to the case where the monoid is the additive monoid $\mathbb N^2$ of pairs on nonnegative integers, we obtain a theory of double Johnson filtrations and homomorphisms.  We apply this theory to the mapping class group $\M$ of a surface $\Sigma_{g,1}$ with one boundary component, equipped with the normal subgroups $\bX$, $\bY$ of $\pi_1(\Sigma_{g,1})$ associated to a standard Heegaard splitting of the $3$-sphere.  We also consider the case where the group $G$ is the automorphism group of a free group.
\end{abstract}

\maketitle

\tableofcontents

\section{Introduction}

One of the main objects associated to a compact, connected, oriented manifold is the group of isotopy classes of orientation-preserving self-homeomorphisms. A natural way of studying this group is by considering the canonical action on the fundamental group of the manifold. This paradigm is also present in the study of  the automorphism group of a given group: the former is studied by considering the canonical action on the latter. 

This general setting can be explained  concretely in the following two cases. 

Firstly, let $\Sigma$ be a compact, connected, oriented surface with  one boundary component. The mapping class group  $\mathcal{M}$  of $\Sigma$ acts on the fundamental group $\pi_1(\Sigma)$. In the eighties Johnson \cite{MR718141} started the study of this action by a \emph{step-by-step approximation}: he considered the lower central series $(\Gamma_i\pi_1(\Sigma))_{i\geq 1}$ of $\pi_1(\Sigma)$ and studied the action of $\mathcal{M}$ on the quotients $\pi_1(\Sigma)/\Gamma_i\pi_1(\Sigma)$. This allowed him to define a decreasing sequence
\begin{gather}\label{intro_equ_1}
\mathcal{M}\supset J_1\mathcal{M}\supset J_2\mathcal{M}\supset \cdots
\end{gather}
of normal subgroups of $\mathcal{M}$, which is now called the \emph{Johnson filtration of} $\mathcal{M}$; this approach was also pursued by Morita~\cite{MR1224104}. The subgroup $J_i\mathcal{M}$ consists of the mapping classes acting trivially on   $\pi_1(\Sigma)/\Gamma_{i+1}\pi_1(\Sigma)$. In particular, $\mathcal{I}:=J_1\mathcal{M}$ is the subgroup consisting of mapping classes acting trivially on the homology of the surface. This group, called the \emph{Torelli group} of $\Sigma$,  plays an important role in the algebraic structure of $\mathcal{M}$ as well as in the study of homology spheres, see \cite{MR1014464,MR1133875,MR1617632}. 

Secondly, let $F_n$ denote the free group on $n$ generators and $\mathcal{A}:=\mathrm{Aut}(F_n)$ the group of automorphisms of $F_n$. The \emph{Andreadakis-Johnson} filtration of $\mathcal{A}$ is the decreasing sequence 
\begin{gather}\label{intro_equ_2}
\mathcal{A}\supset \mathcal{IA}_1\supset \mathcal{IA}_2\supset \cdots
\end{gather}
of normal subgroups of $\mathcal{A}$. The subgroup $\mathcal{IA}_i$ consists of the automorphisms acting trivially on   $F_n/\Gamma_{i+1}F_n$. The subgroup $\mathcal{IA}:=\mathcal{IA}_1$, called \emph{$IA$-automorphism group of $F_n$},  consists of the automorphisms acting trivially on the abelianization of $F_n$.

Each term $J_n\mathcal{M}$ of the Johnson filtration  is equipped with a group homomorphism
$$\tau_n: J_n\mathcal{M}\longrightarrow \mathrm{Der}_n(\mathfrak{Lie}(H_1(\Sigma)))$$ taking values in the abelian group $\mathrm{Der}_n(\mathfrak{Lie}(H_1(\Sigma)))$ of degree $n$ derivations of the graded Lie algebra $\mathfrak{Lie}(H_1(\Sigma))$ freely generated by $H_1(\Sigma)$ in degree~$1$. The homomorphism $\tau_n$ is called the \emph{$n$-th Johnson homomorphism} and it satisfies $\mathrm{ker}(\tau_n)=J_{n+1}\mathcal{M}$. There are similar homomorphisms for the Andreadakis-Johnson filtration~\eqref{intro_equ_2}. These homomorphisms are useful tools for studying the algebraic structure of~$\mathcal{M}$ and~$\mathcal{A}$, see the surveys \cite{satoh, MR3380338} for further details and references.

The underlying abstract setting in the above parallel situations is a group $G$ acting on another group $K$ which is endowed with a \emph{filtration}   $(K_i)_{i\geq 1}$  (called \emph{N-series} in the literature) such that the action preserves the filtration. 

In \cite{HM} Massuyeau and the first  author  considered the case of a group $G$ acting on another group $K$ which is endowed with
an \emph{extended N-series}, which we will call \emph{$\mathbb N$-filtration} below, such that the action preserves the filtration. 

In this paper, we develop a more general theory of Johnson filtrations and homomorphisms for a group~$G$ acting on another group $K$
endowed with a filtration indexed by a more general ordered commutative monoid.
If the indexing monoid is the additive monoid $\mathbb{N}=\{0,1,2,\ldots\}$, we recover the notion of an extended N-series.
Besides, if we consider the additive monoid $\mathbb{N}^2$
of pairs of nonnegative integers we obtain a theory of double Johnson
filtrations and homomorphisms which can be applied to mapping class
groups and to automorphism groups of free groups.

\subsection{\texorpdfstring{$\Lambda$}{L}-filtered groups and \texorpdfstring{$\Lambda_+$}{L+}-graded Lie algebras}

By a \emph{good ordered commutative monoid} we mean a commutative monoid $(\Lambda,+,0)$ together with a partial order $\leq$ which is compatible with the addition and such that the zero element $0$ for the addition is the smallest element, see Section~\ref{sec_1} for the precise definition. For instance, the additive monoid $\mathbb{N}=\{0,1,2,\ldots\}$ with the usual order is a good   ordered commutative monoid. Similarly, for each integer $r\geq 0$ the direct product $\mathbb{N}^r=\mathbb{N}\times \cdots \times \mathbb{N}$ with component-wise addition and order is also an example of good ordered commutative monoid.  

Let $\Lambda$ be a good ordered commutative monoid and $\Lambda_+:=\Lambda\setminus \{0\}$.  A \emph{$\Lambda$-filtration} of a group $K$ is a family $K_* = (K_{\lambda})_{\lambda\in\Lambda}$ of normal subgroups of $K$ such that $K_0=K$, $K_{\lambda} \supset K_{\mu}$ for $\lambda\leq \mu$  and $[K_{\lambda}, K_{\mu}]\subset K_{\lambda + \mu}$ for $\lambda,\mu\in\Lambda$. A \emph{$\Lambda$-filtered group} $K_*$ is a group endowed with a $\Lambda$-filtration.

A \emph{$\Lambda_+$-graded Lie algebra} is a Lie algebra $L$ over
$\mathbb{Z}$ equipped with a direct sum decomposition {$L=\bigoplus_{\lambda \in \Lambda _+}L_\lambda$} such that the Lie bracket satisfies
$[L_\lambda ,L_\mu ]\subset L_{\lambda +\mu }$ for $\lambda ,\mu \in \Lambda _+$.

To each $\Lambda$-filtered group $K_*$ we associate a $\Lambda_+$-graded Lie algebra $\bar{K} := \bigoplus _{\lambda\in\Lambda_+} \bar{K}_{\lambda}$. Here we set
$\bar{K}_{\lambda} := K_{\lambda}/K_{> \lambda}$, where $K_{> \lambda}$ denotes the (normal) subgroup of $K$ generated by $\bigcup_{\mu >\lambda }K_\mu $.  The Lie bracket in $\bar{K}$ is induced by the commutator operation in $K$ (see Proposition~\ref{r24}).

\subsection{Johnson filtration and Johnson homomorphisms}

Let $K_*$ be a $\Lambda $-filtered group.  Let $G$ be a group acting on the
group $K_0$.  For each $\lambda \in \Lambda $, define  $G_\lambda \subset G$ by
\begin{gather*}
  G_\lambda =\{g\in G\mid g^{\pm 1}(k)k^{-1}\in K_{\lambda +\mu }\quad \text{for all $\mu \in\Lambda$ and $k\in K_{\mu}$}\}.
\end{gather*}

We show that    $G_* =(G_\lambda )_{\lambda \in \Lambda }$ is a $\Lambda $-filtration for $G_0$ \emph{acting} on $K_*$ (see Proposition~\ref{r14}).

Let $\bar{K}$ be the $\Lambda_+$-graded Lie algebra associated to $K_*$. For $\lambda\in\Lambda_+$ we denote by $\mathrm{Der}_{\lambda}(\bar{K})$ the abelian group of derivations of $\bar{K}$ of degree $\lambda$ (derivations $d:\bar{K}\rightarrow \bar{K}$ such that $d(\bar{K}_{\mu})\subset \bar{K}_{\lambda+\mu}$ for all $\mu\in\Lambda_+$). We show that for each $\lambda \in \Lambda _+$, there is a group homomorphism
  \begin{gather}\label{intro_equ_3}
    \tau _\lambda \zzzcolon G_\lambda \longrightarrow \mathrm{Der}_\lambda (\bK)
  \end{gather}
such that $G_{>\lambda}\subset\mathrm{ker}(\tau_{\lambda})$. We call it the \emph{$\lambda$-th Johnson homomorphism} (see Proposition~\ref{r39}). Let  $\bar{G}$ be the  $\Lambda_+$-graded Lie algebra associated to  $G_*$,  the family of Johnson homomorphisms \eqref{intro_equ_3} induces an injective morphism of $\Lambda _+$-graded Lie algebras
\begin{gather*}
  \bar{\tau}\zzzcolon \bG\longrightarrow \bigoplus_{\lambda\in\Lambda_+}\mathrm{Der}_{\lambda}(\bK),
\end{gather*}
which we call \emph{the Johnson homomorphism} of $G_*$ with respect to its action on $K_*$ (see Proposition~\ref{r23}).

\subsection{Double lower central series}

If we specialize the above general theory to the monoid $\mathbb{N}^2$ of pairs of nonnegative integers we obtain a theory of double filtrations and double Johnson homomorphisms. In order to apply this theory to  the mapping class group~$\mathcal{M}$ of the surface~$\Sigma$  or to the group of automorphisms~$\mathcal{A}$ of the free group~$F_n$, we need to endow~$\pi_1(\Sigma)$ or~$F_n$ with a $\modN^2$-filtration. This motivates the following notion.   Let~$\bar{X}$ and~$\bar{Y}$ be two normal subgroups of a group~$K$. The \emph{double lower central series} of the triple $(K; \bar{X}, \bar{Y})$ is the $\modN^2$-filtration $(K_{m,n})_{(m,n)\in\modN^2}$ of~$K$ defined inductively by
\begin{gather*}
    K_{0,0}= K,\quad K_{m,0}=\Gamma_m \bX \text{\ \  for } m\ge 1,\quad K_{0,n}=\Gamma_n\bY  \text{\ \ for } n\ge 1 \quad \text{ and }\\
    K_{m,n}=[K_{1,0},K_{m-1,n}]\;[K_{0,1},K_{m,n-1}] \quad \text{ for \ \ } m,n\ge 1.
\end{gather*}

 This is the fastest decreasing $\modN ^2$-filtration of $K$ satisfying $K_{1,0}=\bX$ and $K_{0,1}=\bY$. We are mainly interested in the case where $K=\langle x_1,\ldots,x_p,y_1,\ldots,y_q \rangle$ is a free group and $\bar{X}=\langle\langle x_1,\ldots,x_p\rangle\rangle_K$  (resp. $\bar{Y}=\langle\langle y_1,\ldots,y_q\rangle\rangle_K$) is the normal closure in $K$ of  $\{x_1,\ldots,x_p\}$ (resp. $\{y_1,\ldots, y_q\}$). In this case we can determine explicitly the associated $\modN^2_+$-graded Lie algebra $\bar{K}$ of the double lower central series of the triple $(K;\bar{X},\bar{Y})$. We obtain that it is isomorphic to the $\modN^2_+$-graded  Lie algebra $\mathfrak{Lie}(A,B)$ freely generated by $A$ in degree $(1,0)$ and $B$ in degree $(0,1)$, where $A$ (resp. $B$) is a free abelian group of rank $p$ (resp. $q$) (see Theorem~\ref{r51}).

\subsection{Double filtration for the mapping class group}
We endow  $\pi_1(\Sigma)$ with a $\modN^2$-filtration as follows. Consider the standard Heegaard splitting of the $3$-sphere, i.e., the decomposition $S^3=V\cup V'$ where $V$ and $V'$ are handlebodies of the same genus (which is the same as the genus of the surface $\Sigma$). Fix a disk $D\subset \partial V=\partial V'$. Hence, $\partial V=\partial V'=\Sigma\cup D$. We have two natural embeddings $\iota\zzzcolon\Sigma\rightarrow V$ and $\iota'\zzzcolon\Sigma\rightarrow V'$ which determine normal  subgroups of $\pi_1(\Sigma)$ defined as follows. 
\begin{gather*}
\bar{X}:=\mathrm{ker}\big(\pi_1(\Sigma)\xrightarrow{\  \iota_{\#}\  }\pi_1(V)\big) \quad \text{ and } \quad
\bar{Y}:=\mathrm{ker}\big(\pi_1(\Sigma)\xrightarrow{\  \iota'_{\#}\  }\pi_1(V')\big),
\end{gather*}
where $\iota_{\#}$ and $\iota'_{\#}$ denote the induced maps in homotopy.

The subgroup of $\mathcal{M}$ given by
\begin{gather*}
\mathcal{G}=\big\{h\in\mathcal{M}\ | \ h_{\#}(\bar{X})= \bar{X} \text{ and } h_{\#}(\bar{Y})= \bar{Y}\big\}
\end{gather*}
can be identified with the group of isotopy classes of orientation-preserving homeomorphism $h\zzzcolon {S}^3\rightarrow {S}^3$ such that $h(\partial V)=\partial V$ and $h(D)=D$ or equivalently  the  group of isotopy classes of orientation-preserving homeomorphism $h\zzzcolon {S}^3\rightarrow {S}^3$ preserving the standard Heegaard splitting $S^3=V\cup V'$ and $h(D)=D$. This group is called  the \emph{Goeritz group of~${S}^3$} \emph{(relative to the disk~$D$)}. It acts on $\pi_1(\Sigma)$  preserving the double lower central series  of the triple $(\pi_1(\Sigma); \bX, \bY)$. Therefore we obtain the double Johnson filtration $(\mathcal{G}_{m,n})_{(m,n)\in\modN^2}$ of the Goeritz group: 
 \begin{gather*}
\begin{matrix}
\G_{0,0}&\supset&\G_{0,1}&\supset&\G_{0,2}&\supset&\dots\\
\cup&&\cup&&\cup&&\\
\G_{1,0}&\supset&\G_{1,1}&\supset&\G_{1,2}&\supset&\dots\\
\cup&&\cup&&\cup&&\\
\G_{2,0}&\supset&\G_{2,1}&\supset&\G_{2,2}&\supset&\dots\\
\cup&&\cup&&\cup&&\\
\vdots&&\vdots&&\vdots&&
\end{matrix}
  \end{gather*}
We extend the above double filtration to obtain the double Johnson filtration $(\M_{m,n})_{m,n\geq -1}$ for the whole mapping class group:
  \begin{gather*}
\begin{matrix}
\M_{-1,-1}&\supset&\M_{-1,0}&\supset&\M_{-1,1}&\supset&\dots\\
\cup&&\cup&&\cup&&\\
\M_{0,-1}&\supset&\M_{0,0}&\supset&\M_{0,1}&\supset&\dots\\
\cup&&\cup&&\cup&&\\
\M_{1,-1}&\supset&\M_{1,0}&\supset&\M_{1,1}&\supset&\dots\\
\cup&&\cup&&\cup&&\\
\vdots&&\vdots&&\vdots&&
\end{matrix}
  \end{gather*}
which satisfies $\M_{m,n}=\G_{m,n}$ for $(m,n)\in\modN^2$ and $\M_{-1,-1}=\M$. The double Johnson filtration of the mapping class group is related with its usual Johnson filtration \eqref{intro_equ_1}  as follows. For $m,n\geq -1$ with $m+n\geq 1$ we have $\M_{m,n}\subset J_{m+n}\M$. 

We identify the leading terms of the double Johnson  filtration of $\M$. The groups $\M_{0,-1}$ and $\M_{-1,0}$ are the mapping class groups (relative to the disk $D$) of $V$ and $V'$ respectively, (see Proposition~\ref{prop_leading}). The groups $\M_{1,-1}$ and $\M_{-1,1}$ are the \emph{Twist groups} (relative to the disk $D$) of  $V$ and $V'$ respectively, (see Proposition~\ref{prop2_8}). 

We also obtain an interesting decomposition of the Torelli group:
\begin{gather*}
\mathcal{I} = \mathcal{M}_{2,-1}\cdot \mathcal{M}_{1,0}\cdot \mathcal{M}_{0,1}\cdot \mathcal{M}_{-1,2}\cdot \mathcal{K},
\end{gather*}
where $\mathcal{K}:=J_2\M$ is the so-called \emph{Johnson subgroup} of $\M$ (see Theorem~\ref{r46}).

Besides, the family of Johnson homomorphisms obtained for $(\G_{m,n})_{(m,n)\in\modN^2_+}$ can be extended to the double Johnson filtration of $\M$. In particular,  for $m,n\geq -1$ with $m+n\geq 1$ we have group homomorphisms
\begin{gather*}\tau_{m,n}:\mathcal{M}_{m,n}\longrightarrow D_{m,n}(A,B)
\end{gather*}
such that $$\mathcal{M}_{m+1,n}\cdot \mathcal{M}_{m,n+1}\subset \mathrm{ker}(\tau_{m,n}).$$
Moreover, these homomorphisms are compatible with the usual Johnson homomorphisms: there is a diagram
$$
\xymatrix{\mathcal{M}_{m,n}\ar[r]^{\subset}\ar[d]_{\tau_{m,n}} & J_{m+n}\mathcal{M} \ar[d]^{\tau_{m+n}} \\
						D_{m,n}(A,B)\ar[r]^{j} & D_{m+n}(H_1(\Sigma))}
$$
which is commutative (see Proposition~\ref{r55} and~\ref{r56}). Here $A$ and $B$ are subgroups of $H_1(\Sigma)$ such that $H_1(\Sigma)=A\oplus B$ and $D_{m,n}(A,B)$ and $D_{m+n}(H_1(\Sigma))$ are certain abelian groups.

As a corollary of the above results we obtain that for $h\in\mathcal{I}$  there exist
$$h_1\in\mathcal{M}_{2,-1}, \quad h_2\in\mathcal{M}_{1,0}, \quad h_3\in\mathcal{M}_{0,1}, \quad h_4\in\mathcal{M}_{-1,2}$$ such that
$$\tau_1(h)=\tau_{2,-1}(h_1) + \tau_{1,0}(h_2) + \tau_{0,1}(h_3) + \tau_{-1,2}(h_4).$$

\subsection{The case of the group of automorphisms of a free group}

We can develop a parallel theory of double filtrations for the automorphism group of a free group. Let $F=F_{p,q}=\langle x_1,\ldots,x_p,y_1,\ldots,y_q\rangle$ be the free group on $p+q$ generators. We  define a double filtration $(\mathcal{A}_{i,j})_{i,j\geq -1}$ for $\mathcal{A}=\mathrm{Aut}(F)$ which we call \emph{double Andreadakis-Johnson filtration}. This double filtration satisfies properties similar to those satisfied by the double Johnson filtration of the mapping class group.  

In particular in Theorem~\ref{thm_free} we show that if $\mathcal{IA}:=\mathrm{ker}\left(\mathcal{A}\rightarrow \mathrm{Aut}(F/[F,F])\right)$ is the \emph{IA-automorphism group} of $F$, then we have
$$\mathcal{IA} = \mathcal{A}_{2,-1}\cdot \mathcal{A}_{1,0}\cdot \mathcal{A}_{0,1}\cdot \mathcal{A}_{-1,2}.$$

\subsection{Organization of the paper}

We organize the rest of the paper as follows. In Section~\ref{sec_1} we develop the general theory of filtrations indexed by ordered commutative monoids.    In Section~\ref{sec_2} we specialize  the general theory to the monoid $\modN^2$, in particular we study the double lower central series of a free group. Section~\ref{sec_three} deals with the double filtration of the mapping class group and the group of automorphisms of a free group. Finally, Section \ref{sec_4} is devoted to the study of the double Johnson homomorphisms and their relation with the usual Johnson homomorphisms as well as with the alternative Johnson homomorphisms.

\bigskip
\medskip

\textbf{Acknowledgements.} 
The work of K.H. is supported by JSPS KAKENHI Grant Number 18H01119. 
The work of A.V. is supported by the JSPS Postdoctoral Fellowship for Research in Japan ID~No.~PE19728.
{The authors thank  Quentin Faes, Richard Hain, Mai Katada, Peter Lambert-Cole, and Gw\'ena\"el Massuyeau for helpful comments.}

\section{General filtrations for groups}\label{sec_1}

In this section, we define filtrations of groups indexed by {\em good
  ordered commutative monoids}.

\subsection{Preliminaries and Notations}\label{sec_1.1}

For any elements $a$ and $b$ in a group, set
\begin{gather*}
  [a,b]=aba^{-1}b^{-1}\quad \text{and} \quad {}^a b=aba^{-1}.
\end{gather*}
We will freely use the following identities.
\begin{multicols}{2}
\noindent
\begin{gather}
  \label{e4}
  [a,b]^{-1}=[b,a],\\
  \label{e15}
  [ab,c]={}^a[b,c]\cdot[a,c],
\end{gather}
\columnbreak
\begin{gather}
  \label{e16}
  [a,bc]=[a,b]\cdot{}^b[a,c],\\
  \label{e17}
  [[a,b],{}^bc]\cdot[[b,c],{}^ca]\cdot[[c,a],{}^ab]=1.
\end{gather}
\end{multicols}

For subgroups $S$ and $T$ of a group $G$, we denote by $[S,T]$ the subgroup of $G$ generated by commutators $[s,t]$ with $s\in S$ and $t\in T$. Let $G'$ be a normal subgroup of $G$ and $a,b\in G$, when we write \mbox{$a\equiv b\pmod{G'}$} we mean
that $ab^{-1}\in G'$. If the group $G$ acts on a group~$K$, we regard $K$ and $G$ as subgroups of the semidirect product $K\rtimes G$. In particular, if $g\in G\subset K\rtimes G$ and $k\in K\subset
K\rtimes G$, then we have $[g,k]=g(k)k^{-1}{\in K}$ {and ${}^g k=g(k)\in K$.}

Recall the three subgroups lemma.
\begin{lemma}\label{tsl}
Let $G$ be a group. Let $A$, $B$, $C$ and $N$ be subgroups of $G$ with $N$ a normal subgroup. If $[A,[B,C]]\leq N$ and $[B, [C,A]]\leq N$, then $[[A,B],C]\leq N$. 
\end{lemma}

For a group $G$, we denote by $(\Gamma_n G)_{n\geq 1}$ its lower central series, which is defined by $\Gamma_1 G = G$ and $\Gamma_{n+1}G =[\Gamma_n G, G]$ for $n\geq 1$.

\subsection{Good ordered commutative monoid}\label{sec_1.2}

By an \emph{ordered commutative monoid} $\Lambda $, we mean a commutative
monoid $(\Lambda ,+,0)$ with a partial order $\le $ on the set $\Lambda $ such that
\begin{itemize}
\item if $\lambda \le \lambda '$ then $\lambda +\mu \le \lambda '+\mu $.
\end{itemize}
We call $\Lambda $ {\em good} if the following conditions are satisfied:
\begin{itemize}
\item $0\le \lambda $ for all $\lambda \in \Lambda $,
\item If $\lambda <\lambda '$ then $\lambda +\mu <\lambda '+\mu $.
\end{itemize}

\begin{example}
  \label{r21}
  Here are some examples of good ordered commutative monoids.
  
  \begin{enumerate}
  \item $\{0\}$.
  \item $\modN =\{0,1,\ldots \}$ with the usual addition and order.
  \item (Direct product) For two good ordered commutative monoids $\Lambda $ and $\Lambda '$, the
    direct product monoid $\Lambda \times \Lambda '$ has a structure of good ordered
    commutative monoid by defining the order on $\Lambda \times \Lambda '$ by
    \begin{gather*}
      (\lambda ,\lambda ')\le (\mu ,\mu ')\quad \text{if and only if}\quad \lambda \le \mu \text{ and } \lambda '\le \mu '.
    \end{gather*}
    Similarly, we can define direct products $\Lambda _1\times \dots \times \Lambda _n$ of good
    ordered commutative monoids $\Lambda _1,\ldots ,\Lambda _n$.  Consequently, for each
    integer $r\ge 0$, $\modN ^r=\modN \times \dots \times \modN $ has a structure of a good ordered commutative
    monoid.
  \item Any submonoid $\Lambda '$ of a good ordered commutative monoid $\Lambda $
    inherits a good ordered commutative monoid structure from $\Lambda $.
    For example, $\{(a,b)\in \modN ^2\mid 0\le b\le a\}$ is a good ordered
    commutative monoid.
  \item (Lexicographic product)  For good ordered
    commutative monoids $\Lambda $ and $\Lambda '$, define the lexicographic order $\le _\lex$ on the direct product monoid
    $\Lambda \times \Lambda '$ by
    \begin{gather*}
      (\lambda ,\lambda ')\le _\lex (\mu ,\mu ')\quad \text{if and only if}\quad
      \text{either $\lambda <\mu$ or $\lambda =\mu$, $\lambda ' \le \mu '$}.
    \end{gather*}
    This gives $\Lambda \times \Lambda '$ a good order.  This is called the {\em
      lexicographic product} of~$\Lambda $ and~$\Lambda '$, and it is denoted by
    $\Lambda \times _\lex\Lambda '$.
  \item Let $\omega$ be the smallest infinite ordinal ($\omega=\modN$). The lexicographic product $\modN ^{2,\lex}:=\modN \times _\lex\modN $ is
    isomorphic, as a poset, to the ordinal
    \begin{gather*}
      \begin{split}
	\omega^2
	&=\{0,1,2,\ldots ,\omega,\omega+1,\ldots ,\omega+\omega=\omega2,\omega2+1,\ldots ,\omega3,\ldots \}\\
	&=\{\omega p+q\mid p,q\in \omega\}.
      \end{split}
    \end{gather*}
    Under this identification $\modN ^{2,\lex}\simeq\omega^2$, the addition
    is given by $(\omega p+q)+(\omega p'+q')=\omega(p+p')+(q+q')$.  {This operation is called the
    \emph{natural sum} or \emph{the Hessenberg sum}.  See also Section
    \ref{sec:transfinite-calculus}.  Similarly, $\modN ^{n,\lex}=\modN \times _\lex\dots \times _\lex\modN $ is identified with
    the ordinal $\omega^n$ for $n\in \modN $.}  
  \item Here are some examples of good orders on the additive monoid
    $\modN ^2$: $(a,b)\le (a',b')$ if and only if
    \begin{itemize}
    \item $a\le a'$ and $b\le b'$, (usual)
    \item either $a<a'$ or $a=a'$, $b\le b'$, (lexicographic)
    \item either $(a,b)=(a',b')$ or $a+b<a'+b'$.  (anti-diagonal or
      total)
    \item $a\le a'$ and $a+b\le a'+b'$
    \item either $a+b=a'+b'$, $a\le a'$ or $a+b<a'+b'$.
    \end{itemize}
  \item The above examples of good order commutative monoids are {\em
    well-founded} ones, i.e., those with well-founded partial order (any non-empty subset has a minimum element). However, there are not-well-founded examples such as $\modQ _{\ge 0}$ or
    $\mathbb R_{\ge 0}$. Notice that the ordered commutative monoid $(\mathbb{Z},+, 0)$ is not good.
  \end{enumerate}
\end{example}

\subsection{\texorpdfstring{$\Lambda$}{L}-filtered groups}\label{sec_1.3}

Let $\Lambda$ be a ordered commutative monoid and let $K$ be a group.  By a \emph{$\Lambda $-filtration} of $K$ we mean
a family $K_*=(K_\lambda )_{\lambda \in \Lambda }$ of normal subgroups $K_\lambda $ of $K$ such
that
\begin{itemize}
\item $K_0=K$,
\item if $\lambda \le \mu $, then $K_\lambda \supset K_\mu $,
\item $[K_\lambda ,K_\mu ]\subset K_{\lambda +\mu }$ for all $\lambda ,\mu \in \Lambda $.
\end{itemize}
We call a group with a $\Lambda $-filtration a {\em $\Lambda $-filtered group}.

If $\Lambda =\modN $, then the notion of $\modN $-filtration coincides with that of
\emph{extended N-series} in the sense of \cite{HM}.

For each $\lambda \in \Lambda$, let $K_{>\lambda }$ denote the subgroup of $K$ generated
by $\bigcup_{\mu >\lambda }K_\mu $.  Since each~$K_\mu $ is normal in~$K$, so is
$K_{>\lambda }$.  We also have $K_{>\lambda }\subset K_\lambda $.

We have the following.
\begin{lemma}
  \label{r22}\label{r44c}
  Let $\Lambda $ be a good ordered commutative monoid, and let $K_*$ be a
  $\Lambda $-filtered group.  For \mbox{$a,a'\in K_\lambda$}, $b,b'\in K_\mu $, $c\in K_\nu $,
  $\lambda ,\mu ,\nu \in \Lambda $, $\lambda ,\mu ,\nu >0$, we have
\begin{gather}
  \label{e23}[aa',b]\equiv[a,b]\cdot[a',b]\pmod{K_{>\lambda +\mu }},\\
  \label{e20}[a,bb']\equiv[a,b]\cdot[a,b']\pmod{K_{>\lambda +\mu }},\\
  \label{e21}[[a,b],c]\cdot[[b,c],a]\cdot[[c,a],b]\equiv1\pmod{K_{>\lambda +\mu +\nu }},\\
  \label{e45}  [a^{-1},b]\equiv [a,b^{-1}]\equiv [a,b]^{-1} \pmod{K_{>\lambda+\mu}}.
\end{gather}
\end{lemma}

\begin{proof}
\eqref{e23}, \eqref{e20} and \eqref{e45} follow easily. To show \eqref{e21} we use
 \begin{gather*}
    [[a,b],c]\equiv [[a,b],{}^bc] \pmod{K_{>\lambda +\mu + \nu }},\quad \quad
    [[b,c],a]\equiv [[b,c],{}^ca] \pmod{K_{>\lambda +\mu + \nu}},
  \end{gather*}
  \begin{gather*}
    [[c,a],b]\equiv [[c,a],{}^ab] \pmod{K_{>\lambda +\mu + \nu}}, 
  \end{gather*}
and \eqref{e17}.
\end{proof}

\begin{lemma}
  \label{r43}
  Let $\Lambda $ be a good ordered commutative monoid, and let $K_*$ be a
  $\Lambda $-filtered group.  For any $\lambda ,\mu \in  \Lambda $, we have
  \begin{gather*}
    [K_\lambda ,K_{>\mu }]\subset K_{>\lambda +\mu },\quad
    [K_{>\lambda },K_\mu ]\subset K_{>\lambda +\mu },\quad
    [K_{>\lambda },K_{>\mu }]\subset K_{>\lambda +\mu }.
  \end{gather*}
\end{lemma}

\begin{proof}
  We will prove the first inclusion, from which the others follow.

  For $a\in K_\lambda $ and $b\in K_{>\mu }$, we are going to prove
  $[a,b]\in K_{>\lambda +\mu }$.  By the definition of~$K_{>\mu }$, we can express
  $b$ as a product $b=b_1\dots b_p$, where $p\ge 1$ and $b_i\in K_{\mu _i}$,
  $\mu _i>\mu $ for $i=1,\ldots ,p$.  Then
  \begin{gather*}
    [a,b]=[a,b_1\dots b_p]=\prod_{i=1}^p{}^{b_1\dots b_{i-1}}[a,b_i].
  \end{gather*}
  We have $[a,b_i]\in [K_\lambda ,K_{\mu _i}]\subset K_{\lambda +\mu _i}\subset K_{>\lambda +\mu }$ since
  $\lambda +\mu _i>\lambda +\mu $.
  Since $K_{>\lambda +\mu }$ is normal in $K_{>\mu }$, it follows that
  $[a,b]\in K_{>\lambda +\mu }$.
\end{proof}

\subsection{\texorpdfstring{$\Lambda_+$}{L}-graded Lie algebras associated to \texorpdfstring{$\Lambda$}{L}-filtered groups}\label{sec_1.4}

For a good ordered commutative monoid $\Lambda $, let
\begin{gather*}
  \Lambda _+:=\{\lambda \in \Lambda \mid \lambda >0\}=\Lambda \setminus \{0\},
\end{gather*}
which is an \emph{ordered commutative semigroup} (a commutative semigroup equipped with a partial order $\le $ such that if $\lambda \le \lambda '$ then $\lambda +\mu \le \lambda '+\mu $).

By a {\em $\Lambda _+$-graded Lie algebra} we mean a Lie algebra $L$ over
$\modZ $ equipped with a direct sum decomposition
$L=\bigoplus_{\lambda \in \Lambda _+}L_\lambda $ such that the Lie bracket satisfies
$[L_\lambda ,L_\mu ]\subset L_{\lambda +\mu }$ for any $\lambda ,\mu \in \Lambda _+$.

A {\em morphism} $f\zzzcolon L\rightarrow L'$ between $\Lambda _+$-graded Lie algebras $L$ and
$L'$ is a Lie algebra homomorphism $f\zzzcolon L\rightarrow L'$ such that
$f(L_\lambda )\subset L_{\lambda}'$ for any $\lambda \in \Lambda _+$.

Let $K_*$ be a  $\Lambda $-filtered group.  We define a $\Lambda _+$-graded
Lie algebra $\bK$ as follows.  
For each $\lambda \in \Lambda$, set
\begin{gather}\label{e31}
  \bK_\lambda =K_\lambda /K_{>\lambda }.
\end{gather}
For $a\in K_\lambda $, the coset $aK_{>\lambda } \in \bK_\lambda $ is denoted by $[a]_\lambda $.

For $\lambda >0$, $\bK_\lambda $ is abelian.  Indeed,
since $\Lambda $ is good, we have for all $\lambda >0$
\begin{gather*}
  [K_\lambda ,K_\lambda ]\subset K_{\lambda +\lambda }\subset K_{>\lambda }.
\end{gather*}

For $\lambda ,\mu \in \Lambda $, define a bracket map
\begin{gather}
  \label{e18}
  [\ ,\ ]\zzzcolon \bK_\lambda \times \bK_\mu \rightarrow \bK_{\lambda +\mu }
\end{gather}
by
\begin{gather*}
  [[a]_\lambda , [b]_\mu ] = [[a,b]]_{\lambda +\mu }.
\end{gather*}
By Lemma \ref{r43}, this bracket is well-defined.

\begin{proposition}
  \label{r24}
  Let $\Lambda $ be a good ordered commutative monoid, and let $K_*$ be a
  $\Lambda $-filtered group.  Then the bracket maps \eqref{e18} for $\lambda ,\mu >0$
  form a $\Lambda _+$-graded Lie algebra structure on the graded abelian
  group $\bK=\bigoplus_{\lambda >0}\bK_\lambda $. 
\end{proposition}

{
\begin{proof}
We show that
$[\ ,\ ]$ is bilinear.  For $[a]_\lambda ,[a']_\lambda \in \bK_\lambda $ and $[b]_\mu \in \bK_\mu $,
we have
\begin{gather}
  \label{e3}
  \begin{split}
    [[a]_\lambda [a']_\lambda ,[b]_\mu ]=[[aa']_\lambda ,[b]_\mu ]=[[aa',b]]_{\lambda +\mu }
    =[[a,b]]_{\lambda +\mu } [[a',b]]_{\lambda +\mu }.
  \end{split}
\end{gather}
where we used \eqref{e23}.
Writing additively, we have
\begin{gather*}
  [[a]_\lambda +[a']_\lambda ,[b]_\mu ] = [[a]_\lambda ,[b]_\mu ] + [[a']_\lambda ,[b]_\mu ],
\end{gather*}
i.e., $[[a]_\lambda ,[b]_\mu ]$ is additive in $[a]_\lambda $.  Similarly, one can
check that it is additive in $[b]_\mu $ as well.

We now show that the bracket maps $[\ ,\ ]\zzzcolon \bK_\lambda \times \bK_\mu \rightarrow \bK_{\lambda +\mu }$
form a Lie bracket for $\bK$.  First, for $[a]_\lambda \in \bK_\lambda $, $\lambda >0$,
we have
\begin{gather*}
  [[a]_\lambda ,[a]_\lambda ]=[[a,a]]_\lambda =[1]_\lambda =0.
\end{gather*}
Secondly, for $[a]_\lambda \in \bK_\lambda $, $[b]_\mu \in \bK_\mu $, $\lambda ,\mu >0$, we have
\begin{gather*}
  [[a]_\lambda ,[b]_\mu ]+[[b]_\mu ,[a]_\lambda ]
  =[[a,b]]_{\lambda +\mu }+[[b,a]]_{\lambda +\mu }
  =[[a,b][b,a]]_{\lambda +\mu }
  =[1]_{\lambda +\mu }=0.
\end{gather*}
Thirdly, for $[a]_\lambda \in \bK_\lambda $, $[b]_\mu \in \bK_\mu $, $[c]_\nu \in \bK_\nu $,
$\lambda ,\mu ,\nu >0$, we have
\begin{gather*}
  \begin{split}
  &[[[a]_\lambda ,[b]_\mu ],[c]_\nu ]+
  [[[b]_\mu ,[c]_\nu ],[a]_\lambda ]+
  [[[c]_\nu ,[a]_\lambda ],[b]_\mu ]\\
  &\quad =
  [[[a,b]]_{\lambda +\mu },[c]_\nu ]]+
  [[[b,c]]_{\mu +\nu },[a]_\lambda ]]+
  [[[c,a]]_{\nu +\lambda },[b]_\mu ]]\\
  &\quad =
  [[[a,b],c]]_{\lambda +\mu +\nu }
  [[[b,c],a]]_{\lambda +\mu +\nu }
  [[[c,a],b]]_{\lambda +\mu +\nu }\\
  &\quad =
  [[[a,b],c][[b,c],a][[c,a],b]]_{\lambda +\mu +\nu }\\
  &\quad = 0,
  \end{split}
\end{gather*}
where we used \eqref{e21}.  Thus, $[\ ,\ ]$ gives a  Lie bracket for $\bK$.
\end{proof}
}

\subsection{Actions}\label{sec_actions}

Let $K_*$ be a $\Lambda $-filtered group and let $G$ be a group acting on
$K_0$.  We say that ``$G$ acts on~$K_*$'' if we have
\begin{gather*}
  g(K_\mu )=K_\mu \quad \text{for all $g\in G$ and $\mu \in \Lambda $}.
\end{gather*}

Let $K_*$ and $G_*$ be $\Lambda $-filtered groups.  We say that ``$G_*$ acts
on $K_*$'' if we have
\begin{gather}
  \label{e26}
  [G_\lambda ,K_\mu ]\subset K_{\lambda +\mu }
\end{gather}
{for all $\lambda ,\mu \in \Lambda $}, or, equivalently, if $G_0$ acts on $K_*$ and if
\eqref{e26} holds for all $\lambda \in \Lambda _+$ and $\mu \in \Lambda $. Notice that in \eqref{e26} we see $G_{\lambda}$ and $K_{\mu}$ as subgroups of $K\rtimes G$, from now on we will continue doing this.

\subsection{Johnson filtration}\label{sec_JF}

Let $K_*$ be a $\Lambda $-filtered group.  Let $G$ be a group acting on the
group $K_0$.  For each $\lambda \in \Lambda $, define a subset $G_\lambda $ of $G$ by
\begin{gather*}
  G_\lambda =\modF _\lambda ^{K_*}(G)=\{g\in G\mid [g^{\pm 1},K_\mu ]\subset K_{\lambda +\mu }\quad \text{for all $\mu \in\Lambda$}\}.
\end{gather*}
It is easy to check that $\lambda \le \lambda '$ implies $G_\lambda \supset G_{\lambda '}$, and that
we have
\begin{gather*}
  G_0=\{g\in G\mid g(K_\mu )=K_\mu \quad \text{for all $\mu \in\Lambda$}\},\\
  G_\lambda =G_0\cap \{g\in G\mid [g,K_\mu ]\subset K_{\lambda +\mu }\quad \text{for all
  $\mu \in\Lambda$}\}\quad \text{for $\lambda >0$}.
\end{gather*}
We have $G=G_0$ if and only if $G$ acts on $K_*$.

\begin{proposition}
  \label{r14}
  Let $\Lambda $ be a good ordered commutative monoid and let $K_*$ be a
   $\Lambda $-filtered group.  Let $G$ be a group acting on $K_0$.  Then
   $G_* =(G_\lambda )_{\lambda \in \Lambda }$ is a $\Lambda $-filtration for $G_0$ acting on $K_*$.
   Moreover, $G_*$ is the slowest decreasing $\Lambda $-filtration for $G_0$ acting on
   $K_*$, i.e., if $G'_*$ is another $\Lambda $-filtration of $G_0$ acting
   on $K_*$, then $G'_\lambda \le G_\lambda $ for all $\lambda $.
\end{proposition}

\begin{proof}
  The proof is almost identical to the proof of \cite[Proposition 3.1]{HM} except
  that we use~$\Lambda $ instead of~$\modN $.  Here we give a proof for
  completeness.

  One checks that $G_\lambda $ is a subgroup of $G_0$ as follows.  The case
  $\lambda =0$ is obvious.  Let $\lambda >0$ and let $g,g'\in G_\lambda $.  Then for any
  $k\in K_\mu $, $\mu \in  \Lambda $, we have
  \begin{gather*}
    \begin{split}
    [gg',k]={}^g[g',k]\cdot[g,k]\in {}^{G_\lambda }[G_\lambda ,K_\mu ]\cdot[G_\lambda ,K_\mu ]
    \subset {}^{G_\lambda }K_{\lambda +\mu }\cdot K_{\lambda +\mu }
    \end{split}
  \end{gather*}
  Since $G_\lambda \subset G_0$ we have
  ${}^{G_\lambda }K_{\lambda +\mu }\subset {}^{G_0}K_{\lambda +\mu }\subset K_{\lambda +\mu }$.  Hence
  $[gg',k]\in K_{\lambda +\mu }$.  Therefore, we have $gg'\in G_\lambda $.  Clearly,
  $g^{-1}\in G_\lambda $ for any $g\in G_\lambda $.  Thus, $G_\lambda $ is a subgroup of $G_0$.

  We show that $G_\lambda $ is a normal subgroup of $G_0$.  Let $g\in G_0$ and
  $h\in G_\lambda $.  Then for $k\in K_\mu $, $\mu \in \Lambda $, we have
  \begin{gather*}
    \begin{split}
      [{}^g h,k]={}^g[h,{}^{g^{-1}}k]\subset {}^{G_0}[G_\lambda ,{}^{G_0}K_\mu ]
	\subset {}^{G_0}[G_\lambda ,K_\mu ]
	\subset {}^{G_0}K_{\lambda +\mu }\subset K_{\lambda +\mu }.
    \end{split}
  \end{gather*}
  Thus $[{}^gh,K_\mu ]\subset K_{\lambda +\mu }$.  Hence ${}^gh\in G_\lambda $.  Thus, $G_\lambda $
  is normal in $G_0$.

  We show that $[G_\lambda ,G_\mu ]\subset G_{\lambda +\mu }$ for $\lambda ,\mu \in \Lambda $.
  For any $\nu \in \Lambda $, we have by the three subgroup lemma
  \begin{gather*}
    \begin{split}
      [[G_\lambda ,G_\mu ],K_\nu ]
      &\subset \lala[G_\lambda ,[G_\mu ,K_\nu ]]\;[G_\mu ,[G_\lambda ,K_\nu ]]\rara_{K_0\rtimes G_0}\\
      &\subset \lala[G_\lambda ,K_{\mu +\nu }]\;[G_\mu ,K_{\lambda +\nu }]\rara_{K_0\rtimes G_0}\\
      &\subset \lala K_{\lambda +\mu +\nu }\rara_{K_0\rtimes G_0}\\
      &= K_{\lambda +\mu +\nu }.
    \end{split}
  \end{gather*}
  Hence $[G_\lambda ,G_\mu ]\subset G_{\lambda +\mu }$.
\end{proof}

\begin{definition} The $\Lambda $-filtration $G_*$, in Proposition \ref{r14}, is called {\em the Johnson filtration} of~$G_0$ with respect to the action of~$G$ on~$K_0$.
\end{definition}

\subsection{Automorphism group of a \texorpdfstring{$\Lambda$}{L}-filtration}

Let $\Lambda$ be a good ordered commutative monoid, and let~$K_*$ be a $\Lambda$-filtered group.  The {\em automorphism group} of $K_*$ is
defined by
\begin{gather*}
  \Aut(K_*)=\{g\in \Aut(K_0)\mid g(K_{\lambda})=K_{\lambda}\quad \text{for all } \lambda\in\Lambda\}.
\end{gather*}
Let $\Aut_*(K_*)=(\Aut_{\lambda}(K_*))_{\lambda\in\Lambda}$ denote the Johnson filtration of
$\Aut(K_*)$ defined by its action on $K_*$:
\begin{gather*}
  \Aut_{\lambda}(K_*)=\{g\in \Aut(K_*)\mid [g,K_{\mu}]\subset K_{\lambda + \mu}\quad \text{for all } \lambda\in\Lambda\}.
\end{gather*}

\subsection{Derivation Lie algebras}\label{sec_dla}

Let $\Lambda $ be a good ordered commutative monoid.  Let $L$ be a
$\Lambda _+$-graded Lie algebra.  We define the {\em derivation
$\Lambda _+$-graded Lie algebra} $\Der(L)$ for $L$ as follows.

A derivation $d\zzzcolon L\rightarrow L$ of the Lie algebra $L$ (i.e., a $\modZ $-linear map
such that $d([x,y])=[d(x),y]+[x,d(y)]$ for all $x,y\in L$) is said to be
of {\em degree $\lambda \in\Lambda$} if we have $d(L_\mu )\subset L_{\lambda +\mu }$ for all $\mu \in \Lambda_+ $.
For such a derivation $d$ and $\mu \in \Lambda_+ $, let $d_\mu \zzzcolon L_\mu \rightarrow L_{\lambda +\mu }$ be
its restriction to $L_\mu $.  For $\lambda \in \Lambda _+$, let $\Der_\lambda (L)$ be the
abelian group of derivations of $L$ of degree $\lambda $, and set
\begin{gather*}
  \Der(L)=\bigoplus_{\lambda \in \Lambda _+}\Der_\lambda (L).
\end{gather*}
It is easy to see that $\Der(L)$ is a Lie subalgebra of the Lie
algebra of all derivations of $L$, and that $\Der(L)$ is $\Lambda _+$-graded
Lie algebra, i.e., we have
\begin{gather*}
  [\Der_\lambda (L),\Der_\mu (L)]\subset \Der_{\lambda +\mu }(L)
\end{gather*}
for all $\lambda ,\mu \in \Lambda_+$.

\subsection{The Johnson homomorphisms}\label{sec_1.9}

Let $\Lambda $ be a good ordered commutative monoid.  Let $K_*$ be a
 $\Lambda $-filtered group, and let $\bK$ be its associated $\Lambda_+$-graded
 Lie algebra.  Let $\Der(\bK)$ be the derivation $\Lambda_+$-graded Lie
 algebra of $\bK$.  Let $G_*$ be a $\Lambda $-filtered group acting on
 $K_*$. (Here $G_*$ is not necessarily the Johnson filtration of $G_0$.)

\begin{proposition}[Cf. {\cite[Proposition 6.2]{HM}}]
  \label{r39}
  For each $\lambda \in \Lambda _+$, there is a homomorphism
  \begin{gather*}
    \tau _\lambda \zzzcolon G_\lambda \longrightarrow \Der_\lambda (\bK)
  \end{gather*}
  which  maps each $g\in G_\lambda $ to
  \begin{gather*}
    \tau _\lambda (g)=\left(\tau _\lambda (g)_\mu \zzzcolon \bK_\mu \longrightarrow \bK_{\lambda +\mu }\right)_{\mu \in \Lambda _+}
  \end{gather*} defined by    $\tau _\lambda (g)_\mu ([a]_\mu )=[[g,a]]_{\lambda +\mu }$  for $a\in K_\mu $.
\end{proposition}

\begin{proof}
  Well-definedness of the map $\tau _\lambda (g)_\mu \zzzcolon \bK_\mu \rightarrow \bK_{\lambda +\mu }$ is
  easily verified since $G_*$ acts on~$K_*$.

  We show that $\tau _\lambda (g)_\mu $ is a homomorphism.  Clearly,
  $\tau _\lambda (g)_\mu ([1]_\mu )=[1]_{\lambda +\mu }$.  For $a,b\in K_\mu $, we have
  \begin{gather*}
    \begin{split}
      &\tau _\lambda (g)_\mu ([a]_\mu [b]_\mu )
      =\tau _\lambda (g)_\mu ([ab]_\mu )
       =[[g,{ab}]]_{\lambda +\mu }
       =[[g,a]\cdot {}^a[g,b]]_{\lambda +\mu }\\
       &=[[g,a]]_{\lambda +\mu }\cdot[{}^a[g,b]]_{\lambda +\mu }
       =[[g,a]]_{\lambda +\mu }\cdot[[g,b]]_{\lambda +\mu }
       =\tau _\lambda (g)_\mu ([a]_\mu )\cdot\tau _\lambda (g)_\mu ([b]_\mu ).
    \end{split}
  \end{gather*}
  Here we have used ${}^a[g,b]=[a,[g,b]\;][g,b]\equiv[g,b]$ mod
  $K_{\mu +\lambda +\mu }\subset K_{>\lambda +\mu }$.

  We show that $\tau _\lambda (g)$ is a derivation on $\bK$.
  For $[a]_\mu \in K_\mu $, $[a']_{\mu '}\in K_{\mu '}$, $\mu ,\mu '>0$, we have
  \begin{gather*}
    \begin{split}
      \tau _\lambda (g)_{\mu +\mu '}([[a]_\mu ,[a']_{\mu '}])
    &=\tau _\lambda (g)_{\mu +\mu '}([[a,a']]_{\mu +\mu '})\\
    &=[[g,[a,a']]]_{\lambda +\mu +\mu '}\\
    &=[[[g,a],a']\cdot[a,[g,a']]]_{\lambda +\mu +\mu '}\\
    &=[[[g,a],a']]_{\lambda +\mu +\mu '}\cdot[[a,[g,a']]]_{\lambda +\mu +\mu '}\\
    &=[[[g,a]]_{\lambda + \mu}, [a']_{{\mu'}}]\cdot [[a]_{\mu},[[g, a']]_{\lambda +{\mu'}}]\\
    &=[\tau _\lambda (g)([a]_\mu ),[a']_{\mu '}]\cdot[[a]_\mu ,\tau _\lambda (g)([a']_{\mu '})].
    \end{split}
  \end{gather*}
\end{proof}

Note that we have
\begin{gather}
  \label{e22}
  \ker(\tau _\lambda )= \{g\in G_\lambda \mid [g,K_\mu ]\subset K_{>\lambda +\mu }\text{ for all }\mu \in \Lambda _+\}.
\end{gather}

Let $\bG$ be the  $\Lambda_+$-graded Lie algebra associated to $G_*$. By Proposition \ref{r39}, the homomorphism $\tau _\lambda $ induces a
homomorphism
\begin{gather}
\label{e24}
  \btau_\lambda \zzzcolon \bG_\lambda \longrightarrow \Der_\lambda (\bK)
\end{gather}
for each $\lambda\in\Lambda_+$. Moreover, by \eqref{e22}, we have
\begin{gather*}
  \ker(\btau_\lambda )=\{g\in G_\lambda \mid [g,K_\mu ]\subset K_{>\lambda +\mu }\text{ for }\mu \in \Lambda _+\}/G_{>\lambda }.
\end{gather*}

\subsection{The Johnson morphism}\label{sec_1.10}

\begin{proposition}[Cf.  {\cite[Theorem 6.4]{HM}}]
  \label{r23}
  Let $\Lambda $ be a good ordered commutative monoid. Let $G_*$ and $K_*$ be 
  $\Lambda $-filtered groups such that  $G_*$ acts on $K_*$, and let $\bG$ and $\bK$ be their respective $\Lambda_+$-graded Lie algebras.  Then the family $\btau=(\btau_\lambda )_{\lambda \in \Lambda _+}$ of homomorphism
  $\btau_\lambda $ in~\eqref{e24} is a morphism of $\Lambda _+$-graded Lie
  algebras
\begin{gather}
  \label{e25}
  \btau\zzzcolon \bG\longrightarrow \Der(\bK).
\end{gather}
Moreover,  if $G_*$ is the Johnson
filtration $\modF _*^{K_*}(G)$ then $\btau$ is injective.
\end{proposition}

\begin{proof}
  The proof is essentially the same as a part of the proof of  \cite[Theorem 6.4]{HM}.   We have seen that $\btau_\lambda \zzzcolon \bG_\lambda \rightarrow \Der_\lambda (\bK)$ is a homomorphism
  for $\lambda >0$.  To show that $\btau$ is a morphism of $\Lambda _+$-graded Lie
  algebras, let $g\in G_\lambda $, $g'\in G_{\lambda '}$, $\lambda ,\lambda '\in \Lambda $, $a\in K_\mu $,
  $\mu \in \Lambda_{+}$.  One can prove
  \begin{gather*}
    \btau_{\lambda +\lambda '}([[g]_\lambda ,[g']_{\lambda '}])([a]_\mu )=
    [\btau_\lambda ([g]_\lambda ),\btau_{\lambda '}([g']_{\lambda '})]([a]_\mu )
  \end{gather*}
  in the same way as in the proof of \cite[Theorem 6.4]{HM}.  Here the
  integers $m,
  n, i$ should be replaced with elements of $\Lambda $ by setting $m=\lambda $,
  $n=\lambda '$, $i=\mu $, and the groups $K_{\nu +1}$ should be interpreted as
  $K_{>\nu }$.
\end{proof}

\begin{remark} \label{r42}
In \cite{HM} were introduced the notions of \emph{extended graded Lie algebra} (eg-Lie algebra for short), \emph{derivation eg-Lie algebra of an eg-Lie algebra} and to each $\modN$-filtered group was associated an eg-Lie algebra. Moreover, for  $\modN$-filtered groups $G_*$, $K_*$ with~$G_*$ acting on~$K_*$ it was proved that there is an \emph{eg-Lie algebra homomorphism} between the associated eg-Lie algebra of  $G_*$ and the derivation eg-Lie algebra of the eg-Lie algebra corresponding to $K_*$ \cite[Theorem 6.4]{HM}. We can generalize all those definitions in the context of $\Lambda$-filtered groups developed in the present paper. Thus,
we can  obtain a generalization of Proposition~\ref{r23}. All of these generalizations are straightforward by following \cite{HM}. Hence, we briefly sketch them in the next subsection.
\end{remark}

\subsection{Generalizations to extended \texorpdfstring{$\Lambda$}{L}-graded Lie
  algebras}

Let $\Lambda$ be a good commutative ordered monoid. An \emph{extended $\Lambda $-graded Lie algebra} $L=(L_0,L_+)$ consists of a group~$L_0$, a $\Lambda _+$-graded
Lie algebra $L_+=\bigoplus_{\lambda\in\Lambda_+} L_{\lambda}$ and a grading-preserving action of $L_0$ on $L_+$ denoted by $(g,x)\mapsto {}^gx$ for $g\in L_0$ and $x\in L_+$. 
  
A \emph{morphism} $f=(f_0,f_+)\zzzcolon L\rightarrow L'$ between extended $\Lambda $-graded Lie algebras $L$ and $L'$ consists of a group homomorphism $f_0:L_0\to L'_0$ and a $\Lambda_+$-graded Lie algebra morphism $f_+:L_+\rightarrow L'_+$ such that $f_+({}^gx)= {}^{f_0(g)}(f_+(x))$ for $g\in L_0$ and $x\in L_+$.

Let $K_*$ be a $\Lambda$-filtered group, then $\bK=(\bK_0, \bK_+=\bigoplus_{\lambda\in\Lambda_+}\bK_{\lambda})$, where $\bK_{\lambda}$ is defined in \eqref{e31}, is an extended $\Lambda $-graded Lie algebra. The action of $\bK_0$ on $\bK_+$ is induced by conjugation. 

Let $L=(L_0,L_+)$ be an extended $\Lambda $-graded  Lie algebra. Let $\lambda\in\Lambda_+$, a \emph{derivation} $d=(d_0,d_+)$ of \emph{degree~$\lambda$} of $L$  consists of a derivation $d_+:L_+\to L_+$ of degree~$\lambda$ and a $1$-cocycle $d_0:L_0\to L_{\lambda}$, i.e., a map satisfying
$$d_0(gh)= d_0(g) + {}^gd_0(h)$$
for $g,h\in L_0$, such that for all $\mu\in\Lambda_+$, $g\in L_0$ and $x\in L_{\mu}$ we have
$$d_{\lambda}({}^gx) - {}^gd_{\lambda}(x) = [d_0(g), {}^gx].$$ Denote
by $\mathrm{Der}_{\lambda}(L)$ the abelian group of derivations of $L$
of degree $\lambda$. Given two derivations
$d\in\mathrm{Der}_{\lambda}(L)$ and $d'\in \mathrm{Der}_{\mu}(L)$ of
$L$, it is straightforward to define the Lie bracket $[d,d']\in
\mathrm{Der}_{\lambda+\mu}(L)$.

The only
caveat is that the $1$-cocycle $[d,d']_0$ is defined by
$$[d,d']_0(g)= d(d'(g)) - d'(d(g))-[d_0(g),d'_0(g)]$$
for $g\in L_0$. 

Set $\Der _0(L) = \Aut(L)$ and $\Der _{+}(L)=\bigoplus_{\lambda \in \Lambda _+} \Der _{\lambda}(L)$. The proof in  \cite[Theorem 5.3]{HM} can be adapted straightforwardly to show that $\mathrm{Der}_{0}(L)$ acts on $\mathrm{Der}_{\lambda}(L)$ by conjugation and that  $\Der (L) = \left(\Der _{0}(L), \Der _{+}(L)\right)$ is an extended $\Lambda $-graded  Lie algebra.

 The following is the statement of \cite[Proposition 6.1]{HM} in our context. Let $G_*$ and $K_*$ be 
  $\Lambda $-filtered groups such that  $G_*$ acts on $K_*$, and let $\bG$ and $\bK$ their respective 
  extended $\Lambda$-graded
    Lie algebras.   There is a homomorphism 
\begin{equation*}\label{e32}
\tau_0\zzzcolon G_0\longrightarrow \mathrm{Der}_0(\bK)
\end{equation*}
  which maps $g\in G_0$ to the automorphism $\tau_0(g):\bK\to\bK$ defined by $\tau_0(g)(xK_{>\lambda}) = ({}^gx)K_{>\lambda}$ for $\lambda\in\Lambda_+$ and $x\in K_{\lambda}$. This map induces a homomorphism 
\begin{equation}\label{e33}
\btau_0\zzzcolon \bG_0\longrightarrow \mathrm{Der}_0(\bK).
\end{equation}
  
Finally, the generalization of Proposition \ref{r23} says that the family $\btau=(\btau_\lambda )_{\lambda \in \Lambda}$ of homomorphisms
  $\btau_\lambda$ defined in \eqref{e33} and \eqref{e24} gives a morphism of extended $\Lambda$-graded  Lie  algebras
\begin{gather}\label{e34}
  \btau\zzzcolon \bG\longrightarrow \Der(\bK).
\end{gather}

\subsection{Natural transfinite lower central series}
\label{sec:transfinite-calculus}

In this subsection we discuss a generalization of filtered groups in
the context of ordinal numbers. We follow \cite[Chapter~XIV]{MR0194339} for the preliminary definitions.  Let $\omega$ be
the smallest infinite ordinal. 
Given two ordinals~$\alpha$ and~$\beta$
we can find $k\in\modN$, a decreasing sequence of ordinals
$\epsilon_1,\ldots,\epsilon_k$ and $m_1,n_1,\ldots,m_k,n_k\in\modN$
such that
\begin{gather}\label{cantorform}
\alpha = \sum_{i=1}^k\omega^{\epsilon_i}m_i \quad \quad \text{and} \quad \quad \beta = \sum_{i=1}^k\omega^{\epsilon_i}n_i.
\end{gather}

Then the \emph{natural sum} or the \emph{Hessenberg sum} $\alpha\sharp\beta$ of $\alpha$ and $\beta$ is the ordinal given by
$$\alpha\sharp\beta = \sum_{i=1}^k\omega^{\epsilon_i}(a_i+b_i).$$ 
The natural sum $\sh$ is commutative, associative and unital with unit $0$:
\begin{gather*}
  \alpha \sh \beta =\beta \sh\alpha ,\quad
  (\alpha \sh\beta )\sh\gamma =\alpha \sh(\beta \sh\gamma ),\quad
  0\sh\alpha =\alpha \sh0=\alpha .
\end{gather*}
for any ordinals $\alpha ,\beta ,\gamma $.
Moreover, $\sh$ is (left and right) cancellative, i.e.,
\begin{gather*}
  \alpha \sh \beta =\alpha '\sh \beta  \quad \text{implies}\quad \alpha =\alpha '.
\end{gather*}

For any ordinal $\eta $, the ordinal $\omega ^\eta $ is closed under $\sh$. Hence $(\omega ^\eta ,\sh,0,\le )$ is a good ordered commutative monoid with the usual order of ordinals. If we ignore size issues, we can say that the class of ordinals is a (large) good ordered commutative monoid.

\begin{definition} Let $G$ be a group.  A {\em transfinite strongly central series} for
$G$ is a family $(G_\alpha )_\alpha $ of normal subgroups $G_\alpha $ of $G$ for
all ordinals $\alpha >0$ such that
\begin{itemize}
\item $G_1=G$,
\item if $\alpha \le \beta $ then $G_\alpha \supset G_\beta $,
\item for any ordinals $\alpha ,\beta >0$, we have
  $[G_\alpha ,G_\beta ]\subset G_{\alpha \sh \beta }$.
\end{itemize}
\end{definition}

\begin{definition}
The {\em natural transfinite lower central series} $(N_\alpha G)_{\alpha >0}$ is
defined as follows.
\begin{itemize}
\item if $\alpha =1$, then $N_1G=G$,
\item if $\alpha =\omega ^\eta $ for any ordinal $\eta >0$, then
  $N_{\omega ^\eta }G=\bigcap_{\beta <\omega ^\eta }N_\beta G$,
\item otherwise, we have
  \begin{gather*}
    N_\alpha G=\left\langle\  \bigcup_{\alpha '\sh\alpha ''=\alpha ,\;\alpha ',\alpha ''>0}[N_{\alpha '}G,N_{\alpha ''}G]\ \right\rangle ,
  \end{gather*}
  the subgroup of $G$ generated by all commutators of the form
  $[g',g'']$ where $g'\in N_{\alpha '}G$, $g''\in N_{\alpha ''}G$, $\alpha ',\alpha ''>0$ and $\alpha '\sh\alpha ''=\alpha $.
\end{itemize}
Then $(N_\alpha G)_{\alpha >0}$ is a transfinite strongly central series.
\end{definition}

The definitions of transfinite strongly central series and
  natural transfinite lower central series seem to be natural ones.
  One can apply the notions of Johnson filtrations and Johnson homomorphisms given in this section, which can not be applied to the usual notion of transfinite lower central series.

\section{Double filtrations}\label{sec_2}

Consider the commutative monoid $\modN ^2$ with its usual order, i.e., if $(m,n),(m',n')\in\modN ^2$,  we have \mbox{$(m,n)\leq (m',n')$} if and only if $m\leq m'$ and $n\leq n'$. This section focuses on the study of $\modN ^2$-filtered groups and its associated Johnson filtrations.

\subsection{Definition and examples}\label{sec_2.2}\label{sec_2.3}

By definition, an $\modN ^2$-filtered group consists of a group~$K$ and  normal subgroups $K_{m,n}\trl K$ for $(m,n)\in \modN ^2$ such that $K_{0,0}=K$ and  $K_{m,n}\supset K_{m',n'}$ for $(m,n)\le (m',n')$ and $$[K_{m,n},K_{m',n'}]\subset K_{m+m',n+n'}$$ for any $(m,n),(m',n')\in \modN ^2$.

Let $K_{*,*}=(K_{m,n})_{(m,n)\in \modN ^2}$ be an $\modN ^2$-filtered group. For any $(m,n)\in \modN ^2$ we have
\begin{gather}
  K_{>(m,n)}=K_{m+1,n}\cdot K_{m,n+1}.
\end{gather}

From Sections \ref{sec_actions} and \ref{sec_JF} we have the following. An $\modN ^2$-filtered group $G_{*,*}$ acts on $K_{*,*}$ if we have $[G_{m,n}, K_{i,j}]\subset K_{m+i,n+j}$ for all $(m,n),(i,j)\in\modN ^2$.

Let $G$ be a group acting on $K_{0,0}$. From Section \ref{sec_JF}, we have 
\begin{gather} 
\label{e27}
  G_{0,0}=\{g\in G\mid g(K_{i,j})=K_{i,j}\fa(i,j)\in \modN ^2\},
\end{gather}
and the Johnson filtration of $G_{0,0}$ is given by
\begin{gather}  
\label{e28}
  G_{m,n}=\{g\in G_{0,0}\mid [g,K_{i,j}]\subset K_{m+i,n+j}\fa(i,j)\in \modN ^2\}
\end{gather}
for all $(m,n)>(0,0)$.

We will call an $\modN^2$-filtration a \emph{double filtration}.

\begin{example}[Double lower central series]\label{exdlcs}
  Let $\bX$ and $\bY$ be two normal subgroups of a group $K$.
  Define $K_{m,n}=\Gamma _{m,n}(K;\bX,\bY)\trl K$ inductively by
  \begin{align*}
    K_{0,0}&
    =K,\\
    K_{m,0}&
    =\Gamma _m\bX\quad (m\ge 1),\\
    K_{0,n}&
    =\Gamma _n\bY\quad (n\ge 1),\\
    K_{m,n}&
    =[K_{1,0},K_{m-1,n}]\;[K_{0,1},K_{m,n-1}]\quad (m,n\ge 1).
\end{align*}
Then one can check easily that $\Gamma _{*,*}(K;\bX,\bY)=(\Gamma
_{m,n}(K;\bX,\bY))_{(m,n)\in\modN^2}$ is an $\modN ^2$-filtration of $K$, which we call the {\em double lower central series} of the triple
$(K;\bX,\bY)$.  This is the fastest decreasing $\modN ^2$-filtration
of $K$ satisfying $K_{1,0}=\bX$ and $K_{0,1}=\bY$. By definition, it follows that
$$ \Gamma_{m,n}(K;\bX,\bY)\subset \Gamma_{m+n}(\bX\bY){\subset \Gamma_{m+n}K}$$ for
all $(m,n)\in \modN^2$, where {$\Gamma_0 (\bX\bY):= \bX\bY$ and} $\Gamma_0 K:= K$. Some leading terms of the double lower central series are
\begin{gather*}
\Gamma _{1,1}(K;\bX,Y)=[\bX,\bY],\quad
\Gamma _{2,1}(K;\bX,\bY)=[\bX,[\bX,\bY]],\\
\Gamma _{1,2}(K;\bX,\bY)=[\bY,[\bX,\bY]],\quad
\Gamma_{2,2}(K;\bX,\bY)=[\bY,[\bX,[\bX,\bY]]]\;[\bX,[\bY,[\bX,\bY]]].
\end{gather*}

\medskip
  
In Sections~\ref{sec_dlcs} and~\ref{sec_dlcsfg}  we study double lower central series in detail, particularly when~$K$ is a free group.
\end{example}

\begin{example}
  We have another example of an $\modN ^2$-filtration from the same
  data $(K;\bX,\bY)$ as in Example~\ref{exdlcs} as follows. For $m,n\geq 0$ set $K'_{m,0}=\Gamma_m\bX$, $K'_{0,n}=\Gamma_n\bY$ and for $(m,n)\in \modN ^2$,
  \begin{gather*}
    K'_{m,n}=K'_{m,0}\cap K'_{0,n}.
  \end{gather*}
This is the slowest decreasing $\modN ^2$-filtration {of $K$} such that $K'_{m,0}=\Gamma_m\bX$ and
  $K'_{0,m}=\Gamma_m\bY$ for all $m\ge 1$.
\end{example}

\begin{example}[Double dimension series]
  \label{r20}
  As before, let $\bX$ and $\bY$ be two normal subgroups of a group~$K$.
  Consider the ideals $I_{\bX}:=\ker(\modZ [K]\xto{\epsilon _{\bX}}\modZ [K/\bX])$ and
  $I_{\bY}:=\ker(\modZ [K]\xto{\epsilon_{\bY}}\modZ [K/\bY])$, where $\epsilon _{\bX}$ and $\epsilon _{\bY}$ are the
  natural maps.  Define a double filtration
  $I_{*,*}=I_{*,*}(K;\bX,\bY)=(I_{m,n})_{(m,n)\in\modN ^2}$ of the algebra $\modZ [K]$ by
  \begin{gather*}
    I_{0,0}=\modZ [K],\quad I_{m,0}=I_{\bX}^m \text{\ \  for } m\ge 1,\quad I_{0,n}=I_{\bY}^n  \text{\ \ for } n\ge 1 \quad \text{ and }\\
    I_{m,n}(K)=I_{1,0}I_{m-1,n}+I_{0,1}I_{m,n-1}\quad \text{ for \ \ } (m,n)>(0,0).
  \end{gather*}
  Then we have
  \begin{gather*}
    I_{m,n}I_{m',n'}\subset I_{m+m',n+n'}\quad \text{for\ \ } (m,n),(m',n')\in \modN ^2.\\
\end{gather*}

For $(m,n)\in\modN ^2$, set
  \begin{gather*}
    D_{m,n}=D_{m,n}(K;\bX,\bY)=\{k\in K\mid k-1\in I_{m,n}\}.
  \end{gather*}
One can show that $D_{*,*}=(D_{m,n})_{(m,n)\in\modN^2}$ is an $\modN ^2$-filtration of $K$ and that  $\Gamma _{m,n}(K;\bX,\bY)\subset D_{m,n}$.  An interesting problem is the comparison of the two $\modN^2$-filtrations $\Gamma _{*,*}(K;\bX,\bY)$ and $D_{*,*}(K;\bX,\bY)$ of~$K$, this is related with the \emph{dimension subgroup problem}, see \cite{MR537126, MR798076}.
\end{example}

\subsection{Double lower central series}\label{sec_dlcs}

Let $K$ be a group and $\bX$ and $\bY$ two normal subgroups of it. We consider
the double lower central series
$K_{*,*}= (\Gamma_{m,n}(K;\bX,\bY))_{(m,n)\in\modN^2}$ of the triple $(K;\bar{X},\bar{Y})$.

By the following lemma, one can study the lower central series of a
group by using the double lower central series.

\begin{lemma}
  \label{r12} We have  $$\Gamma _m(\bX\bY)=\prod_{i+j=m}\Gamma _{i,j}(K;\bX,\bY)= \prod_{i+j=m}K_{i,j}$$
  for $m\ge 1$.
\end{lemma}

\begin{proof}
We only need to check the inclusion $\Gamma _m(\bX\bY)\subset\prod_{i+j=m}K_{i,j}$. We proceed by induction on $m$. The case $m=1$ is trivial, if $m\geq 2$ then
\begin{equation*}
\begin{split}
\Gamma_{m}(\bX\bY)&\subset[\Gamma_{m-1}(\bX\bY), K_{1,0}K_{0,1}] \\
&\subset \Big[\prod_{i+j=m-1}K_{i,j}, K_{1,0}K_{0,1}\Big]\\
 & \subset  \left(\prod_{i+j=m-1}[K_{i,j}, K_{1,0}]\right) \left(\prod_{i+j=m-1}[K_{i,j}, K_{0,1}]\right)\\
 &\subset \left(\prod_{i+j=m-1}K_{i+1,j}\right)\left(\prod_{i+j=m-1}K_{i,j+1}\right)\\
 & \subset \prod_{i+j=m}K_{i,j}.
\end{split}
\end{equation*}
\end{proof}

\begin{lemma}\label{r47} For $m,n\geq 1$ we have

\emph{(i)}  $[K_{0,1}, K_{m,0}]\subset [K_{1,0}, K_{m-1,1}]$.

\emph{(ii)} $[K_{1,0}, K_{0,n}]\subset [K_{0,1}, K_{1,n-1}]$.
\end{lemma}

\begin{proof}
The claim (i) is trivial for $m=1$. If $m\geq 2$, then by using identity~\eqref{e17} $m-1$ times
we have 
 \begin{gather*}
    \begin{split}
      [K_{0,1}, K_{m,0}]
      &=[K_{0,1}, [K_{m-1,0}, K_{1,0}]]\\
      &\subset [[K_{0,1},K_{m-1,0}], K_{1,0}]\;[[K_{0,1},K_{1,0}], K_{m-1,0}]\\
      &\subset [K_{m-1,1}, K_{1,0}]\;[K_{1,1}, K_{m-1,0}]\\
      & =  [K_{m-1,1}, K_{1,0}]\;[K_{1,1},[ K_{m-2,0}, K_{1,0}]]\\
      & \subset [K_{m-1,1}, K_{1,0}]\;[[K_{1,1}, K_{m-2,0}], K_{1,0}]\;[[K_{1,1}, K_{1,0}],  K_{m-2,0}]\\
      &\subset [K_{m-1,1}, K_{1,0}]\; [K_{m-1,1}, K_{1,0}]\; [K_{2,1},  K_{m-2,0}]\\
      & = [K_{m-1,1}, K_{1,0}]\; [K_{2,1},  K_{m-2,0}]\\
      &\;\;\vdots\\
      &\subset [K_{m-1,1}, K_{1,0}]\; [K_{m-2,1},  K_{2,0}]\\
      & = [K_{m-1,1}, K_{1,0}]\; [K_{m-2,1},  [K_{1,0}, K_{1,0}]]\\
      & \subset [K_{m-1,1}, K_{1,0}]\;[[K_{m-2,1},  K_{1,0}], K_{1,0}]\;[[K_{m-2,1},  K_{1,0}], K_{1,0}]\\
      & \subset [K_{m-1,1}, K_{1,0}]\;[K_{m-1,1},  , K_{1,0}]\;[K_{m-1,1}, K_{1,0}]\\
      & = [K_{m-1,1}, K_{1,0}].
    \end{split}
  \end{gather*}

 (ii) can be proved similarly to (i).
\end{proof}

\begin{proposition} For $m,n\geq 1$ we have

\emph{(i)} $K_{m,1} = [K_{1,0}, K_{m-1,1}] = [K_{1,0}, [K_{1,0},[\ldots [K_{1,0},[K_{1,0},K_{0,1}]]\ldots]]$. Here $K_{1,0}$ appears~$m$ times.

\emph{(ii)} $K_{1,n} = [K_{0,1}, K_{1,n-1}]=[K_{0,1}, [K_{0,1},[\ldots [K_{0,1},[K_{0,1},K_{1,0}]]\ldots]]$. Here $K_{0,1}$ appears~$n$ times.
\end{proposition}
\begin{proof} This follows {from the}
definition and Lemma~\ref{r47}.
\end{proof}

\subsection{Double lower central series for free groups}\label{sec_dlcsfg}

Several results valid for the lower central series of a free group of finite rank can be developed for the double lower central series, see~\cite[Section 5]{MR2109550}. In this section we provide some of such developments. 
Let $p\ge 1$ and $q\ge 1$. Consider disjoint sets $X=\{x_1,\ldots ,x_p\} $  and $Y=\{ y_1,\ldots ,y_q\}$. Throughout this  subsection, let $K$ be the free group on $X\cup Y$, that is, $K=\langle x_1,\ldots ,x_p,y_1,\ldots ,y_q\rangle$. Let $\bar{X}$ (resp.~$\bar{Y}$) be the normal closure in $K$ of $X$ (resp.~$Y$). We consider the double lower central series
$K_{*,*}= (K_{m,n})_{(m,n)\in\modN^2}$ of the triple $(K;\bar{X},\bar{Y})$.

We define the multibracket $[u_1,...,u_r]$ of $r$ elements $u_1,\ldots, u_r$ in  $K$ as the \emph{right-bracketed} element
$$
[u_1,[u_2,[..,[u_{{r-1}},u_{{r}}]...]\in K
$$
(for $r=1$, we set $[u_1]=u_1$). We say that $[u_1,...,u_{{r}}]$ is an \emph{$(m,n)$-commutator} (in  $X\cup Y$) if  $u_1,\dots,u_r\in X\cup Y$ and 
$\#\{i\mid u_i\in X\}=m$ and $\#\{i\mid u_i\in Y\}=n$. In other words, an $(m,n)$-commutator (in $X\cup Y$) is a multibracket involving exactly~$m$ elements of~$X$ and~$n$ elements of~$Y$ (counted with repetition). By definition,  some $(m,n)$-commutators could be trivial.

\begin{example}\label{ex_comm} Let $K =\la x_1,x_2,y_1,y_2\ra$. Then

$\bullet$ The $(1,0)$-commutators are $x_1,x_2$.

$\bullet$ The $(0,1)$-commutators are $y_1,y_2$.

$\bullet$ The non-trivial $(2,0)$-commutators are $[x_1,x_2]$, $[x_2,x_1]$.

$\bullet$ The non-trivial $(0,2)$-commutators are $[y_1,y_2]$, $[y_2,y_1]$.

$\bullet$ The non-trivial $(1,1)$-commutators are $[x_1, y_1]$, $[x_1, y_2]$, $[x_2, y_1]$, $[x_2, y_2]$, $[y_1, x_1]$, $[y_1, x_2]$, $[y_2, x_1]$, $[y_2, x_2]$.

$\bullet$ {Some} non-trivial $(2,1)$-commutators are  $[x_2, [x_1, y_2]]$, $[x_2, [x_2, y_1]]$ and $[x_2, [x_2, y_2]]$.
\end{example}

\begin{lemma}\label{r49}  Let $A,B\leq H_1(K;\mathbb{Z})$ be the free abelian groups generated by the families $\{x_i\Gamma_{2}K\}_{1\leq i\leq p}$ and $\{y_j\Gamma_{2}K\}_{1\leq j\leq q}$ respectively.  Then $$\bK_{1,0}\simeq A \quad \text{ and } \quad \bK_{0,1}\simeq B.$$
\end{lemma}
\begin{proof}
By Hopf's formula we have $(\Gamma_2 K\cap\bar{X})/[K,\bar{X}]\simeq H_2(K/\bar{X})$ and $H_2(K/\bar{X})=0$ because $K/\bar{X}$ is a free group. Hence
$$K_{>(1,0)} = K_{2,0}\cdot K_{1,1}= [\bar{X},\bar{X}]\;[\bar{Y},\bar{X}] = [K,\bar{X}]=\Gamma_2K\cap\bar{X}.$$
Therefore
$$\bK_{1,0} = \frac{K_{1,0}}{K_{>(1,0)}} = \frac{\bar{X}}{\Gamma_2K\cap\bar{X}} \mathrel{\mathop{\longrightarrow}^{\mathrm{\simeq}}_{\mathrm{ab}}} A,$$
where $\mathrm{ab}\zzzcolon K\rightarrow H_1(K;\modZ)=K/\Gamma_2K$ is the canonical map. Similarly $\bK_{0,1}\simeq B$.
\end{proof}

\begin{proposition}\label{r48} For $(m,n)\in\modN^2_+$, the abelian group $$\bK_{m,n}=\frac{K_{m,n}}{K_{>(m,n)}}=\frac{K_{m,n}}{K_{m+1,n}\cdot K_{m,n+1}}$$
is generated by the cosets of $(m,n)$-commutators.
\end{proposition}
\begin{proof}
The proof is  by induction on $(m,n)\in\modN^2_+$. The cases $(1,0)$ and $(0,1)$ follows {from} Lemma~\ref{r49}. 

Consider the case $m\geq 2$, $n=0$. Then $K_{m,0}=[K_{1,0}, K_{m-1,0}]$ is generated by elements of the form $[z,w]$ with $z\in K_{1,0}$ and $w\in K_{m-1,0}$. It follows that $\bK_{m,0}$ is generated by the cosets of these elements.
By inductive hypothesis 
$$w=\left(\prod_{i=1}^r w_i^{\epsilon_i}\right)w'$$
where $w_i$ are $(m-1,0)$-commutators, $w'\in K_{>(m-1,0)}= K_{m,0}\cdot K_{m-1,1}$ and $\epsilon_i=\pm1$. Write $w'=w'_1w'_2$ with $w'_1\in K_{m,0}$ and $w'_2\in K_{m-1,1}$. Hence, by Lemma~\ref{r44c} we obtain

\begin{gather}\label{e41}
\begin{split}
[z,w] & = \Big[z, \left(\prod_{i=1}^r w_i^{\epsilon_i}\right)w'\Big]\\
	& \equiv \left(\prod_{i=1}^r [z,w_i]^{\epsilon_i}\right) [z,w_1']\;[z,w_2'] \pmod{K_{>(m,0)}}\\
	& \equiv \prod_{i=1}^r [z,w_i]^{\epsilon_i}  \pmod{K_{>(m,0)}}.
\end{split}
\end{gather}

Since $z\in K_{1,0}$, we can write it as
\begin{gather*}
z=\prod_{j=1}^s {}^{u_j}(z_j^{\epsilon_j})= \prod_{j=1}^s ({}^{u_j}z_j)^{\epsilon_j}
\end{gather*}
with $u_j\in \la Y\ra
\subset K_{0,1}$, $z_j\in X\subset K_{1,0}$ and $\epsilon_j=\pm1$.  Hence applying again Lemma~\ref{r44c} for each $i=1,\ldots, r$ we get
\begin{gather*}
\begin{split}
[z,w_i] & = \Big[ \prod_{j=1}^s ({}^{u_j}z_j)^{\epsilon_j}, w_i\Big]\\
	& \equiv \prod_{j=1}^s [({}^{u_j}z_j)^{\epsilon_j},w_i] \pmod{K_{>(m,0)}}\\
	& \equiv \prod_{j=1}^s [z_j^{\epsilon_j}[z_j^{-1},u_j]^{\epsilon_j},w_i]  \pmod{K_{>(m,0)}}\\
	& \equiv \prod_{j=1}^s [z_j,w_i]^{\epsilon_j} \prod_{j=1}^s [[z_j^{-1},u_j]^{\epsilon_j},w_i]   \pmod{K_{>(m,0)}}\\
	& \equiv \prod_{j=1}^s [z_j,w_i]^{\epsilon_j} \pmod{K_{>(m,0)}}.
\end{split}
\end{gather*}

Replacing in \eqref{e41} we obtain

\begin{gather*}
\begin{split}
[z,w] &\equiv \prod_{i=1}^r\prod_{j=1}^s[z_j,w_i]^{\epsilon_j\epsilon_i} \pmod{K_{>(m,0)}}
\end{split}
\end{gather*}
and each $[z_j,w_i]$ is a $(m,0)$-commutator.

The case $m=0$, $n\geq 2$ is proved similarly. Finally suppose $m,n\geq 1$, then
$$K_{m,n}=[K_{1,0}, K_{m-1,n}]\;[K_{0,1}, K_{m,n-1}]$$
is generated by elements of the form $[z,w]$ and  $[z',w']$ with $z\in K_{1,0}$, $z'\in K_{0,1}$, $w\in K_{m-1,n}$ and $w'\in K_{m,n-1}$. It follows that $\bK_{m,n}$ is generated by the cosets modulo $K_{>(m,n)}$ of these elements. Using the inductive hypothesis, similarly to the case $m\geq 2$, $n=0$, one can 
show that the cosets of the above elements can be written as  products of cosets of $(m,n)$-commutators.
\end{proof}

\begin{remark} The cosets in $\bK_{m,n}$ of $(m,n)$-commutators (for $m+n>1$) are not linearly independent. For instance, in Example~\ref{ex_comm} we have
$$[x_1,x_2]K_{>(2,0)} = -[x_2,x_1]K_{>(2,0)}.$$
\end{remark}

Let us now describe explicitly the $\modN^2_+$-graded Lie algebra associated to the double lower central series.

\medskip

Consider disjoint sets $\mathsf{A}=\{a_1,\ldots ,a_p\} $  and $\mathsf{B}=\{b_1,\ldots ,b_q\}$. Let $\mathfrak{Lie}(\mathsf{A},\mathsf{B})$ be the free Lie algebra on~$\mathsf{A}\cup \mathsf{B}$. We endow it with a $\modN^2_+$-grading by declaring the elements of the free abelian group generated by~$\mathsf{A}$ (resp.~$\mathsf{B}$) as being of degree $(1,0)$ (resp.~$(0,1)$). Then we extend this degree recursively using the Lie bracket. More precisely, if $u,v\in\mathfrak{Lie}(\mathsf{A},\mathsf{B})$ are, respectively, of degrees $(i,j)\in\modN^2_+$ and $(m,n)\in\modN^2_+$, then $[u,v]$ has degree $(i+m,j+n)\in\modN^2_+$.  We denote by $\mathfrak{Lie}_{m,n}(\mathsf{A}, \mathsf{B})$ the abelian group generated by elements of degree $(m,n)$. Thus
$$\mathfrak{Lie}(\mathsf{A}, \mathsf{B})=\bigoplus_{(m,n)\in\modN^2_+}\mathfrak{Lie}_{m,n}(\mathsf{A}, \mathsf{B}).$$

\begin{remark}
Let $A$ and $B$ be the free abelian groups generated by $\mathsf{A}$ and $\mathsf{B}$ respectively. We also denote  the $\modN^2_+$-graded Lie algebra $\mathfrak{Lie}(\mathsf{A}, \mathsf{B})$ by $\mathfrak{Lie}(A, B)$.
\end{remark}

For $(m,n)\in\modN^2_+$, we define a \emph{Lie $(m,n)$-commutator} as the Lie multibracket 
$$[u_1,...,u_{m+n}]=[u_1,[u_2,[..,[u_{m+n-1},u_{m+n}]...]\in \mathfrak{Lie}_{m,n}(\mathsf{A},\mathsf{B})$$
(for $m+n=1$, we set $[u_1]=u_1$) such that $u_1,\dots,u_{m+n}\in \mathsf{A}\cup \mathsf{B}$ and 
$\#\{i\mid u_i\in \mathsf{A}\}=m$ and $\#\{i\mid u_i\in \mathsf{B}\}=n$.

\begin{proposition}\label{rr49} For $(m,n)\in\modN^2_+$, the abelian group $\mathfrak{Lie}_{m,n}(\mathsf{A},\mathsf{B})$ is generated by Lie $(m,n)$-com\-mu\-ta\-tors.
\end{proposition}

\begin{proof}
The proof is  by induction on $(m,n)\in\modN^2_+$. The cases $(1,0)$ and $(0,1)$ are trivial.
The other cases follow by mirroring the proof of Proposition~\ref{r48}.
\end{proof}

\begin{theorem}[Cf.  {\cite[Theorem 5.12]{MR2109550}}]\label{r51} Consider the map
\begin{gather}
\varphi:\mathfrak{Lie}(\mathsf{A},\mathsf{B})\longrightarrow \bK
\end{gather}
induced by mapping $a_i\mapsto x_iK_{{>(1,0)}}$ and $b_j\mapsto y_jK_{{>(0,1)}}$ for $1\leq i\leq p$ and  $1\leq j\leq q$. Then $\varphi$ is an isomorphism of $\modN^2_+$-graded Lie algebras. In particular, $$\mathfrak{Lie}_{m,n}(\mathsf{A}, \mathsf{B})\simeq \bK_{m,n}$$
for all $(m,n)\in\modN^2_+$.
\end{theorem}
Before proving this theorem, we state the following.

\begin{corollary}
For $m\geq 1$ we have
$$\frac{\Gamma_m K}{\Gamma_{m+1}K}\simeq \bigoplus_{i+j=m}\bK_{i,j}.$$
\end{corollary}
\begin{proof}
Denote by $\mathfrak{Lie}(\mathsf{A}\cup\mathsf{B})=\bigoplus_{m\geq 1}\mathfrak{Lie}_m(\mathsf{A}\cup\mathsf{B})$ the graded Lie algebra freely generated by $\mathsf{A}\cup\mathsf{B}$ in degree~$1$. We use the classical isomorphism
$$\frac{\Gamma_m K}{\Gamma_{m+1}K}\simeq \mathfrak{Lie}_m(\mathsf{A}\cup\mathsf{B})$$
induced by mapping $a_i\mapsto x_i\Gamma_2K$ and $b_j\mapsto y_j\Gamma_2K$ for $1\leq i\leq p$ and  $1\leq j\leq q$. 
Clearly, we have $\mathfrak{Lie}_m(\mathsf{A}\cup\mathsf{B})\simeq\bigoplus_{i+j=m} \mathfrak{Lie}_{i,j}(\mathsf{A},\mathsf{B})$. These together with Theorem~\ref{r51} yield
$$\frac{\Gamma_m K}{\Gamma_{m+1}K}\simeq \mathfrak{Lie}_m(\mathsf{A}\cup \mathsf{B})\simeq \bigoplus_{i+j=m}\mathfrak{Lie}_{i,j}(\mathsf{A},\mathsf{B})\simeq \bigoplus_{i+j=m}\bK_{i,j}.$$
\end{proof}

In order to prove Theorem~\ref{r51} we need to introduce some notions. Let $\modZ\lala X_1,\ldots,X_p,Y_1,\ldots, Y_q\rara$ (resp. $\modZ\la X_1,\ldots,X_p,Y_1,\ldots, Y_q\ra$) be the algebra of formal power series (resp. the ring of polynomials) in the $p+q$ noncommuting variables $X_1,\ldots,X_p$, $Y_1,\ldots, Y_q$.

A monomial $u$ in the noncommuting variables $X_1,\ldots,X_p$, $Y_1,\ldots, Y_q$  is said to have \emph{degree} $(m,n)$, denoted $\mathrm{deg}(u)= (m,n)$, if it involves $m$ (resp. $n$) occurrences of the variables in $\{X_1,\ldots, X_p\}$ (resp.  $\{Y_1,\ldots, Y_q\}$) counted with repetition. Notice that this degree is compatible with the multiplication, that is, if $u$ and $v$ are two of these monomials of degree $(m,n)$ and $(i,j)$ respectively, then $uv$ has degree $(m+i,n+j)$.

The Magnus expansion is the injective multiplicative map
$$\theta\zzzcolon K\longrightarrow \modZ\lala X_1,\ldots,X_p,Y_1,\ldots, Y_q\rara$$
defined by $x_i\mapsto 1+ X_i$ ($1\leq i\leq p$) and $y_j\mapsto 1+ Y_j$ ($1\leq j\leq q$). In particular,  we have {$\theta(w) = 1+ (\text{terms of degree } >  (0,0))$} for all $w\in K$. {Here we are considering $\modN^2$ with its usual order.}

\begin{lemma}\label{r52} If $z\in K_{m,n}$ then
\begin{gather}\label{e42}
\theta(z)= 1+ z_{{m,n}} + (\mathrm{deg} >  (m,n))
\end{gather}
where $z_{m,n}$ is 
a linear combination of monomials of degree $(m,n)$. Here $(\mathrm{deg} >  (m,n))$ means terms of degree greater than $(m,n)$.
\end{lemma}

\begin{proof}
The proof is by induction on $(m,n)\in\modN^2_+$. Let us consider the case $(1,0)$. Notice that for $u_j\in \la Y\ra
$ and $\epsilon_i=\pm{1}$ we have
$\theta({}^{u_j}(x_i^{\epsilon_i})) = 1 + \epsilon_i x_i + (\mathrm{deg} > (1,0))$. Hence if $z\in K_{1,0}$ we can write it in the form
\begin{gather*}
z=\prod_{j=1}^s {}^{u_j}(z_j^{\epsilon_j})
\end{gather*}
with $u_j\in \la Y\ra 
$, $z_j\in X$ and $\epsilon_j=\pm{1}$. Therefore
\begin{gather*}
\theta(z)= 1 + \sum_{j=1}^s \epsilon_jz_j + (\mathrm{deg} > (1,0)).
\end{gather*}
The case $(0,1)$ is proved similarly.
Suppose $m\geq 2$, since $K_{m,0}$  ($m\geq 2$) is generated by elements of the form $[z,w]$ with $z\in K_{1,0}$ and $w\in K_{m-1,0}$, we only need to show the result for these kind of elements. 
Write
$$\theta(z)= 1 + z_{1,0} + (\mathrm{deg}> (1,0)) \quad \quad \text{and} \quad \quad \theta(w)= 1 + w_{m-1,0} + (\mathrm{deg}> (m-1,0)).$$
We can check that
$$\theta([z,w]) = 1+ z_{1,0}w_{m-1,0} - w_{m-1,0}z_{10} + (\mathrm{deg}> (m,0)).$$
The cases $(0,n)$ with $n\geq 2$ and $(m,n)\geq (1,1)$ are treated similarly.
\end{proof}

Let $z\in K$. For $(m,n)\in\modN^2_+$ we denote by $\delta_{m,n}(z)$ the linear combination of monomials of degree $(m,n)$ in $\theta(z)$.  For instance $\delta_{1,0}([x_1,y_1])=\delta_{0,1}([x_1,y_1])=0$ and $\delta_{1,1}([x_1,y_1])= x_1y_1-y_1x_1$.

\begin{lemma} Let $(m,n)\in\modN^2_+$. The map 
\begin{gather}\label{e43}
\delta_{m,n}:\bK_{m,n}\longrightarrow \modZ\la X_1,\ldots,X_p,Y_1,\ldots, Y_q\ra
\end{gather}
which sends $zK_{>(m,n)}$ to $\delta_{m,n}(z)$ for $z\in K_{m,n}$, is a well-defined linear map.
\end{lemma}
\begin{proof}
By Lemma~\ref{r52} we have $\delta_{m,n}(K_{>(m,n)})=0$. Therefore \eqref{e43} is well-defined. The linearity follows {from} multiplicativity of the Magnus expansion.
\end{proof}

We define the linear map 
$$\delta\zzzcolon \bK\longrightarrow \modZ\la X_1,\ldots,X_p,Y_1,\ldots, Y_q\ra$$
by $\delta =\bigoplus_{(m,n)\in\modN^2_+} \delta_{m,n}$.

\medskip

For $R,S\in\modZ\la X_1,\ldots,X_p,Y_1,\ldots, Y_q\ra$ we set $[R,S]=RS-SR$. Finally we use the following well-known result. 

\begin{proposition}[{\cite[Theorem 5.9]{MR2109550}}]
The map 
$$\mu:\mathfrak{Lie}(\mathsf{A}, \mathsf{B})\longrightarrow \modZ\la X_1,\ldots,X_p,Y_1,\ldots, Y_q\ra$$
induced by mapping $a_i\mapsto X_i$ ($1\leq i\leq p$), $b_j\mapsto Y_j$ ($1\leq j\leq q$) and compatibility with the Lie bracket, is injective.
\end{proposition}

\begin{proof}[Proof of Theorem \ref{r51}]
Clearly, the map $\varphi$ is an $\modN^2_+$-graded Lie algebra homomorphism. By Proposition~\ref{r48} and Proposition~\ref{rr49} we have that $\varphi$ is surjective. The injectivity follows by checking that the map $\mu^{-1}\delta$ is the left inverse.
\end{proof}

We finish this subsection with the following.

\begin{conjecture}\label{conjecture_2} For all $(m,n)\in\modN^2$ we have
$$ K_{m,n} = K_{m,0}\cap K_{0,n}.$$
\end{conjecture}
This conjecture imply a similar property for the Johnson filtration of the Goeritz group, see Proposition~\ref{dfstrong}.

\subsection{Johnson filtrations for groups acting on double lower   central series}\label{sec_2.4}

Let $\bX$ and $\bY$ be two normal subgroups of a group $K$ {such that $K=\bX\bY$}. Let $K_{*,*}=\Gamma _{*,*}(K;\bX,\bY)$ be the double lower central series of the triple $(K;\bX,\bY)$.

Let $G$ a group acting on $K=K_{0,0}$ (for instance $G$ could be a subgroup of the automorphism group $\Aut(K)$ of~$K$).  Let $G_{*,*}=(G_{m,n})_{(m,n)\in \modN ^2}$ be the Johnson filtration of~$G_{0,0}$. In general, the definition of each~$G_{m,n}$, see \eqref{e27} and \eqref{e28}, involves infinitely many inclusion conditions, but in the case \mbox{$K_{*,*}=\Gamma _{**}(K;\bX,\bY)$} we need only two such conditions as follows.

\begin{lemma}
  \label{r30}
  We have
  \begin{gather}\label{e29}
    G_{0,0}=\{g\in G\mid g(K_{1,0})=K_{1,0},\quad g(K_{0,1})=K_{0,1}\}
  \end{gather}
  and
  \begin{gather}\label{e30}
    G_{m,n}=\{g\in G_{0,0}\mid [g,K_{1,0}]\subset K_{m+1,n},\quad [g,K_{0,1}]\subset K_{m,n+1}\}
  \end{gather}
  for $(m,n)>(0,0)$.
\end{lemma}

\begin{proof}
To show equality \eqref{e29} it is enough to check that if $g\in G$ satisfies $g(K_{1,0})=K_{1,0}$ and {$g(K_{0,1})=K_{0,1}$}, then $g(K_{i,j})=K_{i,j}$ for all $(i,j)\in\modN ^2$. This follows easily by induction. Let us show equality \eqref{e30}; the inclusion $\subset$ is obvious. For $g\in G_{0,0}$ such that
  \begin{equation*}
    [g,K_{1,0}]\subset K_{m+1,n},\quad [g,K_{0,1}]\subset K_{m,n+1},
  \end{equation*}
  we need to show that
  \begin{equation*}
    [g,K_{i,j}]\subset K_{m+i,n+j}\quad \text{for all } (i,j)\in\mathbb{N}^2.
  \end{equation*}
  The cases $(i,j)=(0,0),(1,0),(0,1)$ are trivial.

  Consider the case $i\geq 2$, $j=0$.
  Then, by induction on $i$,
  \begin{equation*}
    \begin{split}
    [g,K_{i,0}]
    &=[g,[K_{1,0},K_{i-1,0}]]\\
    &\subset \langle \langle [[\langle g\rangle ,K_{1,0}],K_{i-1,0}]\;[[\langle g\rangle ,K_{i-1,0}],K_{1,0}]\rangle \rangle \\
    &\subset \langle \langle [K_{m+1,n},K_{i-1,0}]\;[K_{m+i-1,n},K_{1,0}]\rangle \rangle \\
    &\subset \langle \langle K_{m+i,n}\rangle \rangle =K_{m+i,n},
    \end{split}
  \end{equation*}
  where $\langle \langle -\rangle \rangle =\langle \langle -\rangle \rangle_{K\rtimes G}$.
  Similarly, the case $i=0$, $j\geq 2$ is proved.

  Finally, let $i,j\geq 1$.
  Then
  \begin{equation*}
    \begin{split}
      [g,K_{i,j}]
    &=[g,[K_{1,0},K_{i-1,j}]\;[K_{0,1},K_{i,j-1}]]\\
    &\subset {[g,[K_{1,0},K_{i-1,j}]]\;\lala [g,[K_{0,1},K_{i,j-1}]] \rara_{K\rtimes G}}.
    \end{split}
  \end{equation*}
  Thus it suffices to check $[g,[K_{1,0},K_{i-1,j}]]\subset K_{m+i,n+j}$ and
  $[g,[K_{0,1},K_{i,j-1}]]\subset K_{m+i,n+j}$.
  The former is checked by induction as follows:

  \begin{equation*}
    \begin{split}
    [g,[K_{1,0},K_{i-1,j}]]
    &\subset \langle \langle [[g,K_{1,0}],K_{i-1,j}]\;[[g,K_{i-1,j}],K_{1,0}]\rangle \rangle _{K\rtimes G}\\
    &\subset \langle \langle [K_{m+1,n},K_{i-1,j}]\;[K_{m+i-1,n+j},K_{1,0}]\rangle \rangle _{K\rtimes G}\\
    &\subset \langle \langle K_{m+i,n+j}\rangle \rangle_{K\rtimes G}=K_{m+i,n+j}.
    \end{split}
  \end{equation*}
  The latter is checked similarly.
\end{proof}

\begin{remark}[Andreadakis problem] Let $K_{*,*}$ be the double lower central series of the triple $(K,\bX,\bY)$ where $K$ is a free group and $\bX$ and $\bY$ are as in Section~\ref{sec_dlcsfg}. Let  $G=\mathrm{Aut}(K;\bX,\bY)$ be the subgroup of automorphisms of $K$ preserving $\bX$ and $\bY$.  The group $G$ acts naturally on $K_{*,*}$. Let~$G_{*,*}$ be the Johnson filtration of $G$. Besides, consider the double lower central series $\Gamma_{*,*} (G; G_{1,0}, G_{0,1})$ of the triple~$(G; G_{1,0}, G_{0,1})$. It follows by definition and Proposition~\ref{r14} that $\Gamma_{m,n} (G; G_{1,0}, G_{0,1})\subset G_{m,n}$ for all $(m,n)\in\modN^2$. The Andreadakis problem consists in the comparison of the two $\modN^2$-filtrations~$\Gamma_{*,*} (G; G_{1,0}, G_{0,1})$ and~$G_{*,*}$. {We refer to \cite{satoh} for details on the classical Andreadakis problem.}
\end{remark}

\section{Double Johnson filtration for the mapping class group}\label{sec_three}

Throughout this section let $g$ be a nonnegative integer and let  $\mathbb{N}^2$ be the additive monoid  with its usual order. In this section we apply the general theory from Sections \ref{sec_1} and \ref{sec_2} to define an $\mathbb{N}^2$-filtration for the subgroup of the mapping class group consisting of the elements which preserve the standard Heegaard splitting of the sphere, called the \emph{Goeritz group} of the $3$-sphere ${S}^3$. 
 Then, we extend this $\mathbb{N}^2$-filtration in order to obtain a double filtration for the whole mapping class group and study its properties.

{
\begin{notation}  For a continuous map $f:(S,s)\rightarrow (T,t)$  between  pointed topological spaces  $(S,s)$ and~$(T,t)$, we denote by $f_{\#}:\pi_1(S,s)\rightarrow\pi_1(T,t)$ and $f_{*}:H_1(S;\mathbb{Z})\rightarrow H_1(T;\mathbb{Z})$ the induced maps on the fundamental group and the first homology group, respectively. 
\end{notation}
}

\subsection{The Goeritz group \texorpdfstring{$\mathcal{G}$}{G} of \texorpdfstring{${S}^3$}{S3}}\label{sec_3_1}  Let $V_g$ denote the standardly embedded handlebody of genus $g$ in ${S}^3$, that is, such that the closure $V'_g:=\overline{{S}^{3}\setminus V_g}$ is also a handlebody (necessarily of genus $g$). A more explicit description of $V_g$ is as follows. Consider the ambient space ${S}^3=\mathbb{R}^3\cup\{\infty\}$ with coordinates $(x,y,z)$. Then $V_g$ is the compact, oriented $3$-manifold obtained from the standard cube $[-1,1]^3\subset\mathbb{R}^3$ by adding $g$  (unknotted) $1$-handles uniformly in the~$y$ direction, see Figure~\ref{sec_2_fig_1}~$(b)$. 
 
Let $\Sigma_{g,1}$ be a compact, oriented surface of genus $g$ with one boundary component and standardly embedded in $\mathbb{R}^3\subset {S}^3$. Explicitly, $\Sigma_{g,1}$ is obtained from the square $[-1,1]\times [-1,1]\times\{0\}$ by adding $g$ (unknotted) handles uniformly in the $y$ direction, see Figure~\ref{sec_2_fig_1}~$(a)$.

\begin{figure}[ht!]
\centering
\includegraphics[scale=0.85]{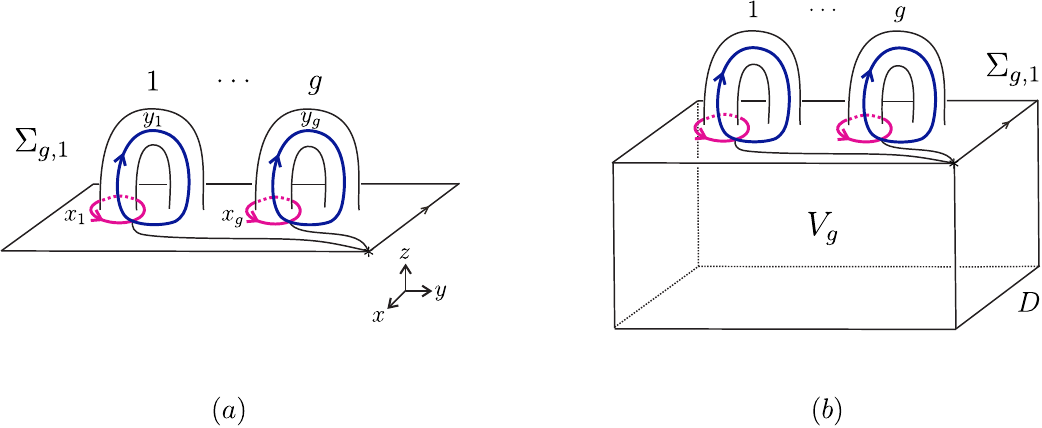}			\caption{$(a)$ Surface $\Sigma_{g,1}$ standardly embedded in ${S}^3$. $(b)$ Handlebody $V_g$ standardly embedded in $S^{3}$ and decomposition $\partial V_g = \Sigma_{g,1}\cup D$.}		\label{sec_2_fig_1}
\end{figure}

We have  $\partial V_g =\partial V'_g =\Sigma_{g,1}\cup D$, where $D\subset\partial V_g=\partial V'_g$ is a fixed disk. Notice that an orientation of $\partial \Sigma_{g,1}$ determines an orientation of $\Sigma_{g,1}\subset\partial V_g\cap\partial V'_g$, $V_g$ and $V_g'$ by using the outward normal first vector convention. From now on, we consider orientations as in Figure \ref{sec_2_fig_1}.

Fix $*\in\partial\Sigma_{g,1}=\partial D\subset V_g\cap V'_g$. Let $K:=\pi_1(\Sigma_{g,1},*)$. We fix a system of meridians and parallels $\{x_i,y_i\}_{1\leq i\leq g}$ and we join them with the base point $*\in\partial\Sigma_{g,1}$ as in Figure~\ref{sec_2_fig_1}. Thus, $\{x_i,y_i\}_{1\leq i\leq g}$ is a free basis for $K$. There are two natural embeddings $\iota\zzzcolon\Sigma_{g,1}\rightarrow V_g$ and $\iota'\zzzcolon\Sigma_{g,1}\rightarrow V'_g$. Consider the following subgroups of $K$:
\begin{gather}\label{equ_X_Y}
\bar{X}:=\mathrm{ker}\big(K\xrightarrow{\  \iota_{\#}\  }\pi_1(V_g,*)\big)=\lala x_1,\dots,x_g\rara_K, \quad
\bar{Y}:=\mathrm{ker}\big(K\xrightarrow{\  \iota'_{\#}\  }\pi_1(V'_g,*)\big)=\lala y_1,\dots,y_g\rara_K.
\end{gather}

From now on we make the following identifications
  \begin{gather*}
\pi_1(V_g,*)=K/\bar{X},\quad \pi_1(V'_g,*)=K/\bar{Y}.
  \end{gather*}

Consider also the following  subgroups of~$H:=H_1(\Sigma;\Z)$:
\begin{equation}\label{equ_A_B}
A:=\mathrm{ker}\big(H\xrightarrow{\  \iota_{*}\  }H_1(V_g;\mathbb{Z})\big) \text{\ \ \ \ \ and \ \  \ \ \ } B:=\mathrm{ker}\big(H\xrightarrow{\  \iota'_{*}\  }H_1(V'_g;\mathbb{Z})\big).
\end{equation}

Let $\mathcal{M}_{g,1}$ denote the \emph{mapping class group} of the surface $\Sigma_{g,1}$, that is, the group of isotopy classes of orientation-preserving homeomorphisms $h\zzzcolon\Sigma_{g,1}\rightarrow \Sigma_{g,1}$ fixing the boundary~$\partial\Sigma_{g,1}$ pointwise. The group~$\mathcal{M}_{g,1}$ acts naturally on $K$. From now on, we consider the semidirect product $K\rtimes \mathcal{M}_{g,1}$. For simplicity of notation, for $f\in\M$ we will also denote by $f$ the induced map on $K$ (instead of $f_{\#}$) since the meaning is clear from the context.

The \emph{handlebody group} $\mathcal{H}_{g,1}$  (resp. $\mathcal{H}'_{g,1}$) \emph{relative to the disk~$D$} is  the mapping class group of~$V_g$ (resp.~$V'_g$) relative to the disk~$D$, i.e., the group of isotopy classes of orientation-preserving homeomorphisms \mbox{$h\zzzcolon V_g\rightarrow V_g$} (resp. $h\zzzcolon V'_g\rightarrow V'_g$) which preserve the disk $D$. Equivalently, the group $\mathcal{H}_{g,1}$ (resp.~$\mathcal{H}'_{g,1}$) can be defined as the subgroup of $\mathcal{M}_{g,1}$ consisting of isotopy classes of homeomorphisms of~$\Sigma_{g,1}$  which extend  to a homeomorphism of~$V_g$ (resp.~$V'_g$).  Dehn's lemma implies the following explicit description of $\mathcal{H}_{g,1}$ and $\mathcal{H}'_{g,1}$, see~\cite[Theorem~10.1]{MR0159313}.
\begin{gather}
\mathcal{H}_{g,1} =\big\{h\in\mathcal{M}_{g,1}\ |\ h( \bar{X} ) = \bar{X}\big\}  \text{\ \ \  and\ \ \   }
\mathcal{H}'_{g,1} =\big\{h\in\mathcal{M}_{g,1}\ |\ h( \bar{Y} ) = \bar{Y}\big\}.
\end{gather}

{The subgroups $\mathcal{H}_{g,1}$ and $\mathcal{H}'_{g,1}$ of  $\mathcal{M}_{g,1}$ are infinite, of infinite index and not normal.} They are finitely generated \cite[Theorem 4.1]{MR433433}. (The arguments given in \cite[Theorem 4.1]{MR433433} are for the case of a closed surface. The case of a surface with one boundary component can be treated with the same arguments).
The subgroup $A$ in \eqref{equ_A_B} allows to define the homological version of the handlebody group. The \emph{Lagrangian mapping class group $\mathcal{L}_{g,1}$ associated with $A$} is the subgroup of $\M_{g,1}$ given by
\begin{equation}
\mathcal{L}_{g,1}=\big\{h\in\mathcal{M}_{g,1}\ |\ h_*( A ) = A\big\}. 
\end{equation}
The group $\mathcal{H}_{g,1}$ (resp.~$\mathcal{H}'_{g,1}$) acts on $\pi_1(V_g,*)\simeq K/\bar{X}$ (resp.~$\pi_1(V'_g,*)\simeq K/\bar{Y}$). 
The \emph{twist groups} or \emph{Luft groups} $\mathcal{T}_{g,1}$ of $V_g$ , and $\mathcal{T}'_{g,1}$ of $V'_g$, (relative to the disk $D$) are defined as the kernels of these actions, that is, 
\begin{equation}
\begin{split}
\mathcal{T}_{g,1} & =\mathrm{ker}\big(\mathcal{H}_{g,1}\longrightarrow \mathrm{Aut}(\pi_1(V_g,*))\big)=\big\{h\in\mathcal{M}_{g,1} \ | \ h(\bar{X})= \bar{X},\  [h,K]\subset\bar{X}\big\}  \text{\ \ \  and }\\
\mathcal{T}'_{g,1} &=\mathrm{ker}\big(\mathcal{H}'_{g,1}\longrightarrow \mathrm{Aut}(\pi_1(V'_g,*))\big)=\big\{h\in\mathcal{M}_{g,1} \ | \ h(\bar{Y})= \bar{Y},\  [h,K]\subset\bar{Y}\big\}.
\end{split}
\end{equation}

The following lemma is stated in a different form in \cite[Example 10.9]{HM}. (See also \cite{HM_2}.)

\begin{lemma}\cite{HM_2,HM}\label{sec_2_lem_T} We have
\begin{equation*}
\begin{split}
\mathcal{T}_{g,1} & =\big\{h\in\mathcal{M}_{g,1} \ | \ h(\bar{X})= \bar{X},\  [h,\bar{Y}]\subset\bar{X}\big\} = \big\{h\in\mathcal{M}_{g,1} \ | \ h(\bar{X})= \bar{X},\  [h,\bar{X}]\subset [\bar{X},\bar{X}]\big\}, \\
\mathcal{T}'_{g,1} &=\big\{h\in\mathcal{M}_{g,1} \ | \ h(\bar{Y})= \bar{Y},\  [h,\bar{X}]\subset\bar{Y}\big\} =  \big\{h\in\mathcal{M}_{g,1} \ | \ h(\bar{Y})= \bar{Y},\  [h,\bar{Y}]\subset [\bar{Y},\bar{Y}]\big\}.
\end{split}
\end{equation*}
\end{lemma}

The  \emph{genus $g$ Goeritz group $\mathcal{G}_{g,1}$ of ${S}^3$} \emph{(relative to the disk $D$)} is the group of isotopy classes of orientation-preserving homeomorphism $h\zzzcolon {S}^3\rightarrow {S}^3$ such that $h(\partial V_g)=\partial V_g$ and $h(D)=D$. Equivalently, it is the subgroup of $\mathcal{M}_{g,1}$ consisting of the mapping classes which preserve the standard Heegaard splitting of the~$3$-sphere.   Waldhausen's theorem implies the following explicit description:
\begin{equation}
\mathcal{G}_{g,1} =\mathcal{H}_{g,1}\cap\mathcal{H}'_{g,1}=\big\{h\in\mathcal{M}_{g,1}\ |\ h(\bar{X})= \bar{X} \text{ and } h(\bar{Y})= \bar{Y}\big\}.
\end{equation}

For $g\geq 4$ it is not known (at the moment of writing this paper) if $\mathcal{G}_{g,1}$ is finitely generated or not.

\begin{notation} From now on, we drop the terminology ``relative to the disk $D$'' and we drop the subscripts from the notation for simplicity. Thus, we will continue to write $\Sigma$, $V$, $V'$, $\mathcal{M}$, $\mathcal{H}$, $\mathcal{H}'$, $\mathcal{L}$, $\mathcal{T}$, $\mathcal{T}'$ and $\mathcal{G}$ instead of the corresponding notations with subscripts. 
\end{notation}

Let $\omega\zzzcolon H\otimes H\rightarrow \mathbb{Z}$ be the intersection form of $\Sigma$.
Let
$$
\mathrm{Sp}(H,\omega)=\{f\in\mathrm{Aut}(H) \ | \  \omega(f(x),f(y))=\omega(x,y) \  \text{ for all } x,y\in H\} 
$$
be the symplectic group of $(H,\omega)$. We have a surjective group homomorphism
\begin{equation}
\sigma\zzzcolon \mathcal{M}\longrightarrow \mathrm{Sp}(H,\omega),
\end{equation}
that sends $h\in\mathcal{M}$ to the induced map $h_*$ on $H$. The homomorphism $\sigma$ is known as the \emph{symplectic representation} of $\mathcal{M}$. Its kernel, denoted by $\mathcal{I}$, is known as the \emph{Torelli group} of~$\Sigma$.  Thus, we have a short exact sequence
\begin{equation}
1\longrightarrow \mathcal{I}\longrightarrow
\mathcal{M} \xrightarrow{\ \sigma \ } \text{Sp}(H,\omega)\longrightarrow 1.
\end{equation}
Let $\{a_i,b_i\}_{1\leq i\leq g}$ be the  symplectic basis of $H$ induced by  the system of meridians and parallels $\{x_i,y_i\}_{1\leq i\leq g}$ on $\Sigma$ shown in Figure \ref{sec_2_fig_1}.  We use this basis to identify $\mathrm{Sp}(H,\omega)$  
with the group  $\mathrm{Sp}(2g,\mathbb{Z})$ of \mbox{$(2g)\times(2g)$} matrices $M$ with integer entries such that  $M^TJM=J$, where $J$ is the standard  skew-symmetric matrix~$\left( \begin{smallmatrix} 0&\text{Id}_g\\ -\text{Id}_g&0 \end{smallmatrix} \right)$. The results in the following lemma are well known, see for instance~\cite[Lemma~6.3]{MR2265877} for~(iv) and~\cite[Section~3]{MR433433} for~(iii), the case~(i) can be deduced from these two ({see also \cite[Lemma 2.2]{birman1975}}) and the remaining cases follows by dual arguments.   We will sketch a proof for completeness.

\begin{lemma}\label{im_symp} We have
\begin{itemize}
\item[\emph{(i)}]  $\sigma(\mathcal{H}) { =\sigma(\mathcal{L})}= \left\{\left( \begin{smallmatrix} P &Q\\ 
0& (P^T)^{-1}\end{smallmatrix} \right) \  \bigg\rvert \ \ P^{-1}Q \text{\ is symmetric}  \right\}$,

\item[\emph{(ii)}]  $\sigma(\mathcal{H}') =\left\{\left( \begin{smallmatrix} P &0\\ 
Q& (P^T)^{-1}\end{smallmatrix} \right) \  \bigg\rvert \ \ P^{T}Q \text{\ is symmetric}  \right\}$,

\item[\emph{(iii)}]  $\sigma(\mathcal{G})= \left\{\left( \begin{smallmatrix} P &0\\ 
0& (P^T)^{-1}\end{smallmatrix} \right) \  \bigg\rvert \ \ P\in\mathrm{GL}(g,\mathbb{Z})\right\}\simeq \mathrm{GL}(g,\mathbb{Z})$,

\item[\emph{(iv)}]  $\sigma(\mathcal{T})= \left\{\left( \begin{smallmatrix} \mathrm{Id}_g &Q\\ 
0& \mathrm{Id}_g\end{smallmatrix} \right) \  \bigg\rvert \ \ Q \text{\ is symmetric}  \right\}\simeq \mathrm{Sym}(g,\mathbb{Z})$,

\item[\emph{(v)}]  $\sigma(\mathcal{T}')= \left\{\left( \begin{smallmatrix} \mathrm{Id}_g &0\\ 
Q& \mathrm{Id}_g\end{smallmatrix} \right) \  \bigg\rvert \ \ Q \text{\ is symmetric}  \right\}\simeq \mathrm{Sym}(g,\mathbb{Z})$,
\end{itemize}
where $\mathrm{Sym}(g,\mathbb{Z})\simeq \mathbb{Z}^{\frac{1}{2}g(g+1)}$ is the group of $g\times g$ symmetric matrices with integer coefficients. In particular, we have the following short exact sequences
\begin{equation}\label{ses1}
1\longrightarrow \mathcal{I}\cap\mathcal{G}\longrightarrow \mathcal{G}
\xrightarrow{\ \sigma\ } \mathrm{GL}(g,\mathbb{Z})\longrightarrow 1,
\end{equation}
\begin{equation}
1\longrightarrow \mathcal{I}\cap\mathcal{T}\longrightarrow \mathcal{T} \xrightarrow{\ \sigma\ } \mathrm{Sym}(g,\mathbb{Z})\longrightarrow 1 \quad \quad \text{and} \quad\quad  1\longrightarrow \mathcal{I}\cap\mathcal{T}'\longrightarrow \mathcal{T}' \xrightarrow{\ \sigma\ } \mathrm{Sym}(g,\mathbb{Z})\longrightarrow 1.
\end{equation}
\end{lemma}

\begin{proof}
By definition, we have that the target groups on each case are as stated. We only need to check the surjectivity on each case. Let $t_{x_k}$ be the Dehn twist about the  meridian curve $x_k$ (shown in Figure~\ref{sec_2_fig_1}~$(a)$) for $1\leq k\leq g$. Similarly, consider the Dehn twist $t_{x_{ij}}$ about the curve $x_{ij}$ (shown in Figure~\ref{sec_2_fig_9}~$(a)$) for $1\leq i< j\leq g$. We have $t_{x_k}, t_{x_{ij}}\in\mathcal{T}$ and they realize a basis of the target group in~(iv).

The group $\mathrm{GL}(g,\mathbb{Z})$ is generated by elementary matrices. More precisely, let $1\leq i,k\leq g$ with $i\not= k$, for $r\in\modZ$ let $e_{ik}(r)$ be the $g\times g$ matrix with $1$'s in the diagonal, $r$ in the position $(i,k)$ and~$0$ elsewhere. For $1\leq j\leq g$, let $\mathrm{diag}_j(t)$ be the $g\times g$  diagonal matrix with $t\in\{\pm 1\}$ in the position $(j,j)$ and~$1$ elsewhere in the diagonal. The families of matrices $e_{ik}(r)$ and $\mathrm{diag}_j(-1)$ generate $\mathrm{GL}(g,\mathbb{Z})$. We proceed to realize them as the image under $\sigma$ of elements in $\mathcal{G}$.

Let $1\leq i,k\leq g$  with $i\not= k$, and let $\phi_{ik}\in \mathcal{G}$ be the composition $t_{\gamma_{ik}}t_{y_i}^{-1}t_{x_k}^{-1}$  where $y_i$ and $x_k$  the $i$-th parallel and $k$-th meridian respectively (see Figure~\ref{sec_2_fig_1}~$(a)$); and $\gamma_{ik}$ is the simple closed curve which goes around the $i$-th parallel and the $k$-th meridian as shown in Figure~\ref{sec_2_fig_9}~$(b)$. We have \mbox{$\sigma(\phi_{ik}) =\left( \begin{smallmatrix} e_{ik}(1) &0\\ 
0& e_{ki}(-1)\end{smallmatrix} \right)$}. Hence the family $\phi_{ik}^r$ realize the family $e_{ik}(r)$.

\begin{figure}[ht!]
\centering   	
	\includegraphics[scale=0.73]{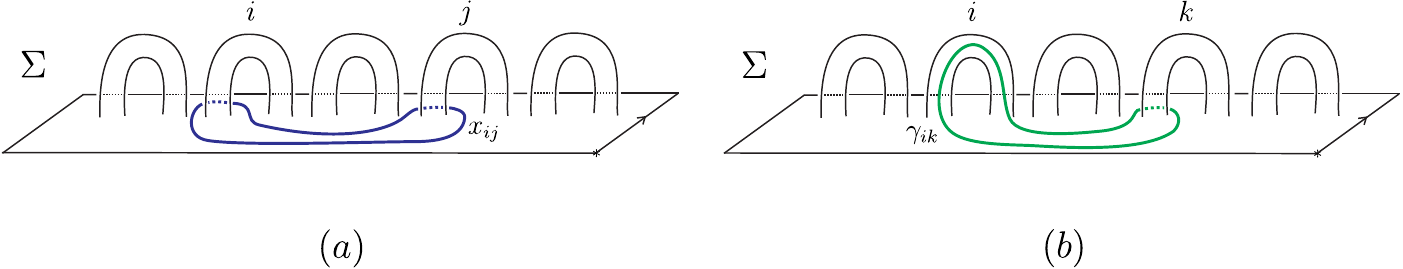}
	\caption{($a$) curve $x_{ij}$ \quad  ($b$) curve $\gamma_{ik}$.}
\label{sec_2_fig_9}
\end{figure}

Let $h_j$ be the element in $\mathcal{G}$ which rotates the $j$-th handle an angle of $\pi$ radians, this element is known as a \emph{knob twist}, see \cite[Section 3.2]{MR433433} for a precise description. We have $\sigma(h_j) = \left( \begin{smallmatrix} \mathrm{diag}_j(-1) &0\\ 
0& \mathrm{diag}_j(-1)\end{smallmatrix} \right)$. In this way we have realized a generating set of the target group in~(iii).

Consider a matrix $M=\left( \begin{smallmatrix} P &Q\\ 
0& (P^T)^{-1}\end{smallmatrix} \right)$ with $P^{-1}Q$  symmetric. By (iv) and (iii), there exist $h\in\mathcal{T}$ and $f\in\mathcal{G}$ such that 
$$ 
\sigma(h) = \left( \begin{smallmatrix} 0 &P^{-1}Q\\ 
0& 0\end{smallmatrix} \right)\quad \text{and} \quad\quad \sigma(f)=\left( \begin{smallmatrix} P &0\\ 
0& (P^T)^{-1}\end{smallmatrix} \right),
$$
hence $fh\in\mathcal{H}\subset\mathcal{L}$ and $\sigma(fh)=M$. Thus
we obtain (i). The remaining cases follow by  dual arguments.
\end{proof}

\subsection{Double Johnson filtration for the Goeritz group \texorpdfstring{$\mathcal{G}$}{G}}
 Consider the group $K=\pi_1(\Sigma,*)=\la
   x_1,\dots,x_g,y_1,\dots,y_g\ra$ and its normal subgroups  $\bar{X}=\lala x_1,\dots,x_g\rara_K$ and $\bar{Y}=\lala
   y_1,\dots,y_g\rara_K$ defined in~\eqref{equ_X_Y}. Let
 $K_{*,*}=(K_{m,n})_{(m,n)\in\mathbb{N}^2}$ be the double lower
 central series of the triple $(K;\bar{X},\bar{Y})$. Since $\mathcal{G}$ acts on $K$ preserving $\bar{X}$ and $\bar{Y}$, the group
$\mathcal{G}$ acts on the $\mathbb{N}^2$-filtration
$K_{*,*}$. By Section \ref{sec_2} the Johnson
$\mathbb{N}^2$-filtration $(\mathcal{G}_{m,n})_{(m,n)\in\mathbb{N}^2}$
of the Goeritz group $\mathcal{G}$ is the family of normal subgroups
$\mathcal{G}_{m,n}$ of $\mathcal{G}$ given by
\begin{equation}
\begin{split}
\mathcal{G}_{m,n} &= \big\{ h\in\mathcal{G}\ | \ [h,K_{i,j}]\subset K_{m+i,n+j} \ \text{for all } (i,j)\in\mathbb{N}^2\big\}\\
& = \big\{ h\in\mathcal{G}\ | \ [h,\bar{X}]\subset K_{m+1,n}, \ [h, \bar{Y}]\subset K_{m,n+1}\big\}
\end{split}
\end{equation}
for all $(m,n)\in \mathbb{N}^2$, where the second equality follows from Lemma \ref{r30}.  We call the $\mathbb{N}^2$-filtration $(\mathcal{G}_{m,n})_{(m,n)\in\mathbb{N}^2}$ the \emph{double Johnson filtration} of the Goeritz group~$\mathcal{G}$.

Recall that this is the slowest decreasing $\mathbb{N}^2$-filtration for $\mathcal{G}$   acting on $K_{*,*}$, see Proposition~\ref{r14}.

\begin{proposition}\label{dfstrong}
If Conjecture \ref{conjecture_2} holds, then we have
$\mathcal{G}_{m,n}=\mathcal{G}_{m,0}\cap\mathcal{G}_{0,n}$ for all $(m,n)\in\modN^2$.
\end{proposition}

\begin{proof}
Suppose
  \begin{gather}\label{e52}
K_{i,j}=K_{i,0}\cap K_{0,j}
\end{gather}
for all $(i,j)\in\modN^2$. Let us show $\mathcal{G}_{m,0}\cap\mathcal{G}_{0,n}\subset \mathcal{G}_{m,n}$ for $(m,n)\in\modN^2$. We have 
$$[\mathcal{G}_{m,0}\cap\mathcal{G}_{0,n}, K_{i,j}]\subset [\mathcal{G}_{m,0},K_{i,j}]\subset K_{i+m,j}.$$
Similarly $[\mathcal{G}_{m,0}\cap\mathcal{G}_{0,n}, K_{i,j}]\subset K_{i,j+n}$. Hence 
$$[\mathcal{G}_{m,0}\cap\mathcal{G}_{0,n}, K_{i,j}]\subset K_{i+m,j}\cap K_{i,n+j}.$$
By assumption  \eqref{e52} we have
$$K_{i+m,j}\cap K_{i,j+n}=(K_{i+m,0}\cap K_{0,j})\cap (K_{i,0}\cap K_{0,j+n})=K_{i+m,0}\cap K_{0,j+n}=K_{i+m,j+n}.$$
Therefore $[\mathcal{G}_{m,0}\cap\mathcal{G}_{0,n}, K_{i,j}]\subset K_{i+m,j+n}$ for all $(i,j)\in\modN^2$, that is,  $\mathcal{G}_{m,0}\cap\mathcal{G}_{0,n}\subset \mathcal{G}_{m,n}$. Hence  for all $(m,n)\in\modN^2$ we have $\mathcal{G}_{m,n}=\mathcal{G}_{m,0}\cap\mathcal{G}_{0,n}$.
\end{proof}

\begin{example}\label{ex_m10}
 Let $1\leq i <j\leq g$ and consider the curves $\epsilon_{ij}$, $\gamma_j$ and $\delta$ on $\Sigma$ shown in Figure~\ref{sec_2_fig_2}. In order to consider these curves as elements of $K$, we join them with the base point and orient them as shown in the same figure. The Dehn twist $t_{\delta}$ is an element of~{$\mathcal{G}_{1,1}$}.
Besides, we can check that  $h_{ij}=t_{\epsilon_{ij}}t^{-1}_{\gamma_j}$ belongs to~{$\mathcal{G}_{1,0}$}. 
\begin{figure}[ht!]
\centering   	
	\includegraphics[scale=0.8]{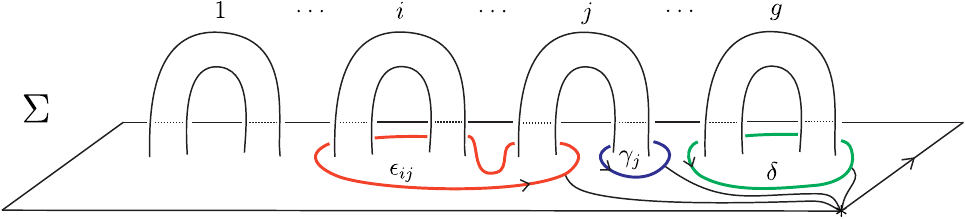}
	\caption{Curves $\epsilon_{ij}$, $\gamma_j$ and $\delta$ on $\Sigma$.}
\label{sec_2_fig_2}
\end{figure}
\end{example}

\subsection{Double Johnson filtration for the mapping class group}\label{sec_31}

We extend the double lower central series $(K_{m,n})_{(m,n)\in\mathbb{N}^2}$ of $(K,\bar{X},\bar{Y})$ to all $(m,n)\in\modZ^2$ by setting 
\begin{equation}\label{dlex}
K_{m,n}:=K_{\mathrm{max}(0,m), \mathrm{max}(0,n)}.
\end{equation}
Note that if $m,m',n,n'\in\Z$, $m\le m'$, $n\le n'$, then we have $K_{m,n}\supset K_{m',n'}$.

The double filtration $(K_{m,n})_{(m,n)\in\Z^2}$ yields a $\Z^2$-indexed filtration of $\M$:
\begin{equation}
\begin{split}
  \mathcal{M}_{m,n} :=&\big\{h\in\mathcal{M}\ |\ [h^{\pm1}, K_{i,j}]\subset K_{m+i,n+j} \text{ for every } (i,j)\in\mathbb{N}^2\big\}\\
  =&\big\{h\in\mathcal{M}\ |\ [h^{\pm1}, K_{i,j}]\subset K_{m+i,n+j} \text{ for every } (i,j)\in\mathbb{Z}^2\big\}.
\end{split}
\end{equation}
It is easy to see that if $m,m',n,n'\in\Z$, $m\le m'$, $n\le n'$, then we have $\M_{m,n}\supset\M_{m',n'}$.

The $\Z^2$-filtration $(\M_{m,n})_{(m,n)\in\Z^2}$ of $\M$ \emph{extends} the double Johnson filtration of the Goeritz group as follows.
\begin{proposition}\label{prop2_7}
We have $\mathcal{M}_{m,n}= \mathcal{G}_{m,n}$ for all $(m,n)\in\mathbb{N}^2$. In particular, we have $\mathcal{M}_{0,0}=\G_{0,0}=\G$.
\end{proposition}

\begin{proof}
This easily follows from the definitions.
\end{proof}

\begin{proposition}\label{r59} For $(m,n)\in\mathbb{Z}^2$, we have
\begin{equation}\label{sec2_equ_12}
\mathcal{M}_{m,n}=\big\{h\in\mathcal{M}\ |\ [h^{\pm1}, K_{1,0}]\subset K_{m+1,n} \text{\ \  and\ \  } [h^{\pm1}, K_{0,1}]\subset K_{m,n+1}\big\}.
\end{equation}
Therefore,
\begin{equation}\label{sec2_equ_13}
\mathcal{M}_{m,n}=\mathcal{M}_{\mathrm{max}(-1,m),\mathrm{max}(-1,n)}
\end{equation}
for $(m,n)\in\mathbb{Z}^2$.
\end{proposition}

\begin{proof}
The proof of \eqref{sec2_equ_12} is the same as that of Lemma~\ref{r30}. 
Notice that \eqref{sec2_equ_13} easily follows from~\eqref{sec2_equ_12}. 
\end{proof}

By the above proposition we are interested in $\mathcal{M}_{m,n}$ only for $m,n\geq -1$. We call the family $(\mathcal{M}_{i,j})_{i,j\geq -1}$ \emph{the double Johnson filtration of the mapping class group} $\M$.

\subsection{The groups \texorpdfstring{$\M_{m,n}$}{M-m,n} for small \texorpdfstring{$m,n$}{m,n}}\label{sec_4_4}

In the following, we look at $\M_{m,n}$ for some small $m,n\ge-1$.

\begin{proposition}\label{prop2_8}
We have $\mathcal{M}_{1,-1} = \mathcal{T}$ and $\mathcal{M}_{-1,1} = \mathcal{T}'$.
\end{proposition}

\begin{proof}
Clearly, $\mathcal{M}_{1,-1}\subset \mathcal{T}$.   The other inclusion follows {from} Lemma~\ref{sec_2_lem_T} and Proposition~\ref{r59}. Similarly, we have $ \mathcal{M}_{-1,1}=\mathcal{T}'$.
\end{proof}

\begin{proposition}\label{prop_leading}
We have 
\begin{itemize}
\item[\emph{(i)}] $\mathcal{M}_{0,-1} = \mathcal{H} = \mathcal{T}\cdot\mathcal{G}= \mathcal{M}_{1,-1}\cdot\mathcal{M}_{0,0}$.
\item[\emph{(ii)}] $\mathcal{M}_{-1,0} = \mathcal{H}' = \mathcal{T}'\cdot\mathcal{G} =  \mathcal{M}_{-1,1}\cdot\mathcal{M}_{0,0}$.

\item[\emph{(iii)}] $\mathcal{M}_{-1,-1} = \langle \mathcal{M}_{1,-1}\cdot\mathcal{M}_{0,0}\cdot\mathcal{M}_{-1,1}\rangle = \langle \mathcal{M}_{1,-1}\cdot\mathcal{M}_{-1,1}\rangle =\M$.
\end{itemize}
\end{proposition}

\begin{proof}
  (i) Clearly, $\mathcal{M}_{0,-1}\subset \mathcal{H}$. The other inclusion follows from Proposition~\ref{r59}.
The equality $\mathcal{H}=\mathcal{T}\cdot\mathcal{G}$ follows by showing that each of the generators of~$\mathcal{H}$ belongs to $\mathcal{T}$ or $\mathcal{G}$, see \cite{MR433433} where the generators of~$\mathcal{H}$ are given together with their action on homotopy from which the latter claim can be easily checked. Equality $\mathcal{T}\cdot\mathcal{G}=\mathcal{M}_{1,-1}\cdot\mathcal{M}_{0,0}$ follows from Propositions~\ref{prop2_7} and~\ref{prop2_8}.

  (ii) can be proved similarly to (i).

  (iii) can be verified, for example, by observing that each of Lickorish's
  generators of $\mathcal{M}$ is contained in either $\mathcal{T}$ or $\mathcal{T}'$.
\end{proof}

\subsection{Behavior of the double Johnson filtration under commutators and conjugates}\label{sec_4_5}

We now consider how the double Johnson filtration of $\M$ behaves under commutators and conjugates.

The commutator inclusion property for the double Johnson filtration for the Goeritz group extends to that for the mapping class group as follows.

\begin{proposition}\label{prop_2.10}
Let $(m,n),(m',n')\in \{(k,l)\in\mathbb{Z}^2\ | \  k,l\geq -1,\ k+l\geq 1\}\cup\{(0,0)\}$. Then we have
\begin{equation}\label{1}
[\mathcal{M}_{m,n}, \mathcal{M}_{m',n'}]\subset \mathcal{M}_{m+m',n+n'}.
\end{equation}
\end{proposition}
\begin{proof}
From Proposition \ref{r14} we already know that the result is true for $(m,n),(m',n')\in\mathbb{N}^2$. In fact, the same proof also works for the described set.
\end{proof}

\begin{proposition}\label{sec_2_prop_711} We have $\ \!^{\mathcal{M}_{0,0}}\mathcal{M}_{i,j}= \mathcal{M}_{i,j}$ for all $i,j\geq -1$. Thus, the double filtration $(\mathcal{M}_{i,j})_{i,j\geq -1}$ is invariant under the conjugation action of $\mathcal{M}_{0,0}$.
\end{proposition}

\begin{proof}
By Proposition \ref{prop_2.10} we have $\ \!^{\mathcal{M}_{0,0}}\mathcal{M}_{i,j}= \mathcal{M}_{i,j}$ for all $i,j\geq -1$ with $i+j\geq 1$.  The remaining cases are trivial.
\end{proof}

\begin{proposition}\label{prop_2_6}
We have $\ \!^{\mathcal{M}_{1,-1}}\mathcal{M}_{i,-1}= \mathcal{M}_{i,-1}$ for all $i\geq 0$, that is, $\mathcal{M}_{i,-1}$ is a normal subgroup of~$\mathcal{M}_{1,-1}$ for all $i\geq 0$. Similarly, $\ \!^{\mathcal{M}_{-1,1}}\mathcal{M}_{-1,j}= \mathcal{M}_{-1,j}$ for all $j\geq 0$.
\end{proposition}

\begin{proof} 
 Let $i\geq 0$, $\varphi\in\mathcal{M}_{1,-1}$ and $h\in\mathcal{M}_{i,-1}$. If $x\in K_{1,0}$, we have
\begin{equation*}
[\varphi h \varphi^{-1},x]=\varphi h \varphi^{-1}(x)x^{-1} =\varphi \big(h(\varphi^{-1}(x))\varphi^{-1}(x^{-1})\big)\in K_{i+1,0}.
\end{equation*}
If $y\in K_{0,1}$, we have
\begin{equation*}
\begin{split}
[\varphi h \varphi^{-1},y] &=\varphi h \varphi^{-1}(x)x^{-1} \\
&=\varphi \big(h(\varphi^{-1}(y)y^{-1})y\varphi^{-1}(y^{-1})\big)\cdot y\varphi(y^{-1})\cdot \varphi\big(h(y)y^{-1}\big)\cdot \varphi(y)y^{-1}\in K_{i,0}.
\end{split}
\end{equation*}
Hence, $\varphi h \varphi^{-1}\in \mathcal{M}_{i,-1}$. Similarly  $\ \!^{\mathcal{M}_{-1,1}}\mathcal{M}_{-1,j}= \mathcal{M}_{-1,j}$ for all $j\geq 0$.
\end{proof}

The following example shows that \eqref{1} does not hold for all $(m,n)$.

\begin{example} We have $[\mathcal{M}_{1,-1}, \mathcal{M}_{-1,1}]\not\subset \mathcal{M}_{0,0}$. For instance, consider the Dehn twists $t_{x_1}\in\mathcal{M}_{1,-1}$ and $t_{y_1}\in\mathcal{M}_{-1,1}$, where $x_1$ and $y_1$ are as in Figure \ref{sec_2_fig_1}. Now the image under the symplectic representation of $[t_{x_1}, t_{y_1}]$ is the matrix $\left( \begin{smallmatrix} 1 &1\\ 
1& 2\end{smallmatrix} \right)$, so  $[t_{x_1}, t_{y_1}]\notin\mathcal{M}_{0,0}$ by Lemma \ref{im_symp}. The same elements show that $[\mathcal{M}_{1,-1},\mathcal{M}_{-1,0}]\not\subset \mathcal{M}_{0,-1}$.
\end{example}

  In \cite[Example 10.9]{HM} the Johnson filtration $(\H_m)_{m\ge0}$ for the handlebody group $\H$ was introduced.  (There, $\H_m$ is denoted by $G_m$.)  The group $\H_m$ can be defined by $\H_0=\H$ and
\begin{equation}
\H_m = \{h\in\H\mid [h,K]\subset \Gamma_{m}\bX,\;[h,\bX]\subset\Gamma_{m+1}\bX\}
\end{equation}
for $m\ge1$.
Note that we have
\begin{equation}
\H_m = \{h\in\H\mid [h,\bY]\subset \Gamma_{m}\bX,\;[h,\bX]\subset\Gamma_{m+1}\bX\}
=\M_{m,-1}.
\end{equation}
Since the $(\H_m)_{m\ge0}$ is an extended N-series \cite{HM} ($\modN$-filtration in our terminology), we have the following.

\begin{proposition}[\cite{HM}]\label{2}
  For all $a,b\ge0$, we have
\begin{equation}
[\M_{a,-1},\M_{b,-1}]\subset \M_{a+b,-1}.
\end{equation}
Similarly, we also have $[\M_{-1,a},\M_{-1,b}]\subset\M_{-1,a+b}$.
\end{proposition}

{
\begin{proposition} We have
$$[\mathcal{M}_{1,-1}, \mathcal{M}_{1,0}]\subset \mathcal{M}_{2,-1}.$$
Similarly, we also have $[\mathcal{M}_{-1,1}, \mathcal{M}_{0,1}]\subset \mathcal{M}_{-1,2}.$
\end{proposition}

\begin{proof} 
We will check
$[\mathcal{M}_{1,-1}, \mathcal{M}_{1,0}]\subset \mathcal{M}_{2,-1}$. The other inclusion follows by dual arguments.
By definition $\mathcal{M}_{1,-1}$ acts on $K_{1,0}=\bar{X}$. Notice that $K_{m,0}=\Gamma_{m} K_{1,0}$ is a normal subgroup of $K_{1,0}\rtimes \mathcal{M}_{1,-1}$ for all~$m\geq 1$. 
\noindent We have 
\begin{itemize}
\item $[\mathcal{M}_{1,-1},[\mathcal{M}_{1,0},K_{1,0}]]\leq K_{3,0}, \quad\quad [\mathcal{M}_{1,0},[K_{1,0},\mathcal{M}_{1,-1}]]\leq K_{3,0},$
\item $[\mathcal{M}_{1,-1},[\mathcal{M}_{1,0},K_{0,1}]]\leq K_{2,0}, \quad\quad [\mathcal{M}_{1,0},[K_{0,1},\mathcal{M}_{1,-1}]]\leq K_{2,0}.$
\end{itemize}

\medskip

\noindent Hence by Lemma~\ref{tsl} we obtain $[[\mathcal{M}_{1,-1},\mathcal{M}_{1,0}],K_{1,0}]\leq K_{3,0}$ and $[[\mathcal{M}_{1,-1},\mathcal{M}_{1,0}],K_{0,1}]\leq K_{2,0}$. Therefore, $[\mathcal{M}_{1,-1}, \mathcal{M}_{1,0}]\subset \mathcal{M}_{2,-1}$.
\end{proof}
}

\begin{remark} In \cite{hain}, Hain considered a certain completion $\widehat{\mathcal{M}}$ of the mapping class group (the definition of this completion depends on the symplectic representation $\sigma:\mathcal{M}\rightarrow \mathrm{Sp}(H,\omega)$)  together with a filtration $(W_k \widehat{\mathcal{M}})_{k\leq 0}$ called the \emph{weight filtration}. This filtration induces a filtration   $(W_k \mathfrak{m})_{k\leq 0}$ of the Lie algebra $\mathfrak{m}$ associated to $\widehat{\M}$. Besides, by considering the surface $\Sigma$ as part of the boundary $\partial V$ of the handlebody~$V$, Hain defined a second filtration $(M^V_k\widehat{\mathcal{M}})_{k\leq 0}$ called the \emph{relative weight filtration}. The latter filtration induces a filtration on each graded quotient $\mathrm{Gr}^W_k\mathfrak{m}:= W_k \mathfrak{m}/W_{k-1} \mathfrak{m}$.

We observe the following relations between the above filtrations and some terms of the double Johnson filtration:

\begin{itemize}
\item   $\H =\mathcal{M}_{0, -1}  =   \M \cap  M^V_{0} \widehat{\M}$.

\item  $\mathcal{T} = \mathcal{M}_{1, -1}  =   \M \cap  M^V_{-2} \widehat{\M}$.

\item $\tau_1(\mathcal{M}_{0,1}) \otimes \mathbb{Q}  =  \mathrm{Gr}^{M^V}_0 \mathrm{Gr}^W_{-1} \mathfrak{m}$.

\item $\tau_2(\mathcal{M}_{0,2}) \otimes \mathbb{Q}  \subset  \mathrm{Gr}^{M^V}_0 \mathrm{Gr}^W_{-2} \mathfrak{m}$.
\end{itemize} 

It would be interesting to compare the filtrations $(\M_{k,-1})_{k\geq 0}$ and  $(\M \cap  M^V_{r} \widehat{\M})_{r\leq 0}$ of the handlebody group $\H$ or more generally to compare some terms of the double Johnson filtration with the bigraded terms $\mathrm{Gr}^{M^V}_i \mathrm{Gr}^W_j\mathfrak{m}$.
\end{remark}

\subsection{Examples of elements in some terms of the filtration}\label{sec_4_6}

In \cite[Section 2]{Freed}, Freedman and Scharlemann introduced some elements of the Goeritz group called \emph{eyeglass twists}. These elements can be used to construct examples of elements in some terms of the double Johnson filtrations of~$\mathcal{M}$ {and~$\G$}.

\begin{definition} An \emph{eyeglass} $e=(D_1,D_2,c)$ in $\Sigma$ consists of:
\begin{itemize}
\item Properly embedded disks $D_1$ in $V$ and $D_2$ in $V'$  such that $\partial D_1 \cap \partial D_2=\emptyset$. The disks $D_i$ ($i=1,2$) are called the \emph{lenses} of $e$.
\item  An embedded arc $c$ in $\Sigma$ connecting $\partial D_1$ to $\partial D_2$ with the interior of $c$ disjoint from $\partial D_1 \cup \partial D_2$. The arc $c$ is called the \emph{bridge} of $e$.
\end{itemize}
\end{definition}

See Figure \ref{sec_2_fig_3}~$(a)$ for an example of an eyeglass in~$\Sigma$. From now on we only draw $\partial D_i$ on $\Sigma$, since the disks $D_i$ are understood. 

Let $e=(D_1,D_2,c)$ be an eyeglass in $\Sigma$. Consider a regular neighborhood of $\partial D_1\cup\partial D_2\cup c$ in $\Sigma$. This neighborhood is homeomorphic to a disc with two holes, in particular it has three boundary components. Let $\gamma_i$ (resp.~$\gamma_j$) be a simple closed curve on $\Sigma$  isotopic to $\partial D_1$ (resp.~$D_2$).  Let $\gamma_{ij}$ be a simple closed curve on $\Sigma$ isotopic the third boundary component, see Figure~\ref{sec_2_fig_3}~$(b)$.

\begin{figure}[ht!]
\centering   	
	\includegraphics[scale=0.73]{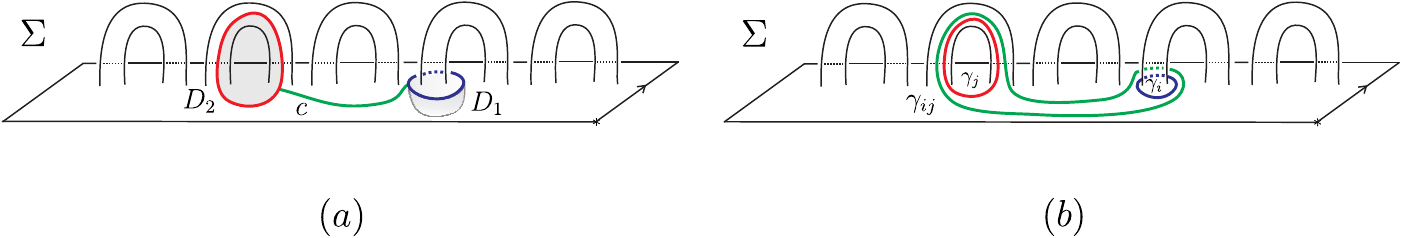}
	\caption{$(a)$ an eyeglass  $e=(D_1,D_2,c)$ in $\Sigma$ \quad  $(b)$ curves $\gamma_{i},\gamma_j$ and $\gamma_{ij}$ determined (up to isotopy) by the eyeglass $e$ shown in $(a)$.}
\label{sec_2_fig_3}
\end{figure}
\begin{definition}
Let $e=(D_1,D_2,c)$ be an eyeglass in $\Sigma$ and let $\gamma_{i},\gamma_j,\gamma_{ij}$ be simple closed curves on~$\Sigma$ determined by~$e$ as above. The \emph{eyeglass twist} on $e$ is the element in the Goeritz group given by the composition of Dehn twists $t_{\gamma_{ij}}t_{\gamma_j}^{-1}t_{\gamma_i}^{-1}$.
\end{definition}

Intuitively, the eyeglass twist on $e=(D_1,D_2,c)$ drags $\partial D_1$ around $\partial D_2$ following $c$, see Figure \ref{sec_2_fig_4}.

\begin{figure}[ht!]
\centering   	
	\includegraphics[scale=0.95]{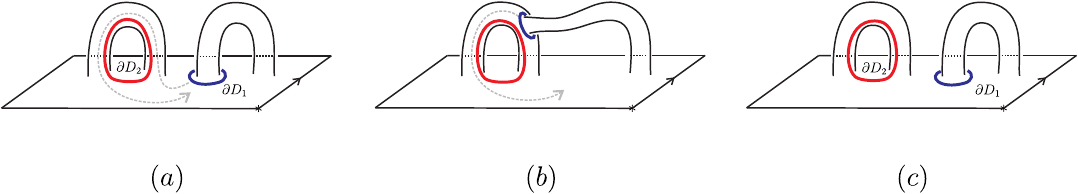}
	\caption{Time-lapse description of the eyeglass twist on $e=(D_1,D_2,c)$ from Figure \ref{sec_2_fig_3}.}
\label{sec_2_fig_4}
\end{figure}

\begin{example}\label{ex_3.17}
Consider the two pairs $e_1,e_2$ and $e_1',e_2'$ of eyeglasses shown in Figure~\ref{sec_2_fig_5}~$(a)$ and~$(b)$, respectively. Let $h_i,h_i'$  ($i=1,2$) be the corresponding  eyeglass twists. Then it can be checked that the commutator $[h_1,h_2]$ belongs to $\mathcal{M}_{1,0}$, and the commutator $[h_1',h_2']$ belongs to $\mathcal{M}_{0,1}$. 
\begin{figure}[ht!]
\centering   	
	\includegraphics[scale=0.73]{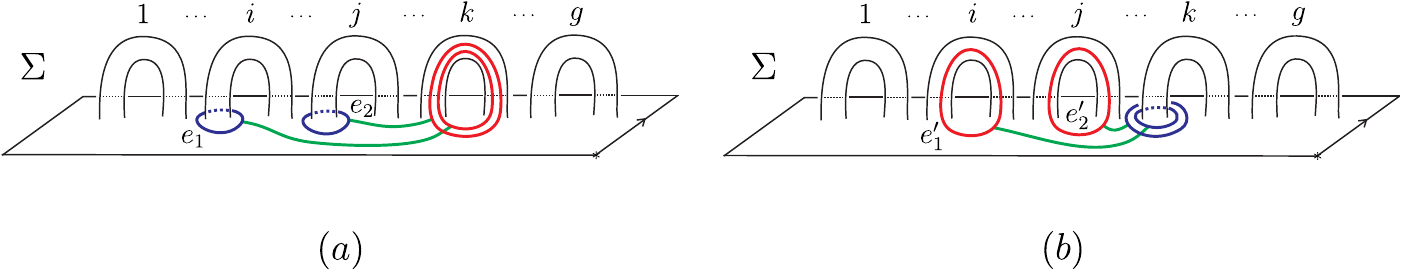}
	\caption{($a$) eyeglasses $e_1,e_2$ \quad  ($b$) eyeglasses $e'_1,e'_2$.}
\label{sec_2_fig_5}
\end{figure} 
\end{example}

\subsection{Images of framed pure {braid} groups into the mapping class group}\label{sec_4_7}

Let us see now how to construct examples of elements in the double Johnson filtration of $\mathcal{M}$ by using embeddings of the framed pure braid group.

Let $l\geq 2$ and let $D_l^{\circ}$ denote a disk with $l$ holes distributed uniformly along the horizontal direction. Recall that the \emph{framed pure braid group on $l$ strands}, denoted $\mathrm{FPB}_l$, can be defined as the group of isotopy classes of homeomorphisms of $D_l^{\circ}$ which are the identity on the boundary. Moreover, any embedding of~$D_l^{\circ}$ into $\Sigma$ induces a homomorphism $\mathrm{FPB}_l\rightarrow \mathcal{M}$. There are interesting applications of some of these embeddings for the study of the mapping class group. For instance, in Figure \ref{sec_2_fig_7} we consider three of such embeddings; the first one was considered by Hatcher and Thurston~\cite{MR579573}, the second one by Oda~\cite{oda,MR1705580}  and the third one by Levine~\cite[Section 2.2]{MR1705580}. These three embeddings satisfy that the image of the holes of $D_l^{\circ}$ bound mutually disjoint disk in $V$.

\begin{figure}[ht!]
\centering   	
	\includegraphics[scale=0.92]{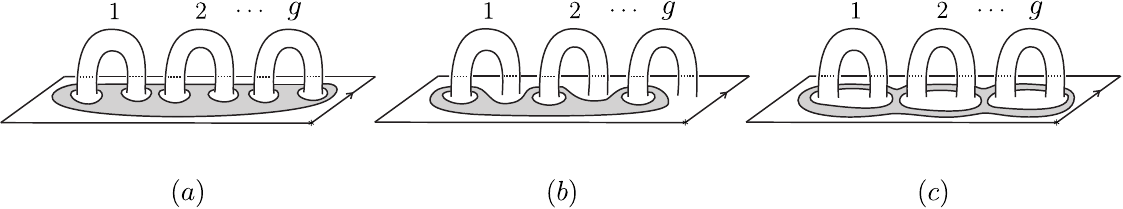}
	\caption{($a$) An embedding $D_{2g}^{\circ}\rightarrow \Sigma$.  \quad  ($b$), ($c$) Embeddings $D_g^{\circ}\rightarrow \Sigma$.}
\label{sec_2_fig_7}
\end{figure}

\begin{proposition}\label{r44} Consider an embedding of  $D_l^{\circ}$ into $\Sigma$ such that the image of the holes of~$D_l^{\circ}$  bound mutually disjoint disks in $V$. Let $f\zzzcolon \mathrm{FPB}_l\rightarrow\mathcal{M}$ be the induced homomorphism. Then we have
    \begin{equation}
f(\Gamma_k\FPB_l)\subset\M_{k,-1}.
    \end{equation}
\end{proposition}

\begin{proof}
If $k=1$, then we clearly have $f(\Gamma_1\FPB)=f(\FPB)\subset\mathcal T=\M_{1,-1}$.
The case $k\ge2$ follows from Proposition \ref{2}.
\end{proof}

A similar result holds to obtain examples of elements in $\mathcal{M}_{-1,k}$ by considering an embedding of~$D_l^{\circ}$ into $\Sigma$ such that the image of the holes of~$D_l^{\circ}$  bound mutually disjoint disks in $V'$.

The description of an eyeglass twist given in Figure \ref{sec_2_fig_4}
shows its similarity with a braid move. More precisely, an eyeglass
twist can be described as the image of the (non-framing) generator of
$\mathrm{FPB}_2$ by a map $\mathrm{FPB}_2\rightarrow \mathcal{M}$
induced by an appropriate embedding of $D_2^{\circ}$ into~$\Sigma$,
see Example \ref{ex_sec2_2}.  The particularity of this kind of
embedding is that one of the holes of $D_2^{\circ}$ bounds a disk in
one of the handlebodies $V$ or $V'$ and the other one bounds a disk in
the other handlebody. This could be generalized as follows.
Fix integers $p,q\geq 1$. Consider {an} embedding of $D_{p+q}^{\circ}$ such that  the first $p$ holes bound disks in one of the handlebodies $V$ or $V'$ and the remaining $q$ holes bound disks in the other handlebody. Such an embedding induces a homomorphism $\mathrm{FPB}_{p,q}\rightarrow \mathcal{M}$ from the so-called \emph{framed mixed pure braid group} $\mathrm{FPB}_{p,q}\simeq \mathrm{FPB}_{p+q}$, see \cite{MR1465028} for the definition and properties of the non-pure version of the groups $\mathrm{FPB}_{p,q}$ which are the interesting ``mixed'' cases.

\begin{example}\label{ex_sec2_2} The eyeglass twist associated to the eyeglass from Figure \ref{sec_2_fig_4} $(a)$ is the image of the (non-framing) generator of $\mathrm{FPB}_{1,1}\simeq\mathrm{FPB}_2$ by the map $\mathrm{FPB}_2\rightarrow \mathcal{M}$ induced by the embedding shown in Figure~\ref{sec_2_fig_6}~$(a)$.
\end{example}
\begin{figure}[ht!]
\centering   	
	\includegraphics[scale=0.8]{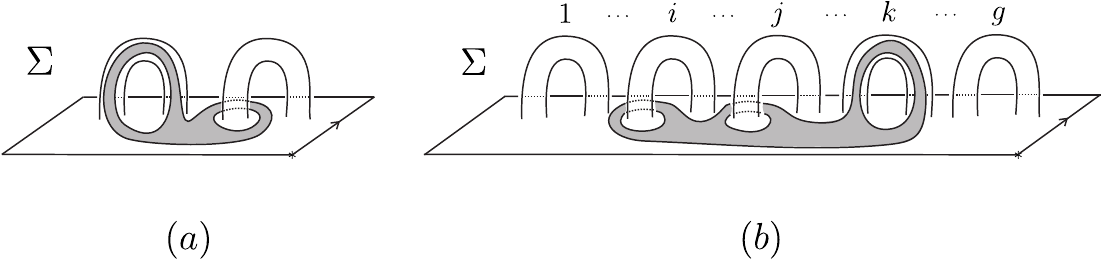}
	\caption{($a$) An embedding $D_2^{\circ}\rightarrow \Sigma$.  \quad  ($b$) An embedding $D_3^{\circ}\rightarrow \Sigma$.}
\label{sec_2_fig_6}
\end{figure} 
\begin{example}\label{ex_mixb}   Let $f\zzzcolon \mathrm{FPB}_{2,1}\simeq\mathrm{FPB}_3\rightarrow \mathcal{M}$ denote the homomorphism induced by  the embedding of~$D_3^{\circ}$ into $\Sigma$ shown in Figure~\ref{sec_2_fig_6}~$(b)$. Let $a_{1,3}$ and $a_{2,3}$ be the two non-framing generators of $\mathrm{FPB}_{2,1}\simeq\mathrm{FPB}_3$. Then $f(a_{1,3})= h_1$ and $f(a_{2,3})=h_2$, where $h_1$ and $h_2$ are the eyeglass twists on the eyeglasses $e_1$ and $e_2$ from Figure~\ref{sec_2_fig_5}~$(a)$. It follows from Example~\ref{ex_3.17} that if $\beta\in\Gamma_2\mathrm{FPB}_3$, then $f(\beta)\in\mathcal{M}_{1,0}$. 
\end{example}

\begin{remark} The above discussion suggests a possible method to construct examples in a general term $\mathcal{M}_{m,n}$ of the double Johnson filtration of the mapping class group by considering the images of the elements in the lower central series of $\mathrm{FPB}_{p,q}$ under homomorphisms of the type $\mathrm{FPB}_{p,q}\rightarrow \mathcal{M}$. It also seems possible to define a double filtration of the group $\mathrm{FPB}_{p,q}$, which would be interesting on its own, using the fundamental groupoid of $D^{\circ}_{p+q}$ and to relate it with the double Johnson filtration of the mapping class group via homomorphisms of the type  $\mathrm{FPB}_{p,q}\rightarrow \mathcal{M}$.
\end{remark}

\subsection{Relation with the usual Johnson filtration}\label{sec_3_4}

By definition we have that $\mathcal{M}_{2,-1}$, $\mathcal{M}_{1,0}$, $\mathcal{M}_{0,1}$ and~$\mathcal{M}_{-1,2}$ are subgroups of the Torelli group $\mathcal{I}$. Moreover,  we have the following.

\begin{proposition}\label{subgr} The product $\mathcal{M}_{2,-1}\cdot \mathcal{M}_{1,0}\cdot \mathcal{M}_{0,1}\cdot \mathcal{M}_{-1,2}$ is a subgroup of $\mathcal{I}$. That is, 
 $$\left\langle \mathcal{M}_{2,-1}\cdot \mathcal{M}_{1,0}\cdot \mathcal{M}_{0,1}\cdot \mathcal{M}_{-1,2} \right\rangle = \mathcal{M}_{2,-1}\cdot \mathcal{M}_{1,0}\cdot \mathcal{M}_{0,1}\cdot \mathcal{M}_{-1,2}.$$
\end{proposition}

\begin{proof}
It is enough to show that an element of $\langle \mathcal{M}_{2,-1}\cdot \mathcal{M}_{1,0}\cdot \mathcal{M}_{0,1}\cdot \mathcal{M}_{-1,2} \rangle$ of the form $abcdxyzw$ with $a,x\in \mathcal{M}_{2,-1} $, $b,y\in \mathcal{M}_{1,0}$, $c,z\in \mathcal{M}_{0,1} $ and $d,w\in \mathcal{M}_{-1,2}$ belongs to  $\mathcal{M}_{2,-1}\cdot \mathcal{M}_{1,0}\cdot \mathcal{M}_{0,1}\cdot \mathcal{M}_{-1,2}$.

By Propositions \ref{prop_2.10} and \ref{sec_2_prop_711} we have:
\begin{equation}\label{equ_1}
\!^{\mathcal{M}_{0,0}}\mathcal{M}_{i,j}\subset \mathcal{M}_{i,j} \text{ for } i,j\geq -1
\end{equation} 
and
$$[\mathcal{M}_{2,-1},\mathcal{M}_{-1,2}]\subset \mathcal{M}_{1,1}.$$ 
Hence 
\begin{equation}\label{equ_2}
\!^{\mathcal{M}_{2,-1}}\mathcal{M}_{-1,2}\subset \mathcal{M}_{1,1}\cdot \mathcal{M}_{-1,2}.
\end{equation}

Thus, using \eqref{equ_1} and \eqref{equ_2} we obtain
\begin{equation*}
\begin{split}
abcdxyzw &= abcx(x^{-1}dx)yzw\\
 & = abcxpqyzw \text{\ \ \ \ \ \  (we write } x^{-1}dx =pq\in\mathcal{M}_{1,1} \cdot\mathcal{M}_{-1,2}\text{)}\\
  & = \big(a(bcx(bc)^{-1})\big)\cdot \big(bpy\big)\cdot \big(((py)^{-1}cpy)z\big)\cdot\big(((yz)^{-1}qyz)w\big) \in \mathcal{M}_{2,-1}\cdot \mathcal{M}_{1,0}\cdot \mathcal{M}_{0,1}\cdot \mathcal{M}_{-1,2}.
\end{split}
\end{equation*}
\end{proof}

\begin{remark} In fact we expect  $\mathcal{M}_{2,-1}\cdot \mathcal{M}_{1,0}\cdot \mathcal{M}_{0,1}\cdot \mathcal{M}_{-1,2} =\mathcal{I}$, see Conjecture~\ref{conjecture_1}.
\end{remark}

Let us recall the Johnson filtration of the mapping class group and its respective Johnson homomorphisms.

Consider the $\modN$-filtration of $K=\pi_1(\Sigma,*)$ induced by the lower central series $(\Gamma_n K)_{n\geq 1}$, i.e., {$K_0=K_1= K$} and $K_n= \Gamma_n K$ for all $n\geq 2$. Then the Johnson filtration of $\mathcal{M}$ is the decreasing sequence of subgroups 
\begin{equation*}\label{JTFequ1}
\mathcal{M}\supset\mathcal{I}=J_1\mathcal{M} \supset J_2\mathcal{M} \supset J_3\mathcal{M} \supset \cdots
\end{equation*}
given by
\begin{equation*}\label{JTFequ2}
J_n\mathcal{M}=\{h\in\mathcal{M}\ | \   [h,K]{\subset}\Gamma_{n+1} K \}
\end{equation*}
for every $n\geq 1$. The group $J_2\mathcal{M}$ is known as the \emph{Johnson subgroup} and it is denoted by~$\mathcal{K}$.

\begin{proposition} We have $\mathcal{M}_{m,n}\subset J_{m+n}\mathcal{M}$ for all $m,n\geq -1$ with $m+n\geq 1$.
\end{proposition}
\begin{proof}
If $h\in\mathcal{M}_{m,n}$, then $[h,K_{1,0}]\subset K_{m+1,n}\subset \Gamma_{m+n+1} K$ and $[h,K_{0,1}]\subset K_{m,n+1}\subset \Gamma_{m+n+1} K$. By \eqref{e16}, we have $[h,K]\subset \Gamma_{m+n+1} K$. Hence $h\in J_{m+n} \mathcal{M}$.
\end{proof}

 The $\modN_+$-graded Lie algebra associated to the lower central series is 
\begin{equation}
\bK = \bigoplus_{n\geq 1}\frac{\Gamma_n K}{\Gamma_{n+1} K}\simeq \bigoplus_{n\geq 1}\mathfrak{Lie}_n(H)=\mathfrak{Lie}(H),
\end{equation}
where $\mathfrak{Lie}(H)$ is the $\modN_+$-graded  Lie algebra freely generated by $H=H_1(\Sigma;\mathbb{Z})$ in degree~$1$. Notice that the abelian group $\mathrm{Der}_n(\mathfrak{Lie}(H))$ of degree $n$ derivations of $\mathfrak{Lie}(H)$ is in bijection with $\mathrm{Hom}(H,\mathfrak{Lie}_n(H))$. Hence, the $n$-th Johnson homomorphism can be described as the homomorphism
\begin{equation}\label{e31_b} 
\tau_n\zzzcolon  J_n\mathcal{M}\longrightarrow \text{Hom}(H,\Gamma_{n+1} K/\Gamma_{n+2} K)\simeq H^*\otimes \mathfrak{Lie}_{n+1}(H) \simeq H\otimes \mathfrak{Lie}_{n+1}(H),
\end{equation}
\noindent which sends   $h\in J_n\mathcal{M}$ to the map
$x\mapsto [h,\tilde x]\Gamma_{n+2} K$
 for all $x\in H=K/\Gamma_2 K$, where $\tilde{x}\in K$ is any lift of $x$. The second isomorphism in  \eqref{e31_b} is given by  the identification $H\stackrel{\sim}{\longrightarrow} H^*$ that maps $x$ to $\omega(x, \cdot)$, where $\omega\zzzcolon H\otimes H\to \mathbb{Z}$ is the intersection form of $\Sigma$. Moreover, by using the system of meridians and parallels $\{x_i,y_i\}_{1\leq i\leq g}$ and the induced symplectic basis $\{a_i,b_i\}_{1\leq i\leq g}$ of $H$ as in Section~\ref{sec_3_1}, for $h\in J_n\mathcal{M}$ we have
\begin{equation}\label{eq_tau_b}
\tau_n(h)  = \sum_{i=1}^g a_i\otimes [h,y_i]\Gamma_{n+2}K - \sum_{i=1}^g b_i\otimes [h,x_i]\Gamma_{n+2}K.
\end{equation}

These homomorphisms were considered by Johnson~\cite{MR579103,MR718141} and extensively studied by Morita~\cite{MR1133875,MR1224104}. In particular, Morita proved that the $n$-th Johnson homomorphism takes values in the kernel $D_n(H)$ of the Lie bracket  $\left[\ ,\ \right]\zzzcolon H\otimes\mathfrak{Lie}_{n+1}(H)\rightarrow\mathfrak{Lie}_{n+2}(H)$~ \cite[Corollary 3.2]{MR1224104}. 

The abelian group $D_1(H)\leq H\otimes \mathfrak{Lie}_2(H)=H\otimes \ext^2H$ can be identified with $\ext^3H$ via the map $\ext^3H\to H\otimes\ext^2 H$ given by 
\begin{equation}\label{eq_D_1}
x\wedge y\wedge z\longmapsto x\otimes y\wedge z + y\otimes z\wedge x + z\otimes x\wedge y.
\end{equation}
Hence, we can see the first Johnson homomorphism as a group homomorphism
\begin{equation}\label{eq_tau_1}
\tau_1\zzzcolon  \mathcal{I}\longrightarrow \ext^3 H,
\end{equation}
such that $\mathrm{ker}(\tau_1)=J_2\mathcal{M}=\mathcal{K}$. It is well known that the homomorphism \eqref{eq_tau_1} is surjective \cite[Theorem~1]{MR579103}.

{Recall that we have two subgroups $A$ and $B$ of $H$ such that $H=A\oplus B$, see \eqref{equ_A_B}. This decomposition of $H$} gives the following decomposition of $\ext^3 H$:
\begin{equation}\label{eq:1}
\ext^3 H = \ext^3 A \oplus (\ext^2 A\otimes B)\oplus (A\otimes \ext^2B)\oplus \ext^3 B.
\end{equation}

\begin{proposition}\label{r45} The first Johnson homomorphism $\tau_1\zzzcolon \mathcal{I}\to \ext^3 H$ induces, by restriction, the following surjective homomorphisms:
\begin{itemize}
\item $\tau_1{|_{\mathcal{M}_{2,-1}}}\zzzcolon \mathcal{M}_{2,-1}\longrightarrow \ext^3 A$,
\item $\tau_1|_{\mathcal{M}_{1,0}}\zzzcolon \mathcal{M}_{1,0}\longrightarrow \ext^2 A\otimes B$,
\item $\tau_1|_{\mathcal{M}_{0,1}}\zzzcolon  \mathcal{M}_{0,1}\longrightarrow  A\otimes \ext^2 B$,
\item $\tau_1|_{\mathcal{M}_{-1,2}}\zzzcolon  \mathcal{M}_{-1,2}\longrightarrow \ext^3B$.
\end{itemize}
\end{proposition}
\begin{proof}
The fact that each one of the restrictions of $\tau_1$ arrives to the good target subgroup of $\ext^3 H$ follows from the explicit formula \eqref{eq_tau_b}. We need only to check the surjectivity of each one of the restrictions. By using the symplectic basis $\{a_i,b_i\}_{1\leq i\leq g}$ of $H$ and the correspondence \eqref{eq_D_1}, we have that the induced basis of $\ext^3 A$,   $\ext^2 A\otimes B$, $A\otimes \ext^2 B$ and $\ext^3 B$ are
$\{a_i\wedge a_j\wedge a_k \ |\ 1\leq i<j<k\leq g\}$,  $\{b_k\wedge a_i\wedge a_j \ |\ 1\leq k\leq g,  \ 1\leq i<j\leq g\}$, $\{a_i\wedge b_j\wedge b_k \ |\ 1\leq i\leq g,\  1\leq j<k\leq g\}$ and $\{b_i\wedge b_j\wedge b_k \ |\ 1\leq i<j<k\leq g\}$, respectively. We proceed to explicitly realize these basis.
\begin{figure}[ht!]
\centering   	
	\includegraphics[scale=0.73]{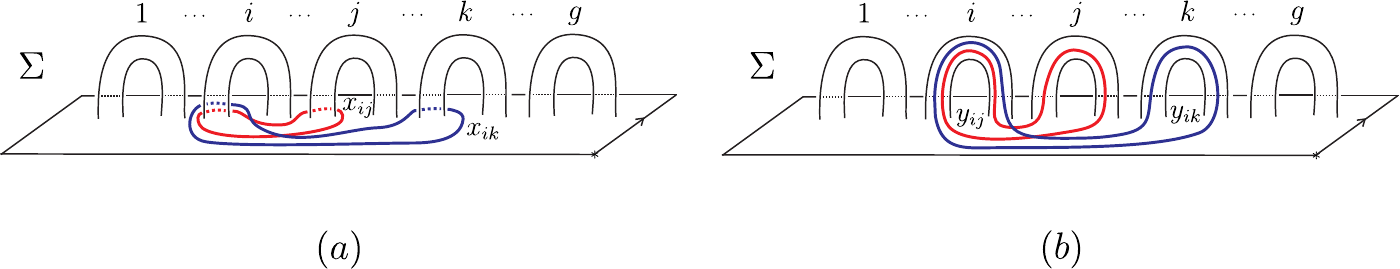}
	\caption{($a$) curves $x_{ij}$ and $x_{ik}$ \quad  ($b$) curves $y_{ij}$ and $y_{ik}$.}
\label{sec_2_fig_8}
\end{figure} 
Let $1\leq i<j<k\leq g$ and consider the curves $x_{ij}$ and $x_{ik}$ as in Figure~\ref{sec_2_fig_8}~$(a)$. Let $t_{x_{ij}}$ and $t_{x_{ik}}$ {be} their associated Dehn twists. We have  $[t_{x_{ik}}, t_{x_{ij}}]\in\mathcal{M}_{2,-1}$ and
$$\tau_1([t_{x_{ik}}, t_{x_{ij}}])= {a_i\otimes [x_j^{-1}, x_k^{-1}]\Gamma_3 K + a_j\otimes [x_k,  x_i]\Gamma_3 K + a_k\otimes [x_j,x_i^{-1}]\Gamma_3 K,}$$
which corresponds to $a_i\wedge a_j \wedge a_k$ under \eqref{eq_D_1}. Hence $\tau_1|_{\mathcal{M}_{2,-1}}\zzzcolon \mathcal{M}_{2,-1}\rightarrow \ext^3 A$ is surjective. Similarly, by using the Dehn twists about the curves $y_{ij}$ and $y_{ik}$, we can show that $\tau_1|_{\mathcal{M}_{-1,2}}\zzzcolon  \mathcal{M}_{-1,2}\rightarrow \ext^3 B$ is surjective.

Let $1\leq i<j\leq g$ and consider the element $h_{ij}\in\mathcal{M}_{1,0}$ defined in Example~\ref{ex_m10}. By a direct computation, we obtain
$$\tau_1(h_{ij}) = a_i\otimes [x_j, y_i]\Gamma_3 K + a_j\otimes [y_i^{-1}, x_i^{-1}]\Gamma_3 K - b_i\otimes [x_j, x_i]\Gamma_3 K,$$
which, under~\eqref{eq_D_1}, corresponds to $b_i\wedge a_i\wedge a_j$. Let $1\leq k\leq g$ with $k\not=i$ and let $\phi_{ki}\in \mathcal{G}$ be the eyeglass twist associated to an eyeglass $(D_1,D_2,c)$ as shown in Figure~\ref{sec_2_fig_3} but such that $\partial D_2$ is isotopic to the parallel $y_k$ and $\partial D_1$ is isotopic to the meridian $x_i$. The action in homology of $(\phi_{ki})^{-1}$ on $b_i$ is $b_i\mapsto b_i+b_k$ and it is trivial on $a_i$ and $a_j$. By the equivariance property of $\tau_1$, we have $\tau_1(({ }^{(\phi_{ki})^{-1}}h_{ij})h_{ij}^{-1})=b_k\wedge a_i\wedge a_j$ and ${ }^{(\phi_{ki})^{-1}}h_{ij}\in\mathcal{G}=\mathcal{M}_{0,0}$ by Proposition~\ref{prop_2_6}. Therefore $\tau_1|_{\mathcal{M}_{1,0}}\zzzcolon \mathcal{M}_{1,0}\rightarrow \ext^2 A\otimes B$ is surjective. A similar argument shows that $\tau_1|_{\mathcal{M}_{0,1}}\zzzcolon  \mathcal{M}_{0,1}\rightarrow  A\otimes \ext^2 B$ is surjective.
\end{proof}

\begin{theorem}\label{r46} We have
$$\mathcal{I} = \mathcal{M}_{2,-1}\cdot \mathcal{M}_{1,0}\cdot \mathcal{M}_{0,1}\cdot \mathcal{M}_{-1,2}\cdot \mathcal{K}.$$
\end{theorem}

\begin{proof}
Let $h\in\mathcal{I}$ and write $\tau_1(h)= p+q+r+s$ with $p\in\ext^3A$, $q\in \ext^2 A\otimes B$, $r\in A\otimes \ext^2 B$ and $s\in\ext^3 B$. By Proposition~\ref{r45} there exist $h_1\in\mathcal{M}_{2,-1}$, $h_2\in\mathcal{M}_{1,0}$, $h_3\in\mathcal{M}_{0,1}$ and $h_4\in\mathcal{M}_{-1,2}$ such that $\tau_1(h_1)=p$, $\tau_1(h_2)=q$, $\tau_1(h_3)=r$ and $\tau_1(h_4)=s$. Hence $\tau_1((h_1h_2h_3h_4)^{-1}h)=0$, that is, $(h_1h_2h_3h_4)^{-1}h\in \mathcal{K}$ which finishes the proof.  
\end{proof}

\begin{conjecture}\label{conjecture_1} We have 
$\mathcal{K}\subset \mathcal{M}_{2,-1}\cdot \mathcal{M}_{1,0}\cdot \mathcal{M}_{0,1}\cdot \mathcal{M}_{-1,2}$ and therefore
$$\mathcal{I}=\mathcal{M}_{2,-1}\cdot \mathcal{M}_{1,0}\cdot \mathcal{M}_{0,1}\cdot \mathcal{M}_{-1,2}.$$
\end{conjecture}

\begin{question}
Do we have 
$$ J_n\mathcal{M} =\prod_{i+j=n,\  i,j\geq -1} \mathcal{M}_{ij}$$
for all $n\geq 1$?
\end{question}

We can state a weaker version of the previous question.

\begin{question}
Do we have 
$$ J_n\mathcal{M} =\left(\prod_{i+j=n,\  i,j\geq -1} \mathcal{M}_{ij}\right)\cdot J_{n+1}\mathcal{M}$$
for all $n\geq 1$?
\end{question}

\begin{corollary}\label{cor_L}
We have $$\mathcal{L} = \H\cdot\M_{-1,2}\cdot \mathcal{K} = \M_{1,-1}\cdot\M_{0,0}\cdot\M_{-1,2}\cdot\mathcal{K}.$$
\end{corollary}

\begin{proof}
By Lemma~\ref{im_symp}~(i) we have $\mathcal{L}=\mathcal{H}\cdot\mathcal{I}$. Combining  Theorem~\ref{r46} and Proposition~\ref{prop_leading}~(i) we obtain the desired result.
\end{proof}

\begin{lemma}\label{lemma_2} We have the following
\begin{itemize}
\item[\emph{(i)}] $\mathcal{M}_{0,0}\cap \mathcal{M}_{2,-1} \subset \mathcal{M}_{1,0}$ \ \ \ and \ \ \ $\mathcal{M}_{0,0}\cap \mathcal{M}_{-1,2} \subset \mathcal{M}_{0,1}$.
\item[\emph{(ii)}] $(\mathcal{M}_{0,0}\cdot \mathcal{M}_{-1,2})\cap \mathcal{M}_{2,-1}\subset \mathcal{M}_{0,0}$ \ \ \ and \ \ \ $(\mathcal{M}_{0,0}\cdot \mathcal{M}_{2,-1})\cap \mathcal{M}_{-1,2}\subset \mathcal{M}_{0,0}$.
\end{itemize}
\end{lemma}
\begin{proof} 
(i) If $h\in\mathcal{M}_{0,0}\cap \mathcal{M}_{2,-1}$, then $[h,K_{1,0}]\subset K_{3,0}\subset K_{{2},0}$ and $$[h,K_{0,1}]\subset K_{2,0}\cap K_{0,1}\subset \bar{X}\cap \bar{Y} =[\bar{X},\bar{Y}]=K_{1,1}.$$ Thus $h\in \mathcal{M}_{1,0}$.  Similarly, if $h\in\mathcal{M}_{0,0}\cap \mathcal{M}_{-1,2}$, then $h\in \mathcal{M}_{0,1}$.

(ii) Let $fh\in(\mathcal{M}_{0,0}\cdot \mathcal{M}_{-1,2})\cap \mathcal{M}_{2,-1}$ with $f\in\mathcal{M}_{0,0}$ and $h\in\mathcal{M}_{-1,2}$. Then $$[fh, K_{1,0}]\subset K_{3,0}\subset K_{1,0}$$ and $[fh, K_{0,1}]\subset K_{0,1}$, that is, $fh\in\mathcal{M}_{0,0}$. Similarly $(\mathcal{M}_{0,0}\cdot \mathcal{M}_{2,-1})\cap \mathcal{M}_{-1,2}\subset \mathcal{M}_{0,0}$.
\end{proof}

\begin{proposition}\label{lemma_3}
$(\mathcal{M}_{2,-1}\cdot \mathcal{M}_{1,0}\cdot \mathcal{M}_{0,1}\cdot \mathcal{M}_{-1,2})\cap \mathcal{M}_{0,0}= \mathcal{M}_{1,0}\cdot \mathcal{M}_{0,1}$.
\end{proposition}

\begin{proof}
Let $w\in (\mathcal{M}_{2,-1}\cdot \mathcal{M}_{1,0}\cdot \mathcal{M}_{0,1}\cdot \mathcal{M}_{-1,2})\cap \mathcal{M}_{0,0}$ and write $w=abcd$ with $a\in\mathcal{M}_{2,-1}$, $b\in\mathcal{M}_{1,0}$, $c\in\mathcal{M}_{0,1}$ and $d\in\mathcal{M}_{-1,2}$. Hence $a=wd^{-1}c^{-1}b^{-1}\in \mathcal{M}_{0,0}\cdot \mathcal{M}_{-1,2}$. Thus, by Lemma~\ref{lemma_2}~{(ii)} $a\in \mathcal{M}_{0,0}$. It follows from Lemma~\ref{lemma_2}~{(i)} that $a\in \mathcal{M}_{1,0}$. It can be shown similarly that $d\in \mathcal{M}_{0,1}$. Therefore, $abcd=(ab)(cd)\in\mathcal{M}_{1,0}\cdot \mathcal{M}_{0,1}$.
\end{proof}

\begin{conjecture}\label{conjecture_3} We have 
$$\overline{\mathcal{M}}_{0,0}:=\frac{\mathcal{M}_{0,0}}{\mathcal{M}_{1,0}\cdot \mathcal{M}_{0,1}}\simeq \left\{\left( \begin{smallmatrix} P &0\\ 0& (P^t)^{-1}\end{smallmatrix} \right) \  \bigg\rvert \ \ P\in\mathrm{GL}(g,\mathbb{Z})  \right\} \simeq \mathrm{GL}(g,\mathbb{Z}).$$  
\end{conjecture}

\begin{remark} Notice that by Lemma~\ref{im_symp} we have a surjective map from $\overline{\mathcal{M}}_{0,0}$ to $\mathrm{GL}(g,\mathbb{Z})$ which we expect to be injective. In fact, Conjecture \ref{conjecture_3} follows from Conjecture~\ref{conjecture_1} because in this case Proposition~\ref{lemma_3} becomes $\mathcal{I}\cap \mathcal{M}_{0,0}=\mathcal{M}_{1,0}\cdot \mathcal{M}_{0,1}$. Hence  the short exact sequence \eqref{ses1}  could be written as
\begin{equation}\label{ses_goe_2}
1\longrightarrow \mathcal{M}_{1,0}\cdot \mathcal{M}_{0,1}\longrightarrow \mathcal{M}_{0,0}\xrightarrow{\ \sigma\ }\mathrm{GL}(g,\mathbb{Z})\longrightarrow 1.
\end{equation}
and we would obtain the desired result.
\end{remark}

\subsection{The case of automorphisms of free groups}\label{sec_5}

Several results from Section~\ref{sec_three} can be developed in the context of the automorphism group of a free group. In this subsection we develop some of these results.

Let $(K;\bX,\bY)$ be as in Section~\ref{sec_dlcsfg}. Let $(K_{m,n})_{(m,n)\in\modN^2}$ be the double lower central series of $(K;\bX,\bY)$. Extend the definition of $K_{m,n}$ for every $m,n\in\mathbb{Z}$ as in \eqref{dlex}.

Let $\A=\mathrm{Aut}(K)$ be the automorphism group of $K$. As usual, we regard $K$ and $\A$ as subgroups of the semidirect product $K\rtimes {\A}$. 
Set 
  \begin{equation}
    \A_{0,0}=\{h\in\A\mid h(\bX)=\bX,h(\bY)=\bY\},
  \end{equation}
  which we call the \emph{fake Goeritz group} of type $(p,q)$.
  We have the double Johnson filtration $(\A_{m,n})_{(m,n)\in\modN^2}$ of $\A_{0,0}$ defined by
  \begin{equation}
    \A_{m,n}=\{h\in\A_{0,0}\mid [h,\bX]\subset K_{m+1,n},\;[h,\bY]\subset K_{m,n+1}\}
  \end{equation}
  
  Similarly to Section \ref{sec_31}, we can extend $(\A_{m,n})_{(m,n)\in\modN^2}$ to $(\A_{m,n})_{m,n\ge-1}$ by setting
\begin{equation}
\begin{split}
  \A_{m,n} =&\big\{h\in\A\mid[h^{\pm1}, K_{i,j}]\subset K_{m+i,n+j} \text{ for every } (i,j)\in\modN^2\big\}\\
  =&\big\{h\in\A\mid[h^{\pm1}, K_{1,0}]\subset K_{m+1,n},\;[h^{\pm1}, K_{0,1}]\subset K_{m,n+1}\big\}.
\end{split}
\end{equation}
The family $(\A_{m,n})_{m,n\ge-1}$ is called the \emph{double
Andreadakis-Johnson filtration} of $\A=\Aut(K)$ with respect to $(\bX,\bY)$. 
It satisfies properties similar to the double Johnson filtration of the mapping class group that are stated in Sections~\ref{sec_31} and~\ref{sec_4_5}.

Suppose that $K= \la x_1,\ldots, x_g, y_1,\ldots, y_g\ra$.  Recall that the Dehn-Nielsen  representation  $\rho:\mathcal{M}\rightarrow \Aut(\pi_1(\Sigma,*))\simeq \A$ which sends $h\in\M$ to the induced map $h_{\#}$ on $\pi_1(\Sigma,*)\simeq K$, is injective. The   double Andreadakis-Johnson filtration of $\A$  and the double Johnson filtration of $\M$ are compatible under the Dehn-Nielsen map, that is, for $m,n\geq -1$ we have $\rho(\M_{m,n})\subset \A_{m,n}$. In particular, we can construct examples in some terms of the family $(\A_{m,n})_{m,n\ge-1}$ by using those constructed for $(\M_{m,n})_{m,n\ge-1}$ in Sections~\ref{sec_4_6} and~\ref{sec_4_7}.

Let us go back to the general case $K= \la x_1,\ldots, x_p, y_1,\ldots, y_q\ra$. We here recall the usual Andreadakis-Johnson filtration of $\A$, {see \cite{satoh} for a survey}. For each $m\geq 0$, the action of~$\A$ on the nilpotent quotient $K/\Gamma_{m+1}K$ induces a homomorphism 
$$\rho_m\zzzcolon \A\longrightarrow\mathrm{Aut}(K/\Gamma_{m+1}K).$$
Set
$$\mathcal{IA}_m=\mathrm{ker}(\rho_m)=\{h\in \A \ | \ [h,K]\subset\Gamma_{m+1} K\}.$$

The family $(\mathcal{IA}_m)_{m\geq 1}$ is called the \emph{Andreadakis-Johnson} filtration of~$\A$. In particular, $\mathcal{IA}_1=\mathcal{IA}$ is the \emph{$\mathrm{IA}$-automorphism group of $K$}. By definition for $(m,n)> (-1,-1)$  with $m+n\geq 1$, we have $\A_{m,n}\subset \mathcal{IA}_{m+n}$. 

{It is well known that $\mathcal{IA}$ is finitely generated by an explicit set of automorphisms, see \cite[Theorem~4.1]{satoh} for the list of generators  and see \cite[Theorem~5.6]{bbm} for a modern proof. {We just need to state  such a set of automorphisms by  considering  the partition of the generators of $K=\la x_1,\ldots, x_p, y_1,\ldots,y_q\ra$.}

\begin{theorem}[Magnus \cite{mag}]\label{Magnus_thm}
Let $p+q\geq 3$ and $a,b,c\in \{x_1,\ldots, x_p,y_1,\ldots,y_q\}$.  Let $\varphi_{ab}$ be the automorphism of $K$ defined by $\varphi(a)=b^{-1}ab$ and $\varphi(z)=z$ for $z\in\{x_1,\ldots, x_p,y_1,\ldots,y_q\}$ with $z\not=a$. Let $\varphi_{abc}$ be the automorphism of $K$ defined by $\varphi(a)=a[b,c]$ and $\varphi(z)=z$ for $z\in\{x_1,\ldots, x_p,y_1,\ldots,y_q\}$ with $z\not=a$. Then the $\mathrm{IA}$-automorphism group $\mathcal{IA}$ of $K$ is finitely generated by 
    \begin{enumerate}
        \item $\varphi_{x_ix_j}$ with $1\leq i,j\leq p$ and $i\not = j$.
        \item $\varphi_{y_iy_j}$ with $1\leq i,j\leq q$ and $i\not = j$.
        \item $\varphi_{x_iy_j}$ with $1\leq i\leq p$ and $1\leq j\leq q$.
        \item $\varphi_{y_ix_j}$ with $1\leq i\leq q$ and $1\leq j\leq p$.
        \item $\varphi_{x_ix_jx_k}$ with $1\leq i,j,k\leq p$ distinct and $j<k$.
        \item $\varphi_{x_ix_jy_k}$ with $1\leq i,j\leq p$ distinct and $1\leq~k\leq q$.
        \item $\varphi_{x_iy_jy_k}$ with $1\leq i\leq p$ and $1\leq~j<k\leq q$.
        \item $\varphi_{y_iy_jy_k}$ with $1\leq i,j,k\leq q$ distinct and $j<k$.
        \item $\varphi_{y_ix_jy_k}$ with $1\leq i,k\leq q$ distinct and $1\leq~j\leq p$.
        \item $\varphi_{y_ix_jx_k}$ with $1\leq i\leq q$ and $1\leq~j<k\leq p$.
    \end{enumerate}
\end{theorem}

Mirroring Proposition~\ref{subgr} we also have that $\mathcal{A}_{2,-1}\cdot\mathcal{A}_{1,0}\cdot\mathcal{A}_{0,1}\cdot\mathcal{A}_{-1,2}$ is a subgroup of~$\mathcal{IA}$. We can now prove the analogue of Theorem~\ref{r46} in this context. 

\begin{theorem}\label{thm_free} We have
$$\mathcal{IA} = \mathcal{A}_{2,-1}\cdot\mathcal{A}_{1,0}\cdot\mathcal{A}_{0,1}\cdot\mathcal{A}_{-1,2}.$$
\end{theorem}
\begin{proof}
Clearly, $\mathcal{A}_{2,-1}\cdot\mathcal{A}_{1,0}\cdot\mathcal{A}_{0,1}\cdot\mathcal{A}_{-1,2}\subset \mathcal{IA}$. The other inclusion follows by using Magnus generators. {The} generators described in Theorem \ref{Magnus_thm} {\it{(1)}},{\it{(4)}}, {\it{(5)}} and {\it{(9)}} belong to $\mathcal{A}_{1,0}$. The generators described in {\it{(2)}}, {\it{(3)}}, {\it{(6)}} and {\it{(8)}} belong to $\mathcal{A}_{0,1}$. The generators described in {\it{(7)}} belong to $\mathcal{A}_{-1,2}$. Finally, the generators described in {\it{(10)}} belong to $\mathcal{A}_{2,-1}$.  
\end{proof}

{
We can identify  $\Aut(K/\Gamma_2K)$ with the general linear group $\mathrm{GL}(p+q, \Z)$. Besides, it is well known that the map $\rho_1\zzzcolon \A\rightarrow \Aut(K/\Gamma_2K)\simeq \mathrm{GL}(p+q, \Z)$ is surjective, see for instance \cite[Chapter~I: Proposition~4.4]{lyndon}. Hence we have the following short exact sequence
\begin{equation}
1\longrightarrow \mathcal{IA}\longrightarrow
\A \xrightarrow{\ \rho_1 \ } \mathrm{GL}(p+q,\Z)\longrightarrow 1.
\end{equation}

Moreover we have
$$\rho_1(\A_{0,0}) = \left\{\left( \begin{matrix} R &0\\ 
0& S\end{matrix} \right) \  \bigg\rvert \ \ R\in\mathrm{GL}(p,\Z), \   S\in\mathrm{GL}(q,\Z)\right\}\simeq \mathrm{GL}(p,\Z)\times \mathrm{GL}(q,\Z),$$
Thus we have the short exact sequence 
\begin{equation}\label{ses_free}
1\longrightarrow \mathcal{IA}\cap\A_{0,0}\longrightarrow \A_{0,0}
\xrightarrow{\ \rho_1\ } \mathrm{GL}(p,\Z)\times\mathrm{GL}(q,\Z)\longrightarrow 1.
\end{equation}

\begin{corollary}\label{cor_free}
 We have 
$$\overline{\A}_{0,0}:=\frac{\A_{0,0}}{\A_{1,0}\cdot A_{0,1}}\simeq  \mathrm{GL}(p,\Z)\times \mathrm{GL}(q,\Z).$$  
\end{corollary}
\begin{proof}
The result follows {from} Theorem~\ref{thm_free}, Lemma~\ref{lemma_3_free} below and \eqref{ses_free}.
\end{proof}
Compare Corollary~\ref{cor_free} with Conjecture~\ref{conjecture_3}.

\begin{lemma}\label{lemma_3_free}
$(\mathcal{A}_{2,-1}\cdot \mathcal{A}_{1,0}\cdot \mathcal{A}_{0,1}\cdot \mathcal{A}_{-1,2})\cap \mathcal{A}_{0,0}= \mathcal{A}_{1,0}\cdot \mathcal{A}_{0,1}$.
\end{lemma}

\begin{proof}
The proof is the same as the proof of Proposition~\ref{lemma_3}.
\end{proof}

}

\section{Double Johnson homomorphisms for the mapping class group}\label{sec_4}

In this section we apply the general theory from Sections \ref{sec_1} and \ref{sec_2} to obtain Johnson homomorphisms for the $\mathbb{N}^2$-filtration of $\mathcal{G}$. Then, following Morita's result for the usual Johnson homomorphisms, we refine their target groups. Finally we extend this family of homomorphisms to obtain a family of Johnson homomorphisms for the double Johnson filtration of the mapping class group and we study their relation with the usual Johnson homomorphisms.

\subsection{Double Johnson homomorphisms}
Consider the group $K=\pi_1(\Sigma,*)$ and its normal subgroups~$\bar{X}$ and~$\bar{Y}$ defined in \eqref{equ_X_Y}. Let $K_{*,*}=(K_{m,n})_{(m,n)\in\mathbb{N}^2}$ be the double lower central series of the triple $(K;\bar{X},\bar{Y})$ as in previous section. Let $\bK=\bigoplus_{(m,n)\in\modN^2_+}\bK_{m,n}$ be its associated {$\modN^2_+$-}graded Lie algebra. 

Let $A,B\leq H$ be the subgroups defined in~\eqref{equ_A_B}. By Lemma~\ref{r49} we have 
\begin{gather}
\bK_{1,0}\simeq A \quad \text{ and } \quad \bK_{0,1}\simeq B.
\end{gather}

Let $$\mathfrak{Lie}(A,B)=\bigoplus_{(m,n)\in\modN^2_+}\mathfrak{Lie}_{m,n}(A,B)$$ be the $\modN^2_+$-graded Lie algebra (over~$\mathbb{Z}$) {freely} generated by $A$ in degree $(1,0)$ and $B$ in degree $(0,1)$. Theorem~\ref{r51} yields
$$\bK = \bigoplus_{(m,n)\in\modN^2_+}\bK_{m,n}\simeq\bigoplus_{(m,n)\in\modN^2_+}\mathfrak{Lie}_{m,n}(A,B) = \mathfrak{Lie}(A,B).$$

Following Section~\ref{sec_dla}, for $(m,n)\in\modN^2_+$ denote by $\mathrm{Der}_{m,n}(\mathfrak{Lie}(A,B))$ the abelian group of derivations of $\mathfrak{Lie}(A,B)$ of degree $(m,n)$, i.e., the derivations $d$ of $\mathfrak{Lie}(A,B)$ such that $d(A)\subset \mathfrak{Lie}_{m+1,n}(A,B)$ and $d(B)\subset \mathfrak{Lie}_{m,n+1}(A,B)$. Let $d$ be a derivation of $\mathfrak{Lie}(A,B)$, we denote its restriction to $\mathfrak{Lie}_{i,j}(A,B)$ by~$d_{i,j}$.

For $(m,n)\in\modN^2_+$, set
\begin{gather}\label{e44}
\begin{split}
\mathrm{D}_{m,n}(\mathfrak{Lie}(A,B)) &:= \mathrm{Hom}(A, \mathfrak{Lie}_{m+1,n}(A,B))\oplus \mathrm{Hom}(B, \mathfrak{Lie}_{m,n+1}(A,B))\\
	&\simeq \mathrm{Hom}(\bK_{1,0}, \bK_{m+1,n})\oplus \mathrm{Hom}(\bK_{0,1}, \bK_{m,n+1}).
\end{split}
\end{gather}

The following is a classical result, see for instance  \cite[Lemma 0.7]{MR1231799}.
\begin{proposition}\label{r53} 
For every $(m,n)\in\modN^2_+$, there is a bijection
\begin{gather}
\Psi_{m,n}\zzzcolon \mathrm{Der}_{m,n}(\mathfrak{Lie}(A,B))\longrightarrow \mathrm{D}_{m,n}(\mathfrak{Lie}(A,B)),
\end{gather}
defined by $\Psi_{m,n}(d)= d_{1,0} + d_{0,1}$ for $d\in\mathrm{Der}_{m,n}(\mathfrak{Lie}(A,B))$.
\end{proposition}

The intersection form $\omega\zzzcolon H\otimes H\rightarrow \mathbb{Z}$ yields the identifications $A\simeq B^*$  and  $B\simeq A^*$ (mapping~$z$ to~$ \omega(z,\cdot)$). Hence 
\begin{gather}\label{e47}
\begin{split}
\mathrm{D}_{m,n}(\mathfrak{Lie}(A,B)) & \simeq (A^*\otimes\mathfrak{Lie}_{m+1,n}(A,B))\oplus(B^*\otimes\mathfrak{Lie}_{m, n+1}(A,B))\\
			& \simeq (B\otimes\mathfrak{Lie}_{m+1,n}(A,B))\oplus(A\otimes\mathfrak{Lie}_{m,n+1}(A,B))\\
			& = (A\otimes\mathfrak{Lie}_{m,n+1}(A,B))\oplus (B\otimes\mathfrak{Lie}_{m+1,n}(A,B)).
\end{split}
\end{gather}

Let $(\mathcal{M}_{m,n})_{(m,n)\in\modN^2_+}$ be the double Johnson filtration of the Goeritz group. By Proposition~\ref{r39}, for all $(m,n)\in\modN^2_+$ there is a group homomorphism
\begin{gather}
\tau_{m,n}\zzzcolon \mathcal{M}_{m,n}\longrightarrow \mathrm{Der}_{m,n}(\mathfrak{Lie}(A,B))\stackrel{\Psi_{m,n}}{\simeq} \mathrm{D}_{m,n}(\mathfrak{Lie}(A,B)).
\end{gather}
such that
\begin{gather}\label{ker_tau}
\begin{split}
\mathrm{ker}(\tau_{m,n})&=\{h\in\mathcal{M}_{m,n} \ | \ [h,K_{i,j}]\subset K_{>(m+i,n+j)} \ \text{for all } \ (i,j)\in\modN^2_+\}\\
 & =\{h\in\mathcal{M}_{m,n} \ | \ [h,K_{1,0}]\subset K_{>(m+1,n)}, \ [h,K_{0,1}]\subset K_{>(m,n+1)}\}.
\end{split}
\end{gather}
The second equality in \eqref{ker_tau} can be checked by induction. Notice that $\mathcal{M}_{m+1,n}\cdot \mathcal{M}_{m,n+1}\subset\mathrm{ker}(\tau_{m,n})$.

We refer to  $\tau_{m,n}$ (or to $\Psi_{m,n}\circ\tau_{m,n}$) as the \emph{$(m,n)$-double Johnson homomorphism of the Goeritz group}.

In terms of the generators $\{x_i,y_i\}$ of $K$, the symplectic basis $\{a_i,b_i\}$ of $H$ and identification \eqref{e47}, the $(m,n)$-double Johnson homomorphism $\tau_{m,n}(h)$ for $h\in\mathcal{M}_{m,n}$ is given by
\begin{gather}\label{e48}
\tau_{m,n}(h)  = \sum_{i=1}^g a_i\otimes \left([h,y_i]K_{{>(m,n+1)}}\right) - \sum_{i=1}^g b_i\otimes \left([h,x_i]K_{{>(m+1,n)}}\right).
\end{gather}

\subsection{Double Johnson homomorphisms and symplectic derivations}

We always consider homology and cohomology groups with integer coefficients, so from now on we omit $\mathbb{Z}$ in the notation.  Throughout this subsection we consider the symplectic basis $\{a_i,b_i\}$ of $H$ as in Section~\ref{sec_3_1}.

 From the long exact sequence in homology associated to the pair $(V,\partial V)$ we obtain the short exact sequence
\begin{gather}
0\longrightarrow H_2(V,\partial V)\xrightarrow{\ \delta_*\ } H_1(\Sigma)\xrightarrow{\ \iota_*\ } H_1(V)\longrightarrow 0.
\end{gather}
Hence $A\simeq H_2(V,\partial V)\simeq H^1(V)$, where the latter isomorphism is given by Poincar\'e duality. Therefore we have an intersection form
\begin{gather}\label{intV}
\omega_V\zzzcolon H_1(V)\times H_2(V,\partial V)\longrightarrow \mathbb{Z}.
\end{gather}
Similarly $B\simeq H_2(V',\partial V')\simeq H^1(V')$ and we have an intersection form
\begin{gather}
\omega_{V'}\zzzcolon H_1(V')\times H_2(V',\partial V')\longrightarrow \mathbb{Z}.
\end{gather}

The above intersection forms are related with the intersection form $\omega\zzzcolon H\times H\rightarrow \mathbb{Z}$ of~$\Sigma$ by the following commutative diagrams

\begin{multicols}{2}
\noindent
\begin{gather}
\xymatrix{  \ \ \ \ \ \ \ \ \ H_1(V)\times H_2(V,\partial V)\ar[r]^{\ \ \ \ \ \ \ \ \ \ \ \ \ \ \ \ \omega_{V}}\ar@<3.7ex>[d]_{\simeq}^{\delta_*} & \mathbb{Z}.  \\
\frac{H_1(\Sigma)}{A}\times A \ar@<3ex>[u]^{\iota_*}_{\simeq}\ar@/_1pc/[ru]_{\ \ \ \omega} & }
\end{gather}
\columnbreak
\begin{gather}
\xymatrix{  \ \ \ \ \ \ \ \ H_1(V')\times H_2(V',\partial V')\ar[r]^{\ \ \ \ \ \ \ \ \ \ \ \ \ \ \ \ \omega_{V'}}\ar@<3.7ex>[d]_{\simeq}^{\delta_*} & \mathbb{Z}.  \\
\frac{H_1(\Sigma)}{B}\times B \ar@<3ex>[u]^{\iota'_*}_{\simeq}\ar@/_1pc/[ru]_{\ \ \ \omega} & }
\end{gather}
\end{multicols}

The intersection form $\omega\zzzcolon H\otimes H\rightarrow\mathbb{Z}$ determines an element $\Omega\in\mathfrak{Lie}_{1,1}(A,B)$. In terms of the symplectic basis $\{a_i,b_i\}$ of $H$ this element is given by
\begin{gather}\label{e46}
\Omega=\sum_{i=1}^g[a_i,b_i].
\end{gather}

\begin{definition}
Let $d$ be a derivation of $\mathfrak{Lie}(A,B)$.
We say that $d$ is a \emph{symplectic} derivation if $d(\Omega)=0$.
\end{definition}

For $(m,n)\in\modN^2_+$, denote by  $\mathrm{Der}^{\omega}_{m,n}(\mathfrak{Lie}(A,B))$   the abelian group of symplectic derivations of $\mathfrak{Lie}(A,B)$ of degree $(m,n)$.

For $(m,n)\in\modN^2_+$, consider the Lie bracket map
\begin{gather}\label{e49}
\Xi_{m,n}\zzzcolon (A\otimes\mathfrak{Lie}_{m,n+1}(A,B))\oplus(B\otimes\mathfrak{Lie}_{m+1,n}(A,B))\longrightarrow \mathfrak{Lie}_{m+1,n+1}(A,B).
\end{gather}

Set $D_{m,n}(A,B):=\mathrm{ker}(\Xi_{m,n})$. 
\begin{proposition}\label{r54}
Let $d\in \mathrm{Der}_{m,n}(\mathfrak{Lie}(A,B))$. Then  $\Xi_{m,n}\Psi_{m,n}(d)=0$ if and only if $d(\Omega)=0$. That is
$$\mathrm{Der}^{\omega}_{m,n}(\mathfrak{Lie}(A,B))\simeq D_{m,n}(A,B).$$
Here $\Psi_{m,n}$ is the map from Proposition~\ref{r53}.
\end{proposition}

\begin{proof}
Using the symplectic basis $\{a_i,b_i\}$ of $H$ and  identification \eqref{e47} we obtain
\begin{gather}\label{JTH2equ10}
\Psi_{m,n}(d)=\sum_{i=1}^g b^*_i\otimes d(b_i) + \sum_{i=1}^g a^*_i\otimes d(a_i) = \sum_{i=1}^g a_i\otimes d(b_i)-\sum_{i=1}^g b_i\otimes d(a_i).
\end{gather}
Hence
\begin{gather}\label{JTH2equ11}
\Xi_{m,n}\Psi_{m,n}(d)=\sum_{i=1}^g [a_i, d(b_i)]-\sum_{i=1}^g [b_i, d(a_i)]=d\left(\sum_{i=1}^g [a_i, b_i]\right)=d(\Omega).
\end{gather}
Therefore $\Xi_{m,n}\Psi_{m,n}(d)=0$ if and only if $d(\Omega)=0$.
\end{proof}

\begin{proposition}\label{r55}
Let $(m,n)\in\modN^2_+$. For $h\in \mathcal{M}_{m,n}$ we have $\Xi_{m,n}\tau_{m,n}(h)=0$, that is $$\tau_{m,n}(h)\in D_{m,n}(A,B)\simeq \mathrm{Der}^{\omega}_{m,n}(\mathfrak{Lie}(A,B)).$$
In other words, $\tau_{m,n}(h)$ is a symplectic derivation of $\mathfrak{Lie}(A,B)$.
\end{proposition}

\begin{proof}
The proof follows the lines of the proof of \cite[Corollary 3.2]{MR1224104}. Consider the free basis $\{x_i,y_i\}$ of $K$ as in Section~\ref{sec_3_1}.

Let $h\in \mathcal{M}_{m,n}$. The induced map $h_{\#}$ in homotopy fixes the inverse of the homotopy class $[\partial\Sigma]$ of a loop parallel to the boundary  $\Sigma$. That is,
\begin{equation}\label{JTH2equ15}
h_{\#}\left(\prod_{i=1}^g [x_i^{-1}, y_i^{-1}]\right) = \prod_{i=1}^g [x_i^{-1}, y_i^{-1}].
\end{equation} 
For $1\leq i\leq g$ we have
$$h_{\#}(x_i^{-1})x_i=\gamma_i\in K_{m+1,n}\ \ \ \ \ \text{ and }\ \ \ \ \ h_{\#}(y_i^{-1})y_i=\delta_i\in K_{m,n+1}.$$
Hence
\begin{equation*}
\begin{split}
[h_{\#}(x^{-1}_i),h_{\#}(y^{-1}_i)] & = [\gamma_ix^{-1}_i, \delta_iy^{-1}_i]\\
 & = \left(\gamma_i[x^{-1}_i, \delta_i]\gamma^{-1}_i\right)\left(\gamma_i\delta_i[x^{-1}_i, y^{-1}_i]\delta^{-1}_i\gamma^{-1}_i\right)[\gamma_i,\delta_i]\left(\delta_i[\gamma_i,y^{-1}_i]\delta^{-1}_i\right).
\end{split}
\end{equation*}

\medskip

\noindent It follows from Equality (\ref{JTH2equ15}) that
\begin{align}\label{JTH2equ161}
\prod_{i=1}^g [x^{-1}_i, y^{-1}_i] 
= \prod_{i=1}^g \left(\gamma_i[x^{-1}_i, \delta_i]\gamma^{-1}_i\right)\left(\gamma_i\delta_i[x^{-1}_i, y^{-1}_i]\delta^{-1}_i\gamma^{-1}_i\right)[\gamma_i,\delta_i]\left(\delta_i[\gamma_i,y^{-1}_i]\delta^{-1}_i\right).
\end{align}

\medskip

\noindent Now $[x^{-1}_i, \delta_i]\in K_{m+1,n+1}$, $[\gamma_i,\delta_i]\in K_{2m+1,2n+1}\subset K_{{>(m+1,n+1)}}$ and $[\gamma_i,y^{-1}_i]\in K_{m+1,n+1}$. Therefore, by considering Equation (\ref{JTH2equ161}) modulo $K_{{>(m+1,n+1)}}$ and using Lemma~\ref{r22} we obtain
\begin{align*}
\prod_{i=1}^g [x^{-1}_i, y^{-1}_i]& \equiv \prod_{i=1}^g [x^{-1}_i, \delta_i][x^{-1}_i, y^{-1}_i][\gamma_i,y^{-1}_i]\\
 & \equiv \left(\prod_{i=1}^g [x^{-1}_i, y^{-1}_i]\right)\left(\prod_{i=1}^g  [x_i, \delta^{-1}_i][\gamma^{-1}_i,y_i]\right).
\end{align*}
Thus
\begin{equation}\label{e50}
\prod_{i=1}^g  [x_i, \delta^{-1}_i][\gamma^{-1}_i,y_i]\in K_{{>(m+1,n+1)}}.
\end{equation}

\noindent From \eqref{e50}, identification \eqref{e47} and \eqref{e48} we have
\begin{equation}\label{e51}
\begin{split}
0 & = \Xi_{m,n}\left(\sum_{i=1}^{g}a_i\otimes (\delta^{-1}_iK_{{>(m,n+1)}}) - \sum^{g}_{i=1}b_i\otimes (\gamma^{-1}_iK_{{>(m+1,n)}})\right)\\
 & = \sum_{i=1}^{g}[a_i,[h,y_i]K_{{>(m,n+1)}}]-\sum_{i=1}^g [b_i,[h,x_i]K_{{>(m+1,n)}}]\\
 & = \Xi_{m,n}\tau_{m,n}(h).
\end{split}
\end{equation}
In the second equality of \eqref{e51} we used
\begin{gather}
\begin{split}
\delta^{-1}_iK_{{>(m,n+1)}} & = y_i\delta^{-1}_iy^{-1}_iK_{{>(m,n+1)}}\\
& = h_{\#}(y_i)y^{-1}_iK_{{>(m,n+1)}}\\
& = [h,y_i]K_{{>(m,n+1)}},
\end{split}
\end{gather}
and a similar equivalence for $\gamma_i$.
\end{proof}

\subsection{Relation with the usual Johnson homomorphisms}

We have seen in Section~\ref{sec_3_4} that $n$-th Johnson homomorphism take values in the kernel $D_n(H)$ of the Lie bracket  $\left[\ ,\ \right]\zzzcolon H\otimes\mathfrak{Lie}_{n+1}(H)\rightarrow\mathfrak{Lie}_{n+2}(H)$~\cite[Corollary 3.2]{MR1224104}. The group $D_n(H)$ is isomorphic to the group $\mathrm{Der}^{\omega}_n(\mathfrak{Lie}(H))$ of symplectic (vanishing in $\Omega$) degree $n$ derivations of $\mathfrak{Lie}(H)$.

For $m,n\geq 2$ we set 
\begin{gather}
D_{m,-1}(A,B)=D_{m-1}(A)=\mathrm{ker}\big(A\otimes\mathfrak{Lie}_{m}(A)\xrightarrow{\ [\ , \ ]\ }\mathfrak{Lie}_{m+1}(A)\big),\\
D_{-1,n}(A,B)=D_{n-1}(B)=\mathrm{ker}\big(B\otimes\mathfrak{Lie}_{n}(B)\xrightarrow{\ [\ , \ ]\ }\mathfrak{Lie}_{n+1}(B)\big),
\end{gather}
where $\mathfrak{Lie}(A)=\bigoplus_{j\geq 1}\mathfrak{Lie}_j(A)$ (resp. $\mathfrak{Lie}(B)=\bigoplus_{j\geq 1}\mathfrak{Lie}_j(B)$) is the  $\modN_+$-graded Lie algebra freely generated by $A$ (resp. $B$) in degree~$1$. For instance $D_{2,-1}(A,B)\simeq\ext^3 A$ and $D_{-1,2}(A,B)\simeq\ext^3 B$. Notice that $\mathfrak{Lie}(A)$ and $\mathfrak{Lie}(B)$ can be seen as subalgebras of $\mathfrak{Lie}(H)$. Similarly  $D_n(A)$ and $D_n(B)$ can be seen as subgroups of $D_n(H)$.

\begin{lemma}\label{r57} For $m,n\geq -1$ with $m+n\geq 1$, there is a well-defined (inclusion) homomorphism $$j\zzzcolon D_{m,n}(A,B)\longrightarrow D_{m+n}(H).$$
\end{lemma}
\begin{proof}
The cases $m=-1$ and $n=-1$ follows by definition. If $(m,n)\in\modN^2_+$ we have a  map
$$\mathfrak{Lie}_{m,n}(A,B)\simeq\frac{K_{m,n}}{K_{{>(m,n)}}}\longrightarrow \frac{\Gamma_{m+n}K}{\Gamma_{m+n+1}K}\simeq \mathfrak{Lie}_{m+n}(H),$$
which sends $zK_{{>(m,n)}}$ to $z\Gamma_{m+n+1}K$ for all $z\in K_{m,n}$. This map is compatible with the Lie bracket on both sides, in particular, the following diagram commutes.
\begin{gather*}
\xymatrixcolsep{4pc}\xymatrix{\left(A\otimes\mathfrak{Lie}_{m,n+1}(A,B)\right)\oplus\left(B\otimes\mathfrak{Lie}_{m+1,n}(A,B)\right)\ar[r]^-{\Xi_{m,n}\ }\ar[d]_{} & \mathfrak{Lie}_{m+1,n+1}(A,B) \ar[d]^{} \\
						H\otimes\mathfrak{Lie}_{m+n+1}(H)\ar[r]^-{\left[\ ,\ \right]} & \mathfrak{Lie}_{m+n+2}(H).}
\end{gather*}
Therefore we obtain a  homomorphism $j\zzzcolon D_{m,n}(A,B)\rightarrow D_{m+n}(H).$
\end{proof}

Let $(\mathcal{M}_{m,n})_{m,n\geq -1}$ be the double Johnson filtration of the mapping class group. We extend the double Johnson homomorphisms for the Goeritz group to the whole family $(\mathcal{M}_{m,n})_{m,n\geq -1}$ by using the usual Johnson homomorphisms as follows. For $m,n\geq 2$ we set
\begin{gather}
\tau_{m,-1} ={\tau_{m-1}}{|}_{\mathcal{M}_{m,-1}} \quad \quad \text{and} \quad \quad \tau_{-1,n} ={\tau_{n-1}}{|}_{\mathcal{M}_{-1,n}}.
\end{gather}
We can check that $\tau_{m,-1}$ (resp. $\tau_{-1,n}$) takes values in $D_{m,-1}(A,B)$ (resp. $D_{-1,n}(A,B)$). Moreover, we have
$\mathcal{M}_{m+1,-1}\cdot\mathcal{M}_{m,0}\subset \mathrm{ker}(\tau_{m,-1})$ and $\mathcal{M}_{-1,n+1}\cdot\mathcal{M}_{0,n}\subset \mathrm{ker}(\tau_{-1,n})$.

\begin{proposition}\label{r56}
 For $m,n\geq -1$ with $m+n\geq 1$, the following diagram commutes. 
\begin{equation*}%
\xymatrix{\mathcal{M}_{m,n}\ar[r]^{\subset}\ar[d]_{\tau_{m,n}} & J_{m+n}\mathcal{M} \ar[d]^{\tau_{m+n}} \\
						D_{m,n}(A,B)\ar[r]^{j} & D_{m+n}(H).}
\end{equation*}
In other words, ${\tau_{m+n}}{|}_{{\mathcal{M}_{m,n}}} = j\tau_{m,n}$.
\end{proposition}
\begin{proof} The cases $m=-1$ and $n=-1$ follows by definition. Let $(m,n)\in\modN^2_+$ and $h\in\mathcal{M}_{m,n}$. If we apply the map $j$ to $\tau_{m,n}(h)$ expressed as in \eqref{e48} we obtain  $\tau_{m+n}(h)$ expressed as in \eqref{eq_tau_b}.
\end{proof}

\begin{corollary} The homomorphisms $\tau_{i,j}\zzzcolon \mathcal{M}_{i,j}\rightarrow D_{i,j}(A,B)$ for $(i,j)=(2,-1)$, $(1,0)$, $(0,1)$, $(-1,2)$
are surjective.
\end{corollary}
\begin{proof}
This follows {from} Propositions~\ref{r56} and \ref{r45}.
\end{proof}

\begin{corollary} Let $h\in\mathcal{I}$, then there exist $h_1\in\mathcal{M}_{2,-1}$, $h_2\in\mathcal{M}_{1,0}$, $h_3\in\mathcal{M}_{0,1}$ and $h_4\in\mathcal{M}_{-1,2}$ such that
$$\tau_1(h)=\tau_{2,-1}(h_1) + \tau_{1,0}(h_2) + \tau_{0,1}(h_3) + \tau_{-1,2}(h_4).$$
\end{corollary}

\begin{proof}
This follows {from} Proposition~\ref{r56} and Theorem~\ref{r46}. Notice that the elements $h_1$, $h_2$, $h_3$ and $h_4$ can be found explicitly from the action of $h$ in homotopy as we have seen in Proposition~\ref{r45}.
\end{proof}

{
\begin{remark} In this remark we work over $\Q$ and understand $A$, $B$, etc. as $A\otimes\Q, B\otimes\Q$, etc.
The vector spaces $D_k(H)=D_k(A\oplus B)$ have a natural structure of $\mathrm{GL}(H)$-module, and their irreducible decompositions are known for some small values of $k$.

Once such an irreducible decomposition is given for $D_k(H)$, one can decompose the $\mathrm{GL}(A)\times \mathrm{GL}(B)$-module $D_{m,n}(A\oplus B)$ with $m+n=k$, by using the behavior of the Schur functor~$S^\lambda$ with respect to direct sum:
\begin{gather*}
  S^{\lambda}(A\oplus B) = \bigoplus_{\mu,\nu} \mathrm{LR}^\lambda_{\mu,\nu}(S^\mu(A)\otimes S^\nu(B)),
\end{gather*}
where $\lambda,\mu,\nu$ are partitions, and $\mathrm{LR}^\lambda_{\mu,\nu}$ denotes the Littlewood-Richardson coefficient.
For instance,
\begin{gather*}
\begin{split}
D_1(A\oplus B) &= S^{[111]}(A\oplus B)\\
 &= \big(S^{[111]}(A)\otimes S^{[0]}(B)\big)\oplus \big(S^{[11]}(A)\otimes S^{[1]}(B)\big) \oplus \big(S^{[1]}(A)\otimes S^{[11]}(B)\big) \oplus \big(S^{[0]}(A)\otimes S^{[1]}(B)\big)\\
 &=\ext^3 A \oplus (\ext^2 A\otimes B)\oplus (A\otimes \ext^2B)\oplus \ext^3 B \\
 &= D_{2,-1}(A,B)\oplus D_{1,0}(A,B) \oplus D_{0,1}(A,B) \oplus D_{-1,2}(A,B), 
\end{split}
\end{gather*}
which is the decomposition given previously in \eqref{eq:1} (in the case of integer coefficients). Now, for $k=2$, we have
\begin{gather*}
\begin{split}
D_2(A\oplus B) &= S^{[22]}(A\oplus B)\\
 &= \big(S^{[22]}(A)\otimes S^{[0]}(B)\big)\oplus \big(S^{[21]}(A)\otimes S^{[1]}(B)\big) \oplus \big(\big(S^{[2]}(A)\otimes S^{[2]}(B)\big) \oplus \big(S^{[11]}(A)\otimes S^{[11]}(B)\big)\big)\\ & \quad \quad \oplus  \big(S^{[1]}(A)\otimes S^{[21]}(B)\big)\oplus  \big(S^{[0]}(A)\otimes S^{[22]}(B)\big)\\ 
 &= D_{3,-1}(A,B)\oplus D_{2,0}(A,B) \oplus D_{1,1}(A,B) \oplus D_{0,2}(A,B)\oplus D_{-1,3}(A,B).
\end{split}
\end{gather*}
\end{remark}
}

\begin{remark} For each $k\geq 1$, the space $D_k(H)\otimes \mathbb{Q}= D_k(A\oplus B)\otimes\mathbb{Q}$ can be identified with the space generated by \emph{tree-like Jacobi diagrams} (i.e. planar unitrivalent trees subject to some local relations) with~$k$ trivalent vertices and whose univalent vertices are labelled with elements of $A\oplus B$, see for instance~\cite{MR1783857}. This implies a diagrammatic description of the space $D_{m,n}(A,B)\otimes \mathbb{Q}$: it can be identified with the space generated by tree-like Jacobi diagrams with exactly $i+1$ univalent vertices colored by $A$ and $j+1$ univalent vertices colored by $B$. 
\end{remark}

\subsection{Relation with the alternative Johnson homomorphisms}
In \cite{HM} Massuyeau and the first author considered the $\modN$-filtration $(K^{\mathfrak{a}}_m)_{m\in\modN}$ of $K=\pi_1(\Sigma,*)$ defined recursively as follows: $K^{\mathfrak{a}}_0 = K^{\mathfrak{a}}_1 = K$, $K^{\mathfrak{a}}_2 = [K,K]\bX$ and 
$$K^{\mathfrak{a}}_m = [K^{\mathfrak{a}}_1, K^{\mathfrak{a}}_{m-1}]\;[K^{\mathfrak{a}}_2, K^{\mathfrak{a}}_{m-2}]$$
for $m\geq 3$. The Lagrangian mapping class group $\mathcal{L}$ acts on this $\modN$-filtration. The associated Johnson filtration of $\mathcal{L}$ is given by
$$J^{\mathfrak{a}}_m\mathcal{L} = \{h\in\mathcal{L}\ | \ [h, K^{\mathfrak{a}}_1]\subset K^{\mathfrak{a}}_{m+1} \quad \text{and} \quad [h, K^{\mathfrak{a}}_2]\subset K^{\mathfrak{a}}_{m+2}\}.$$
The $\modN$-filtration $(J^{\mathfrak{a}}_m\mathcal{L})_{m\in\modN}$ is called the \emph{alternative Johnson filtration}. We refer to~\cite{vera} for more details about this filtration and its associated \emph{alternative} Johnson homomorphisms. 

It follows easily that for $(m,n)\in\modN^2$ we have $K_{m,n}\subset K^{\mathfrak{a}}_{2m+n}$. Therefore, using \eqref{e16} and \eqref{e17}, we obtain 
$$\mathcal{M}_{m,n}\subset J^{\mathfrak{a}}_{2m+n}\mathcal{L}$$
for all $m,n\geq -1$ with $2m+n\geq 1$. The $\modN_+$-graded Lie algebra associated to $(K^{\mathfrak{a}}_m)_{m\in\modN}$ is isomorphic to the $\modN_+$-graded Lie algebra $\mathfrak{Lie}^{\mathfrak{a}}(H_1(V),A)=\bigoplus_{m\geq 1}\mathfrak{Lie}^{\mathfrak{a}}_m(H_1(V),A)$ freely generated by $H_1(V)$ in degree~$1$ and~$A$ in degree~$2$. 

A derivation of $\mathfrak{Lie}^{\mathfrak{a}}(H_1(V),A)$ is said to be \emph{symplectic} if it vanishes on the element $\Omega_V\in\mathfrak{Lie}^{\mathfrak{a}}_3(H_1(V), A)$ determined by the intersection form $\omega_V$ given in \eqref{intV}. 
For $m\geq 1$, consider the Lie bracket
\begin{gather}
\Xi^{\mathfrak{a}}_m: \left(A\otimes\mathfrak{Lie}^{\mathfrak{a}}_{m+1}(H_1(V),A)\right)\oplus\left(H_1(V)\otimes\mathfrak{Lie}^{\mathfrak{a}}_{m+2}(H_1(V),A)\right)\longrightarrow \mathfrak{Lie}^{\mathfrak{a}}_{m+3}(H_1(V),A).
\end{gather}
Set $D_m^{\mathfrak{a}}(H_1(V), A):=\mathrm{ker}(\Xi_m^{\mathfrak{a}})$. This group can be identified with the group of symplectic derivations of~$\mathfrak{Lie}^{\mathfrak{a}}(H_1(V), A)$ of degree $m$.

For $m\geq 1$ we have a group homomorphism
$$\tau^{\mathfrak{a}}_m: J^{\mathfrak{a}}_m\mathcal{L}\longrightarrow D^{\mathfrak{a}}_m(H_1(V),A),$$
called the \emph{$m$-th alternative Johnson homomorphism}, which in terms of the generators $\{x_i,y_i\}$ of $K$ and the symplectic basis $\{a_i,b_i\}$ of $H$  is given by
\begin{gather}\label{e53}
\tau^{\mathfrak{a}}_{m}(h)  = \sum_{i=1}^g a_i\otimes \left([h,y_i]K^{\mathfrak{a}}_{m+2}\right) - \sum_{i=1}^g b_i\otimes \left([h,x_i]K^{\mathfrak{a}}_{m+3}\right)
\end{gather}
for every $h\in J^{\mathfrak{a}}_m\mathcal{L}$.

For $m,n\geq -1$ with $2m+n\geq 1$, there is a  group homomorphism 
\begin{gather}
j^{\mathfrak{a}}\zzzcolon D_{m,n}(A,B)\longrightarrow D_{2m+n}^{\mathfrak{a}}(H_1(V),A)
\end{gather}
defined as follows.

For $m=-1$ and $n\geq 3$, the homomorphism $j^{\mathfrak{a}}:D_{-1,n}(A,B)\rightarrow D^{\mathfrak{a}}_{n-2}(H_1(V),A)$ is induced by the  map 
$$B\otimes \mathfrak{Lie}_n(B)\longrightarrow \big(A\otimes \mathfrak{Lie}^{a}_{n-1}(H_1(V),A)\big)\oplus\big(H_1(V)\otimes \mathfrak{Lie}^{a}_{n}(H_1(V),A)\big)$$
arising from the identification $B\simeq H_1(V)$.

For $n=-1$ and $m\geq 1$, the homomorphism $j^{\mathfrak{a}}:D_{m,-1}(A,B)\rightarrow D^{\mathfrak{a}}_{2m-1}(H_1(V),A)$ is induced by the canonical map 
$$A\otimes \mathfrak{Lie}_m(A)\longrightarrow \big(A\otimes \mathfrak{Lie}^{a}_{2m}(H_1(V),A)\big)\oplus\big(H_1(V)\otimes \mathfrak{Lie}^{a}_{2m+1}(H_1(V),A)\big).$$

Finally, for $(m,n)\in\modN^2_+$ the homomorphism $j^{\mathfrak{a}}\zzzcolon D_{m,n}(A,B)\longrightarrow D_{2m+n}^{\mathfrak{a}}(H_1(V),A)$ is induced by the identification $B\simeq H_1(V)$ and the map
$$\mathfrak{Lie}_{m,n}(A,B)\simeq\frac{K_{m,n}}{K_{{>(m,n)}}}\longrightarrow \frac{K^{\mathfrak{a}}_{2m+n}}{K^{\mathfrak{a}}_{2m+n+1}}\simeq \mathfrak{Lie}^{\mathfrak{a}}_{2m+n}(H_1(V), A),$$
which sends $zK_{{>(m,n)}}$ to $zK^{\mathfrak{a}}_{2m+n+1}$ for all $z\in K_{m,n}$.

We can now prove the analogue of Proposition~\ref{r56} in the case of the alternative Johnson homomorphisms.

\begin{proposition}\label{r58}
 For $m,n\geq -1$ with $2m+n\geq 1$, the following diagram commutes. 
\begin{equation*}%
\xymatrix{\mathcal{M}_{m,n}\ar[r]^-{\subset}\ar[d]_{\tau_{m,n}} & J^{\mathfrak{a}}_{2m+n}\mathcal{L} \ar[d]^{\tau^{\mathfrak{a}}_{2m+n}} \\
						D_{m,n}(A,B)\ar[r]^-{j^{\mathfrak{a}}} & D^{\mathfrak{a}}_{2m+n}(H_1(V),A).}
\end{equation*}
In other words, ${\tau^{\mathfrak{a}}_{2m+n}}{|}_{{\mathcal{M}_{m,n}}} = j^{\mathfrak{a}}\tau_{m,n}$.
\end{proposition}
\begin{proof} Let $h\in\mathcal{M}_{m,n}$. If we apply the map $j^{\mathfrak{a}}$ to $\tau_{m,n}(h)$ (expressed as in \eqref{eq_tau_b} in the cases $m=-1$ or $n=-1$; or as in \eqref{e48} in the remaining cases) we obtain  $\tau^{\mathfrak{a}}_{2m+n}(h)$ expressed as in \eqref{e53}.
\end{proof}

\bibliographystyle{plain} 
\bibliography{DF_MCG}

\end{document}